\documentclass{article}
\pdfoutput=1

\makeatletter

\@addtoreset{equation}{section}
\makeatother

\input xy
\xyoption{all}

\usepackage[cp1252]{inputenc}
\usepackage[OT1]{fontenc}
\usepackage[english]{babel}

\usepackage{amssymb, amsmath, amsthm,
amscd,
graphicx,
cite,
}

\def\R{{\mathbb R}}
\def\N{{\mathbb N}}
\def\Z{{\mathbb Z}}
\def\C{{\mathbb C}}

\def\CA{\mathcal{A}}
\def\CO{\mathcal{O}}

\def\ds{\displaystyle}

\def\p{\psi}
\def\a{\alpha}
\def\b{\beta}
\def\d{\delta}
\def\D{\Delta}

\def\g{\gamma}
\def\eps{\varepsilon}
\def\f{\varphi}
\def\S{\Sigma}
\def\s{\sigma}
\def\t{\theta}

\def\lam{\lambda}
\def\Lam{\Lambda}

\def\dth{\dot \theta}
\def\dc{\dot c}
\def\dlam{\dot \lam}
\def\dx{\dot x}
\def\dy{\dot y}

\def\dq{\dot q}
\def\dh{\dot h}

\def\tcut{t_{\operatorname{cut}}}
\def\tt{\mathbf t}

\newcommand{\am}{\operatorname{am}\nolimits}
\newcommand{\sn}{\operatorname{sn}\nolimits}
\newcommand{\cn}{\operatorname{cn}\nolimits}
\newcommand{\dn}{\operatorname{dn}\nolimits}
\newcommand{\sgn}{\operatorname{sgn}\nolimits}
\newcommand{\const}{\operatorname{const}\nolimits}
\newcommand{\Id}{\operatorname{Id}\nolimits}
\newcommand{\spann}{\operatorname{span}\nolimits}

\newcommand{\E}{\operatorname{E}\nolimits}
\newcommand{\e}{\operatorname{e}\nolimits}

\newcommand{\MAX}{\operatorname{MAX}\nolimits}

\renewcommand{\Vec}{\operatorname{Vec}\nolimits}
\newcommand{\Exp}{\operatorname{Exp}\nolimits}

\newcommand{\cl}{\operatorname{cl}\nolimits}
\newcommand{\Lie}{\operatorname{Lie}}

\newcommand{\intt}{\operatorname{int}}

\def\vh{\vec h}
\def\vH{\vec H}

\newcommand{\der}[2]{\frac{d \, #1}{d\, #2} }
\newcommand{\pder}[2]{\frac{\partial \, #1}{\partial \, #2} }
\newcommand{\be}[1]{\begin{equation}\label{#1}}
\newcommand{\ee}{\end{equation}}
\newcommand{\ddef}[1]{{\em #1\/}}
\newcommand{\map}[3]{#1 \, : \, #2 \to #3}
\newcommand{\mapto}[3]{#1 \, : \, #2 \mapsto #3}
\newcommand{\vect}[1]{\left( \begin{array}{c} #1 \end{array} \right)}
\newcommand{\eq}[1]{$(\protect\ref{#1})$}
\newcommand{\restr}[2]{\left. #1 \right|_{#2}}

\newcommand{\orr}[2]{\left[\begin{array}{l}{#1}\\{#2}\end{array}\right.}
\newcommand{\orrr}[3]{\left[\begin{array}{l}{#1}\\{#2}\\{#3}\end{array}\right.}

\def\tlam{\widetilde{\lambda}}
\def\tM{\widetilde{M}}
\def\tP{\widetilde{P}}

\def\tq{\widetilde{q}}
\def\tu{\widetilde{u}}

\def\hC{\widehat{C}}
\def\hN{\widehat{N}}
\def\tC{\widetilde{C}}
\def\tk{\widetilde{k}}
\def\hk{\widehat{k}}

\def\CO{{\mathcal O}}

\def\ts{\,{\sn \tau}\,}
\def\tc{\,{\cn \tau}\,}
\def\td{\,{\dn \tau}\,}
\def\ss{\,{\sn  p}\,}
\def\cc{\,{\cn  p}\,}
\def\dd{\,{\dn  p}\,}
\def\tsp{\,{\sn^2 \tau}\,}
\def\tcp{\,{\cn^2 \tau}\,}
\def\tdp{\,{\dn^2 \tau}\,}
\def\ssp{\,{\sn^2  p}\,}
\def\ccp{\,{\cn^2  p}\,}
\def\ddp{\,{\dn^2  p}\,}

\def\sst{\,{\sn^3  p}\,}

\def\tsf{\,{\sn^4 \tau}\,}

\def\ssf{\,{\sn^4  p}\,}
\def\ccf{\,{\cn^4  p}\,}

\def\sqq{\frac{1}{\sqrt 2}}

\def\then{\quad\Rightarrow\quad}
\def\iff{\quad\Leftrightarrow\quad}

\newcommand{\twofiglabel}[6]
{
\begin{figure}[htbp]
\includegraphics[width=0.47\textwidth]{#1}
\hfill
\includegraphics[width=0.47\textwidth]{#4}
\\
\parbox[t]{0.45\textwidth}{\caption{#2}\label{#3}}
\hfill
\parbox[t]{0.45\textwidth}{\caption{#5}\label{#6}}
\end{figure}
}

\newcommand{\twofiglabelsize}[8]
{
\begin{figure}[htbp]
\includegraphics[width=#7\textwidth]{#1}
\hfill
\includegraphics[width=#8\textwidth]{#4}
\\
\parbox[t]{#7\textwidth}{\caption{#2}\label{#3}}
\hfill
\parbox[t]{#8\textwidth}{\caption{#5}\label{#6}}
\end{figure}
}

\newcommand{\onefiglabel}[3]
{
\begin{figure}[htbp]
\begin{center}
\includegraphics[width=0.47\textwidth]{#1}
\\
\parbox[t]{0.45\textwidth}{\caption{#2}\label{#3}}
\end{center}
\end{figure}
}

\newcommand{\onefiglabelsize}[4]
{
\begin{figure}[htbp]
\begin{center}
\includegraphics[width=#4\textwidth]{#1}
\\
\parbox[t]{#4\textwidth}{\caption{#2}\label{#3}}
\end{center}
\end{figure}
}

\newtheorem{theorem}{Theorem}[section]
\newtheorem{lemma}{Lemma}[section]
\newtheorem{corollary}{Corollary}[section]
\newtheorem{proposition}{Proposition}[section]

\theoremstyle{remark}
\newtheorem*{remark}{Remark}

\title{Maxwell strata in Euler's elastic problem%
\footnote{Work supported by the Russian Foundation for Basic Research, project
No.~05-01-00703-a.}
}

\author{
Yu. L. Sachkov \\
Program Systems Institute \\
Russian Academy of Sciences \\
Pereslavl-Zalessky 152020 Russia \\
E-mail: sachkov@sys.botik.ru}

\date{}

\begin{document}

\maketitle

\begin{abstract}
The classical Euler's problem on stationary configurations of elastic rod with fixed endpoints and tangents at the endpoints is considered as a left-invariant optimal control problem on the group of motions of a two-dimensional plane $\E(2)$. 

The attainable set is described, existence and boundedness of optimal controls are proved. Extremals are parametrized by Jacobi's elliptic functions of natural coordinates induced by the flow of the mathematical pendulum on fibers of the cotangent bundle of $\E(2)$. 

The group of discrete symmetries of Euler's problem generated by reflections in the phase space of the pendulum is studied. The corresponding Maxwell points are completely described via the study of fixed points of this group. As a consequence, an upper bound on cut points in Euler's problem is obtained. 


\end{abstract}


\newpage
\tableofcontents

\newpage
\section{Introduction}
\label{sec:intro}

In 1744 Leonhard Euler considered the following problem on stationary configurations of elastic rod. Given a rod in the plane with fixed endpoints and tangents at the endpoints, one should determine possible profiles of the rod under the given boundary conditions. Euler obtained ODEs for stationary configurations of the elastic rod and described their possible qualitative types. These configurations are called \ddef{Euler elasticae}.

An Euler elastica is a critical point of the functional of elastic energy on the space of smooth planar curves that satisfy the boundary conditions specified. In this paper we address the issue of \ddef{optimality} of an elastica: whether a critical point is a minimum of the energy functional? That is, which elasticae provide the minimum of the energy functional among all curves satisfying the boundary conditions (the \ddef{global optimality}), or the minimum compared with sufficiently close curves satisfying the boundary conditions (the \ddef{local optimality}). These questions remained open despite their obvious importance.

For the elasticity theory, the problem of local optimality is essential since it corresponds to \ddef{stability} of Euler elasticae under small perturbations that preserve the boundary conditions. In the calculus of variations and optimal control, the point where an extremal trajectory loses its local optimality is called a \ddef{conjugate point}. We will  give an exact description of conjugate points in the problem on Euler elasticae, which were previously  known only  numerically.

From the mathematical point of view, the problem of global optimality is fundamental. We will study \ddef{cut points}
 in Euler's elastic problem --- the points where elasticae lose their global optimality.

This is the first of two planned works on Euler's elastic problem. The aim of this work is to give a complete description of \ddef{Maxwell points}, i.e., points where distinct extremal trajectories with the same value of the cost functional meet one another. Such points provide an upper bound on cut points: an extremal trajectory cannot be globally optimal after a Maxwell point. In the second work~\cite{elastica_conj} we prove that conjugate points in Euler's elastic problem are bounded by Maxwell points. Moreover, we pursue the study of the global optimal problem: we describe the global diffeomorphic properties of the \ddef{exponential mapping}.

This paper is organized as follows. In Sec.~\ref{sec:history} we review the history of the problem on elasticae.
In Sec.~\ref{sec:problem} we state Euler's problem as a left-invariant optimal control problem on the group of motions of a two-dimensional plane $\E(2)$ and discuss the continuous symmetries of the problem. 
In Sec.~\ref{sec:att_set} we describe the attainable set of the control system in question. 
In Sec.~\ref{sec:exist} we prove existence and boundedness of optimal controls in Euler's problem. 
In Sec.~\ref{sec:extremals} we apply Pontryagin Maximum Principle to the problem, describe abnormal extremals, and derive the Hamiltonian system for normal extremals. 

Due to the left-invariant property of the problem, the normal Hamiltonian system of PMP becomes triangular after an appropriate choice of parametrization of fibers of the cotangent bundle of $\E(2)$: the vertical subsystem is independent on the horizontal coordinates. Moreover, this vertical subsystem is essentially the equation of the \ddef{mathematical pendulum}. 
For the detailed subsequent analysis of the extremals, it is crucial to choose convenient coordinates. 
In Sec.~\ref{sec:ell_coords} we construct such natural coordinates in the fiber of the cotangent bundle over the initial point. First we consider the ``angle-action'' coordinates in the phase cylinder of the standard pendulum, and then continue them to the whole fiber via continuous symmetries of the problem.
One of the coordinates is the time of motion of the pendulum, and the other two are integrals of motion of the pendulum.
In Sec.~\ref{sec:integration} we apply the \ddef{elliptic coordinates} thus constructed for integration of the normal Hamiltonian system. In particular, we recover the different classes of elasticae discovered by Leonhard Euler. 

The flow of the pendulum plays the key role not only in the parametrization of extremal trajectories, but also in the study of their optimality.
In Sec.~\ref{sec:discr_sym} we describe the \ddef{discrete symmetries} of Euler's problem generated by reflections in the phase cylinder of the standard pendulum. Further, we study the action of the group of reflections in the preimage and image of the exponential mapping of the problem.

In Sec.~\ref{sec:Maxwell} we consider Maxwell points of Euler's problem. The Maxwell strata corresponding to reflections are described by certain equations in elliptic functions. 
In Sec.~\ref{sec:MAX_complete} we study solvability of these equations, give sharp estimates of their roots, and describe their mutual disposition via the analysis of the elliptic functions involved. 

A complete description of the Maxwell strata obtained is important both for global and for local optimality of extremal trajectories. 
In Sec.~\ref{sec:up_bound} we derive an upper bound on the cut time in Euler's problem due to the fact that such a trajectory cannot be globally optimal after a Maxwell point.
In our subsequent work~\cite{elastica_conj} we will show that conjugate points in Euler's problem are bounded by Maxwell points and give a complete solution to the problem of local optimality of extremal trajectories.

In Sec.~\ref{sec:append} we collect the definitions and properties of Jacobi's elliptic functions essential for this work.

We used the system ``Mathematica''~\cite{math} to carry out complicated calculations and to produce illustrations in this paper.

\bigskip 
\noindent
{\em Acknowledgment.}
The author wishes to thank Professor A.A.Agrachev for bringing the pearl of Euler's problem
to author's attention, and for numerous fruitful discussions of this problem.

\section{History of Euler's elastic problem}
\label{sec:history}

In addition to the original works of the scholars who contributed to the theory, 
in this section
we follow also the sources on history of the subject by C.Truesdell~\cite{truesdell}, A.E.H.Love\cite{love}, and S.Timoshenko~\cite{timoshenko}.

In 1691 James Bernoulli considered the problem on the form of a uniform planar elastic bar bent by the external force $F$. His hypothesis was that the bending moment of the rod is equal to $\dfrac{\cal B}{R}$, where $\cal B$ is the ``flexural rigidity'', and $R$ is the radius of curvature of bent bar. For elastic bar of unit excursion built vertically into a horizontal wall and bent by a load sufficient to make its top horizontal (rectangular elastica, see Fig.~\ref{fig:rect_el}), James Bernoulli obtained the ODEs
$$
dy = \frac{x^2 \, dx}{\sqrt{1-x^4}}, \qquad
ds = \frac{dx}{\sqrt{1-x^4}},
\qquad x \in [0,1],
$$
(where $(x,y)$ is the elastic bar, and $s$ is its length parameter), integrated them in series and calculated precise upper and lower bounds for their value at the endpoint $x=1$, see~\cite{JBernoulli}.

\onefiglabelsize{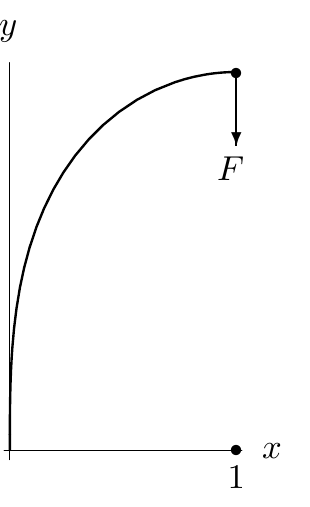}{James Bernoulli's rectangular elastica}{fig:rect_el}{0.25}

In 1742 Daniel Bernoulli in his letter~\cite{DBernoulli} to Leonhard Euler wrote that the elastic energy of bent rod is proportional to the magnitude
$$
E = \int \frac{ds}{R^2}
$$
and suggested to find the elastic curves from the variational principle $E \to \min$.

At that time Euler was writing his treatise on the calculus of variations ``Methodus inveniendi \dots''~\cite{euler} published in 1744, and he adjoined to his book an appendix ``De curvis elasticis'', where he applied the newly developed techniques to the problem on elasticae. Euler considered a thin homogeneous elastic plate, rectilinear in the natural (unstressed) state. For the profile of the plate, Euler stated the following problem:
\be{euler_problem}
\begin{split}
&\text{\em ``That among all curves of the same length which not only } \\
&\text{\em pass through the points $A$ and $B$, but are also tangent } \\ 
&\text{\em to given straight lines at these points, that curve be determined} \\
&\text{\em in which the value of  $\ds \int_A^B \frac{ds}{R^2}$ be a minimum.''}
\end{split}
\ee

Euler wrote down the ODE known
now as Euler-Lagrange equation for the corresponding
problem of calculus of variations and reduced it to the equations:
\be{dyds}
dy = \frac{(\a + \b x + \g x^2)\, dx}{\sqrt{a^4 - (\a + \b x + \g x^2)^2}},
\qquad
ds = \frac{a^2\, dx}{\sqrt{a^4 - (\a + \b x + \g x^2)^2}},
\ee
where $\ds\frac{\a}{a^2}$, $\ds\frac{\b}{a}$ and $\g$ are real parameters expressible in terms of $\cal B$, the load of the elastic rod, and its length. Euler studied the quadrature determined by the first of equations~\eq{dyds}. In the modern terminology, he investigated the qualitative behavior of the elliptic functions that parametrize the elastic curves via the qualitative analysis of the determining ODEs. Euler described all possible types of elasticae and indicated the values of parameters for which these types are realized (see a copy of Euler's original sketches at Fig.~\ref{fig:Euler}).

\bigskip

\onefiglabelsize{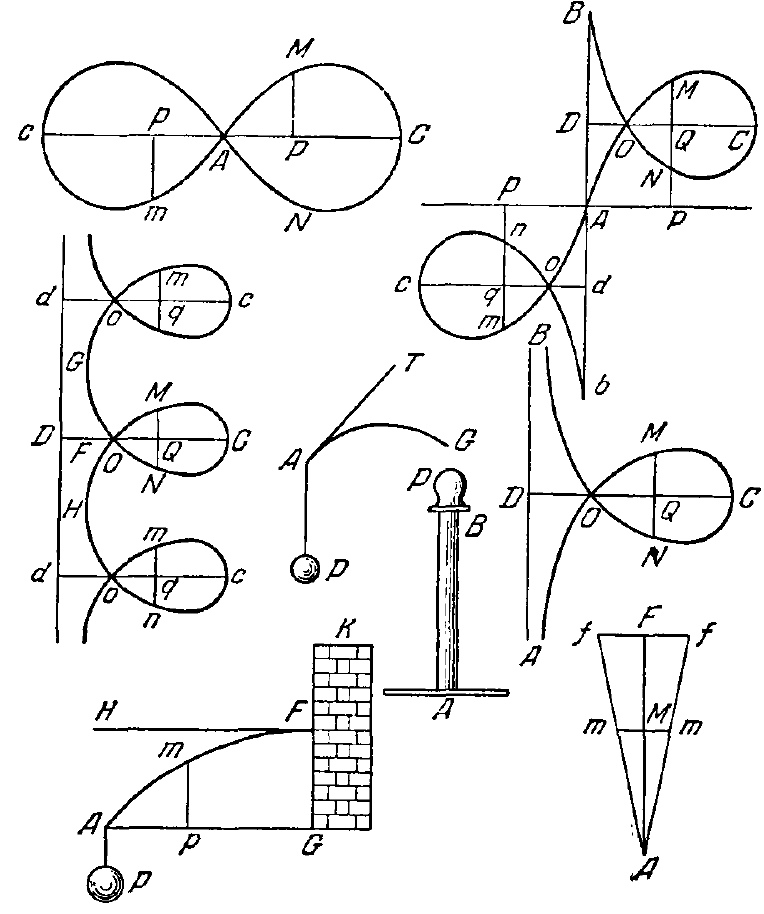}{Euler's sketches}{fig:Euler}{0.6}

Euler divided all elastic curves into nine classes, they are plotted respectively as follows:
\begin{enumerate}
\item
straight line, Fig.~\ref{fig:elastica1},
\item
Fig.~\ref{fig:elastica2},
\item
rectangular elastica, Fig.~\ref{fig:elastica3},
\item
Fig.~\ref{fig:elastica4},
\item
periodic elastica in the form of figure 8, Fig.~\ref{fig:elastica5},
\item
Fig.~\ref{fig:elastica6},
\item
elastica with one loop, Fig.~\ref{fig:elastica7},
\item
Fig.~\ref{fig:elastica8},
\item
circle, Fig.~\ref{fig:elastica9}.
\end{enumerate}

Following the tradition introduced by A.E.H.Love~\cite{love}, the elastic curves with inflection points (classes 2--6) are called \ddef{inflectional}, the elastica of class 7 is called \ddef{critical}, and elasticae without inflection points of class 8 are called \ddef{non-inflectional}.

Further, Euler established the magnitude of the force applied to the elastic plate that results in each type of elasticae. He indicated the experimental method for evaluation of the flexural rigidity of the elastic plate by its form under bending. Finally, he studied the problem of stability of a column modeled by the loaded rod whose lower end is constrained to remain vertical, by presenting it as an elastica of the class 2 close to the straight line (thus a sinusoid).

After the work of Leonhard Euler, the elastic curves are called \ddef{Euler elasticae}.

The first explicit parametrization of Euler elasticae was performed by L.Saal\-ch\"utz in 1880~\cite{saalchutz}.

In 1906 the future Nobel prize-winner Max Born defended a Ph.D. thesis called ``Stability of elastic lines in the plane and the space''~\cite{born}. Born considered the problem on elasticae as a problem of calculus of variations and derived from Euler-Lagrange equation that its solutions $(x(t), y(t))$ satisfy the ODEs of the form:
\begin{align}
&\dx = \cos \t, \qquad \dy = \sin \t, \nonumber \\
&A \ddot \t + R \sin (\t - \g) = 0, \qquad A, \ R, \ \g = \const,
\label{pend1}
\end{align}
thus the angle $\t$ determining the slope of elasticae satisfies the equation of the 
\ddef{mathematical pendulum}~\eq{pend1}.

Further, Born studied stability of elasticae with fixed endpoints and fixed tangents at the endpoints. Born proved that an elastic arc  without inflection points is stable (in this case the angle $\t$ is monotone, thus it can be taken as a parameter along elastica; Born showed that the second variation of the functional of elastic energy $\ds E = \dfrac 12 \int \dot \t^2 \, dt$ is positive). In the general case, Born wrote down the Jacobian that vanishes at conjugate points. Since the functions entering this Jacobian were too complicated, Born restricted himself to numerical investigation of conjugate points. He was the first to plot elasticae numerically and check the theory against experiments on elastic rods, see the photos from Born's thesis at Fig.~\ref{fig:born1}. Moreover, Born studied stability of Euler elasticae with various other boundary conditions, and obtained some results for elastic curves in $\R^3$.

\onefiglabelsize{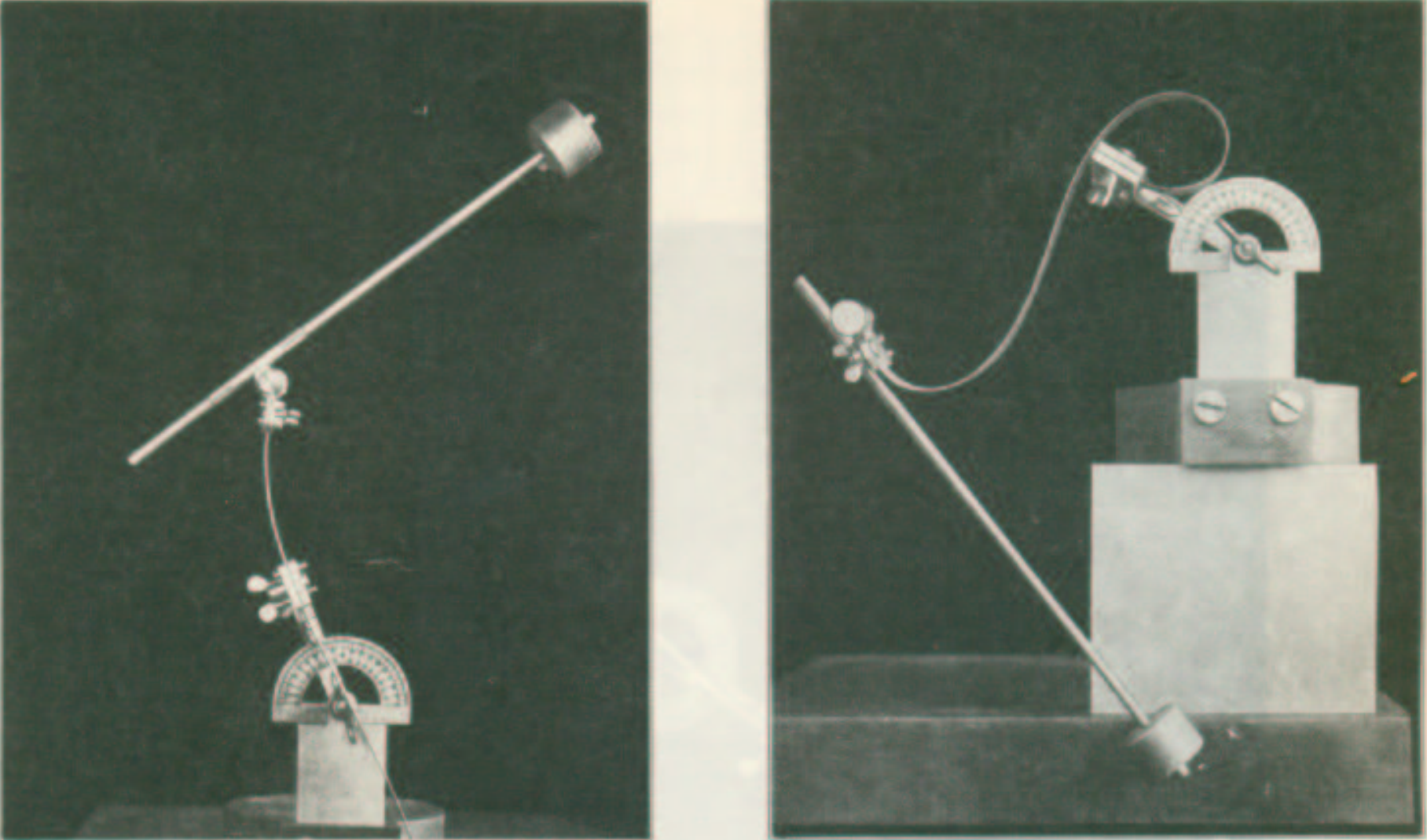}{Max Born's experiments}{fig:born1}{1}

In 1986 A.Arthur and G.R.Walsh~\cite{arthur_walsh} 
and, independently,  in 1993 V.Jur\-dje\-vic~\cite{jurd_ball_plate}
discovered that Euler elasticae appear in the ball-plate problem stated as follows. Consider a ball rolling on a horizontal plane without slipping or twisting. The problem is to roll the ball from an initial contact configuration (defined by contact point of the ball with the plane, and orientation of the ball in the 3-space) to a terminal contact configuration, so that the curve traced by the contact point in the plane was the shortest possible. Arthur and Walsh, and Jurdjevic showed that such optimal curves are Euler elasticae.
 Moreover, Jurdjevic also extensively studied the elastic problem in $\R^3$, its analogs in the sphere $S^3$, and in the Lorentz space $H^3$~\cite{jurd_noneuclid, jurd_book}.

In 1993 R.Brockett and L.Dai~\cite{brock_dai} discovered that Euler elasticae are projections of optimal trajectories in the nilpotent sub-Riemannian problem with the growth vector (2,3,5) known also as generalized Dido problem~\cite{dido_exp, max1, max2, max3}.

Elasticae were considered in approximation theory as nonlinear splines~\cite{birkhoff64, jerome73, jerome75, golumb_jerome82, linner96}, in computer vision as a maximum likelihood reconstruction of occluded edges~\cite{mumford}, their 3-dimensional analogues are used in the modeling of DNA minicircles~\cite{manning96, manning98} etc.

Euler elasticae and their various generalizations play an important role in modern mathematics, mechanics, and their applications. Although, the initial variational problem as it was stated by Euler~\eq{euler_problem} is far from complete solution: neither local nor global optimality of Euler elasticae is studied. This is the first of two planned works  that will give a complete description of local optimality, and present an essential progress in the study of the global optimality of elasticae. In this paper we give an upper bound on the cut points along Euler elasticae, i.e., points where they lose their global optimality. In the next work~\cite{elastica_conj} we obtain a complete characterization of conjugate points, i.e., points where elasticae lose their local optimality.

We would like to complete this historical introduction by two phrases of S.Antman~\cite{antman}. On the one hand,
``Fortunately Euler left some unsolved issues for his successors,''
but on the other hand,
``There is unfortunately a voluminous and growing literature devoted to doing poorly what Euler did well.''
With the hope to contribute to the first tradition rather than to the second one, we start this work.

\section{Problem statement}
\label{sec:problem}

\subsection{Optimal control problem}
First we state the elastic problem mathematically.
Let a homogeneous elastic rod in the two-dimensional Euclidean plane $\R^2$ have a fixed length $l>0$. Take any  points $a_0, \ a_1 \in \R^2$ and arbitrary unit tangent vectors at these points $v_i \in T_{a_i} \R^2$, $|v_i| = 1$, $i=0,1$. The problem consists in finding the profile of a rod $\map{\g}{[0, t_1]}{\R^2}$, starting at the point $a_0$ and coming to the point $a_1$ with the corresponding tangent vectors $v_0$ and $v_1$:
\begin{align}
&\g(0) = a_0, \qquad \g(t_1) = a_1, \label{bound1} \\
&\dot\g(0) = v_0, \qquad \dot\g(t_1) = v_1, \label{bound2}
\end{align}
with the minimal elastic energy.
The curve $\g(t)$ is assumed absolutely continuous with Lebesgue square-integrable curvature $k(t)$. We suppose that $\g(t)$ is arc-length parametrized, i.e., $|\dot \g(t)| \equiv 1$, so the time of motion along the curve $\g$ coincides with its length:
\be{t1=l}
t_1 = l.
\ee
The elastic energy of the rod is measured  by the integral
$$
J = \frac 12 \int_0^{t_1} k^2(t) \, dt.
$$

We choose Cartesian coordinates $(x,y)$ in the two-dimensional plane $\R^2$. Let the required curve be parameterized as $\g(t) = (x(t), y(t))$, $t \in [0, t_1]$, and let its endpoints have coordinates $a_i = (x_i, y_i)$, $i = 0, 1$. Denote by $\t$ the angle between the tangent vector to the curve $\g$ and the positive direction of the axis $x$. Further, let the tangent vectors at the endpoints of $\g$ have coordinates $v_i = (\cos \t_i, \sin \t_i)$, $i = 0, 1$, see Fig.~\ref{fig:prob_statement}.

\onefiglabel{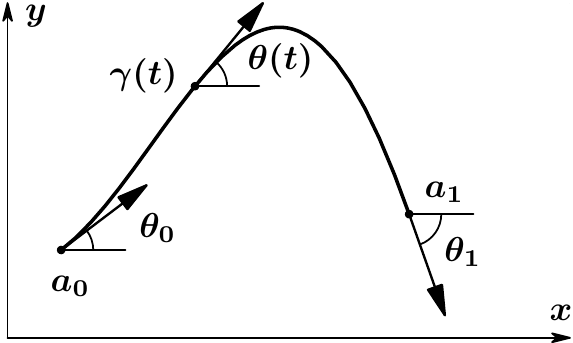}{Statement of Euler's problem}{fig:prob_statement}

Then the required curve $\g(t) = (x(t), y(t))$ is determined by a trajectory of the following control system:
\begin{align}
&\dx = \cos \t, \label{sys1} \\
&\dy = \sin \t, \label{sys2} \\
&\dot \t = u, \label{sys3} \\
&q=(x,y,\t) \in M= \R^2_{x,y} \times S^1_{\t}, \qquad u \in \R,  \label{sys4} \\
&q(0) = q_0 = (x_0, y_0, \t_0), \qquad
q(t_1) = q_1 = (x_1, y_1, \t_1),
\qquad t_1 \text{ fixed}. \label{sys5}
\end{align}

For an arc-length parametrized curve, the curvature is, up to sign, equal to the angular velocity: $k^2 = \dth^2 = u^2$, whence we obtain the cost functional
\be{J}
J = \frac 12 \int_0^{t_1} u^2(t) \, dt \to \min.
\ee

We study the optimal control problem~\eq{sys1}--\eq{J}. Following V.Jurdjevic~\cite{jurd_book}, this problem is called \ddef{Euler's elastic problem}. Admissible controls are $u(t) \in L_2[0, t_1]$, and admissible trajectories are absolutely continuous curves $q(t) \in AC([0, t_1]; M)$.

In vector notations, the problem reads as follows:
\begin{align*}
&\dq = X_1(q) + u X_2(q), \qquad q \in M = \R^2 \times S^1, \quad u \in \R, \qquad\qquad\qquad (\S)\\
&q(0) = q_0, \qquad q(t_1) = q_1, \qquad t_1 \text{ fixed}, \\
&J = \frac 12 \int_0^{t_1} u^2 dt \to \min, \\
&u \in L_2[0, t_1],
\end{align*}
where the vector fields in the right-hand side of system $\S$ are:
$$
X_1 = \cos \t \pder{}{x} + \sin \t \pder{}{y},
\qquad X_2 = \pder{}{\t}.
$$
Notice the multiplication table in the Lie algebra of vector fields generated by $X_1$, $X_2$:
\begin{align}
&[X_1, X_2] = X_3 = \sin \t \pder{}{x} - \cos \t \pder{}{y}, \label{X1X2} \\
&[X_2, X_3] = X_1, \qquad [X_1, X_2] = 0. \label{X2X3}
\end{align}

\subsection[Left-invariant problem on the group of motions of a plane]
{Left-invariant problem \\ on the group of motions of a plane}
Euler's elastic problem has obvious symmetries --- parallel translations and rotations of the two-dimensional plane $\R^2$. Thus it can naturally be stated as an invariant problem on the group of proper motions of the two-dimensional plane
$$
\E(2) =
\left\{
\left(
\begin{array}{ccc}
\cos \t & -\sin \t & x \\
\sin \t & \cos \t & y \\
0       & 0       & 1
\end{array}
\right)
\mid (x,y) \in \R^2, \ \t \in S^1
\right\}.
$$
Indeed, the state space of the problem $M = \R^2_{x,y} \times S^1_{\t}$ is parametrized by matrices of the form
$$
q =
\left(
\begin{array}{ccc}
\cos \t & -\sin \t & x \\
\sin \t & \cos \t & y \\
0       & 0       & 1
\end{array}
\right)
\in \E(2),
$$
and dynamics~\eq{sys1}--\eq{sys3} is left-invariant on the Lie group $\E(2)$:
\begin{align*}
\dot q
&=
\der{}{t}
\left(
\begin{array}{ccc}
\cos \t & - \sin \t & x \\
\sin \t &  \cos \t & y \\
0 & 0 & 1
\end{array}
\right)
=
\left(
\begin{array}{ccc}
- u \sin \t  & - u \cos \t    & \cos \t \\
u \cos \t    &  -u \sin \t   & \sin \t \\
0 & 0 & 0
\end{array}
\right) = \\
&=
\left(
\begin{array}{ccc}
\cos \t & - \sin \t & x \\
\sin \t &  \cos \t & y \\
0 & 0 & 1
\end{array}
\right)
\left(
\begin{array}{ccc}
0 & - u & 1 \\
u &  0 & 0 \\
0 & 0 & 0
\end{array}
\right).
\end{align*}
The Lie algebra of the Lie group $\E(2)$ has the form
$$
\e(2) = \spann(E_{21}-E_{12}, E_{13}, E_{23}),
$$
where $E_{ij}$ denotes the $3 \times 3$ matrix with the only identity entry in the $i$-th row and $j$-th column, and zeros elsewhere. In the basis
$$
e_1 = E_{13}, \qquad e_2 = E_{21}-E_{12}, \qquad e_3 = - E_{23},
$$
the multiplication table in the Lie algebra $\e(2)$ takes the form
$$
[e_1,e_2] = e_3, \qquad [e_2,e_3] = e_1, \qquad [e_1,e_3] = 0.
$$

Then Euler's elastic problem becomes the following left-invariant problem on the Lie group $\E(2)$:
\begin{align*}
&\dq = X_1(q) + u X_2(q), \qquad q \in \E(2), \quad u \in \R. \\
&q(0) = q_0, \qquad q(t_1) = q_1, \qquad t_1 \text{ fixed}, \\
&J = \frac 12 \int_0^{t_1} u^2 dt \to \min,
\end{align*}
where
$$
X_i(q) = q \, e_i, \qquad i = 1, 2, \quad q \in \E(2),
$$
are basis left-invariant vector fields on $\E(2)$ (here $q \, e_i$ denotes the product of $3 \times 3$ matrices).

\subsection[Continuous symmetries and normalization of conditions of the problem]
{Continuous symmetries \\ and normalization of conditions of the problem}
\label{subsec:cont_symm}
Left translations on the Lie group $\E(2)$ are symmetries of Euler's elastic problem. By virtue of these symmetries, we can assume that initial point of trajectories is the identity element of the group
$$
\Id =
\left(
\begin{array}{ccc}
1 & 0 & 0 \\
0 & 1 & 0 \\
0 & 0 & 1
\end{array}
\right),
$$
i.e.,
\be{q0=0}
q_0 = (x_0, y_0, \t_0) = (0, 0, 0).
\ee
In other words, parallel translations in the plane $\R^2_{x,y}$ shift the initial point of the elastic rod $\g$ to the origin $(0,0) \in \R^2_{x,y}$, and rotations of this plane combine the initial tangent vector $\dot \g(0)$ with the positive direction of the axis $x$.

Moreover, one can easily see one more continuous family of symmetries of the problem --- dilations in the plane $\R^2_{x,y}$. Consider the following one-parameter group of transformations of variables of the problem:
\be{group1}
(x,y,\t,t,u,t_1,J) \mapsto (\tilde x, \tilde y, \tilde \t, \tilde t, \tilde u, \tilde t_1, \tilde J) = (e^sx, e^sy, \t, e^st, e^{-s}u, e^st_1, e^{-s}J).
\ee
One immediately checks that Euler's problem is preserved by this group of transformations. Thus, choosing $s = - \ln t_1$, we can assume that $t_1 = 1$.
In other words, we obtain an elastic rod of unit length by virtue of dilations in the plane $\R^2_{x,y}$.

In the sequel we usually fix the initial point $q_0$ as in~\eq{q0=0}. Although, the terminal time $t_1$ will remain a parameter, not necessarily equal to 1. 

\section{Attainable set}
\label{sec:att_set}
Consider a smooth control system of the form
\be{qdot}
\dq = f(q,u), \qquad q \in M, \quad u \in U.
\ee
Let $u= u(t)$ be an admissible control, and let $q_0 \in M$. Denote by $q(t;u,q_0)$ the trajectory of the system corresponding to the control $u(t)$ and satisfying the initial condition $q(0; u, q_0) = q_0$. \ddef{Attainable set} of system~\eq{qdot} from the point $q_0$ for time $t_1$ is defined as follows:
$$
\CA_{q_0}(t_1) = \{q(t_1;u,q_0) \mid u = u(t) \text{ admissible control}, \ t \in [0, t_1]\}.
$$
Moreover, one can consider the attainable set for time not greater than $t_1$:
$$
\CA_{q_0}^{t_1} = \bigcup_{0 \leq t \leq t_1} \CA_{q_0}(t),
$$
and the attainable set for arbitrary nonnegative time:
$$
\CA_{q_0} = \bigcup_{0 \leq t < \infty} \CA_{q_0}(t).
$$
The \ddef{orbit} of the system~\eq{qdot} is defined as
$$
\CO_{q_0} = \left\{
e^{\tau_N f_N} \circ \dots \circ e^{\tau_1 f_1}(q_0) \mid \tau_i \in \R, \ f_i = f(\cdot, u_i), \ u_i \in U, \ N \in \N \right\},
$$
where $e^{\tau_i f_i}$ is the flow of the vector field $f_i$.
See~\cite{jurd_book, notes} for basic properties of attainable sets and orbits.

In this section we describe the orbit and attainable sets for Euler's elastic problem.

Multiplication rules~\eq{X1X2}, \eq{X2X3} imply that control system~$\S$ is full-rank:
$$
\Lie_q(X_1, X_2) = \spann(X_1(q), X_2(q), X_3(q)) = T_q M \quad \forall \ q \in M.
$$
By the Orbit Theorem of Nagano-Sussmann~\cite{jurd_book, notes}, the whole state space is a single orbit:
$$
\CO_{q_0} = M \qquad \forall \, q_0 \in M.
$$
Moreover, the system is completely controllable:
$$
\CA_{q_0} = M \qquad \forall \, q_0 \in M.
$$
This can be shown either by applying a general controllability condition for control-affine systems with recurrent drift (Th.~5 in Sec.~4~\cite{jurd_book}), or via controllability test for left-invariant systems on semi-direct products of Lie groups (Th.~10 in Sec.6~\cite{jurd_book}).

On the other hand, it is obvious that system~$\S$ is not completely controllable on a compact time segment $[0, t_1]$:
$$
\CA_{q_0}^{t_1} \neq M
$$
in view of the bound $(x(t)-x_0)^2 + (y(t) - y_0)^2 \leq t_1^2$, the distance between the endpoints of the elastic rod should not exceed the length of the rod. We have the following description of the exact-time attainable sets for Euler's problem.

\begin{theorem}
\label{th:att_set}
Let $q_0 = (x_0, y_0, \t_0) \in M = \R^2 \times S^1$ and $t_1 > 0$. Then the attainable set of system~$\S$ is
\begin{align*}
\CA_{q_0}(t_1) =
\{ (x,y,\t) \in M 
&\mid
(x-x_0)^2 + (y - y_0)^2 < t_1^2
\\
&\qquad 
\text{ or }
(x,y,\t) = (x_0 + t_1 \cos \t_0, y_0 + t_1 \sin \t_0, \t_0) \}.
\end{align*}
\end{theorem}
\begin{proof}
In view of the continuous symmetries of the problem (see Subsec.~\ref{subsec:cont_symm}), it suffices to prove this theorem in the case $q_0 = \Id = (0,0,0)$, $t_1 = 1$, so we show that
$$
\CA = \CA_{\Id}(1) =
\left\{ (x,y,\t) \in M \mid
x^2 + y^2 < 1
\text{ or }
(x,y,\t) = (1, 0, 0) \right\}.
$$

(1) It is easy to see that
$$
x^2 + y^2 > 1 \then q = (x,y,\t) \notin \CA.
$$
Indeed, the curve $\g(t) = (x(t),y(t))$ has unit velocity, thus
\be{x2+y2}
x^2(t) + y^2(t) = |\g(t)-\g(0)|^2 \leq t^2.
\ee
So points $(x,y)$ at the distance greater then 1 from the origin $\g(0) = (0,0)$ are not attainable from the origin for time $1$.

(2) Let $x^2 + y^2 = 1$. We show that
$$
q = (x,y,\t) \in \CA \iff q = (1,0,0).
$$
It is obvious that the point $q = (1,0,0)$ is attainable from the point $q_0 = (0,0,0)$ for time 1 via the control $u(t) \equiv 0$, $t \in [0, 1]$.

Conversely, let $q = (x,y,\t) \in \CA$. Consider the function $f(t) = x^2(t) + y^2(t) \in W_{2,2}$, where $q(t) = (x(t), y(t), \t(t))$, $t \in [0, 1]$, is a trajectory of  system~$\S$ connecting the points $q_0$ and $q$. We prove that $f(t) \equiv t^2$, $t \in [0,1]$. It was shown in~\eq{x2+y2} that $f(t) \leq t^2$, $t \in [0, 1]$. The function $f(t)$ takes the same values as $t^2$ at the endpoints of the segment $[0, 1]$. So if $f(t) \not\equiv t^2$, $t \in [0, 1]$, then $f'(t_0) > (t^2)'|_{t=t_0} = 2 t_0$ at some point $t_0 \in [0, 1]$. But this inequality is impossible in view of the chain
$$
|f'(t)| = 2 |\dx x + \dy y| \leq 2 \sqrt{x^2 + y^2} = 2 \sqrt{f(t)} \leq 2 t, \qquad t \in [0, 1].
$$
Hence $f(t) \equiv t^2$, $t \in [0,1]$, and the preceding inequalities turn into equalities. Then $(x(t), y(t)) = \a(t)(\dx(t), \dy(t))$, $\a(t) \geq 0$, whence it follows that $\dot \t \equiv 0$ and $q = (1,0,0)$.

(3) Finally, we show that for any angle $\t \in S^1$
$$
x^2 + y^2 < 1 \then q = (x,y,\t) \in \CA.
$$

First we mention some simple trajectories of system~$\S$. In the case $u \equiv 0$ we obtain a straight line $(x(t), y(t))$, and in the case $u \equiv C \neq 0$ the curve $(x(t),y(t))$ is a circle of radius $\dfrac{1}{|C|}$. Notice that the time of complete revolution along such circle is $\dfrac{2 \pi}{|C|} \to 0$ as $C \to \infty$.

Now we construct a trajectory of system~$\S$ connecting the initial point $q_0= (0,0,0)$ with the terminal one $q = (x,y,\t)$, $x^2 + y^2 < 1$, for time 1.

Assume first that $(x,y) \neq (0,0)$. In the plane $\R^2_{x,y}$, construct a circle of small radius starting at the point $(0,0)$ with the tangent vector $(1,0)$, and a circle of small radius starting at the point $(x,y)$ with the tangent vector $(\cos \t, \sin \t)$. It is obvious that there exists a straight line segment in the plane $\R^2_{x,y}$ tangent at its initial point to the first circle, and tangent at its terminal point to the second circle, in such a way that direction of motion along the circles and the segment was consistent, see Fig.~\ref{fig:att_set}.

\onefiglabelsize{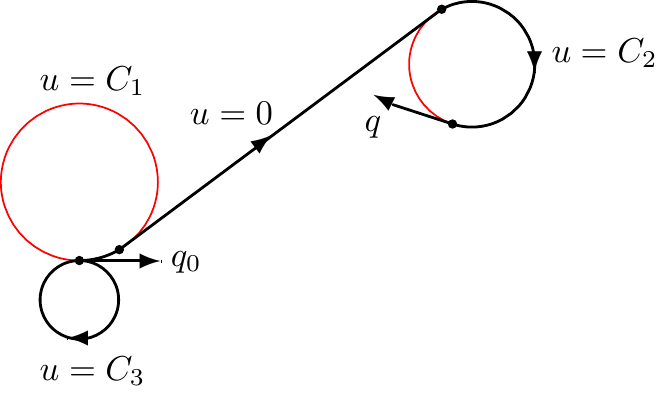}{Steering $q_0$ to $q$}{fig:att_set}{0.55}

In such a way we obtain a trajectory of system~$\S$ corresponding to a piecewise constant control $u(\cdot)$ taking some values $C_1 \neq 0$, 0, $C_2 \neq 0$; this trajectory projects to the plane  $\R^2_{x,y}$ to the concatenation of a circle arc, a line segment, and a circle arc. Choosing circles of sufficiently small radii $\dfrac{1}{|C_1|}$, $\dfrac{1}{|C_2|}$, we can obtain the total time of motion along this trajectory $\tau < 1$. In order to have a trajectory with the same endpoints at the time segment $t \in [0,1]$, it is enough to add a circle of radius $\ds\frac{1}{C_3}= \dfrac{1-\tau}{2 \pi}$ before the first circle, see Fig.~\ref{fig:att_set}. The trajectory constructed steers the point $q_0 = (0,0,0)$ to the point $q = (x,y,\t)$ for time 1.

If $(x,y) = (0,0)$, then we move from the point $(0,0)$ along a short segment to a point $(\eps, 0)$, and repeat the preceding argument.

Now the statement of Th.~\ref{th:att_set} follows from the statements (1)--(3) proved above.
\end{proof}

The following properties of attainable sets of system~$\S$ follow immediately from  Th.~\ref{th:att_set}.

\begin{corollary}
\label{cor:att_set}
Let $q_0$ be an arbitrary point of $M$. Then:
\begin{itemize}
\item[$(1)$]
$\CA_{q_0}(t_1) \subset  \CA_{q_0}(t_2)$ for any $0 < t_1 < t_2$.
\item[$(2)$]
$\CA_{q_0}^t = \CA_{q_0}(t)$ for any $t> 0$.
\item[$(3)$]
$q_0 \in \intt \CA_{q_0}^t$ for any $t > 0$.
\end{itemize}
\end{corollary}

Item (3) means that system~$\S$ is small-time locally controllable. Although, the restriction of $\S$ to a small neighborhood of a point $q_0 \in M$ is not controllable since some points in the neighborhood of $q_0$ are reachable from $q_0$ by trajectories of $\S$ far from $q_0$.

Topologically, the attainable set $\CA_{q_0}(t)$ is an open solid torus united with a single point at its boundary.
In particular, the attainable set is neither open nor closed.


In the sequel we study Euler's problem under the natural condition
\be{q1inA}
q_1 \in \CA_{q_0}(t_1).
\ee

\section{Existence and regularity of optimal solutions}
\label{sec:exist}
We apply known results of optimal control theory in order to show that in Euler's elastic problem optimal controls exist and are essentially bounded.

\subsection{Embedding the problem into $\R^3$}
\label{subsec:embed}
The state space and attainable sets of Euler's problem have nontrivial topology, and we start from embedding the problem into Euclidean space. By Th.~\ref{th:att_set}, the attainable set $\CA = \CA_{q_0}(1)$, $q_0 = (0,0,0)$, is contained in the set
$$
\tM = \cl \CA = \{(x,y,\t) \in M \mid x^2 + y^2 \leq 1 \}.
$$
Moreover, by item (2) of Corollary~\ref{cor:att_set}, any trajectory of system $\S$ starting at $q_0$ does not leave the set $\tM$ at the time segment $t \in [0,1]$. So this set can be viewed as a new state space of the problem. The set $\tM$ is embedded into the Euclidean space $\R^3_{x_1x_2x_3}$ by the diffeomorphism
\begin{align}
&\map{\Phi}{\tM}{\R^3_{x_1x_2x_3}}, \nonumber \\
&\Phi(x,y,\t) = (x_1, x_2, x_3) = ((2+x) \cos \t, (2+x) \sin \t, y). \label{Phi}
\end{align}
The image
$$
\Phi(\tM) = \{(x_1, x_2, x_3) \in \R^3 \mid (2 - \rho)^2 + x_3^2 \leq 1 \}, \qquad \rho = \sqrt{x_1^2 + x_2^2},
$$
is the closed solid torus. 

In the coordinates $(x_1, x_2, x_3)$, Euler's problem reads as follows:
\begin{align}
&\dx_1 = \frac{x_1^2}{x_1^2 + x_2^2} - u x_2, \label{pr1} \\
&\dx_2 = \frac{x_1 x_2}{x_1^2 + x_2^2} + u x_1, \label{pr2} \\
&\dx_3 = \frac{x_2}{\sqrt{x_1^2 + x_2^2}}, \label{pr3} \\
&x = (x_1, x_2, x_3) \in \Phi(\tM), \qquad u \in \R, \label{pr4} \\
&x(0) = x^0 = (2,0,0), \qquad x(1) = x^1 = (x_1^1, x_2^1,x_3^1), \label{pr5} \\
&J = \frac 12 \int_0^1 u^2 dt \to \min,  \label{pr6} \\
&u(\cdot) \in L_2[0,1], \qquad x(\cdot) \in AC[0,1]. \label{pr7}
\end{align}

\subsection{Existence of optimal controls}
\label{subsec:exist}
First we cite an appropriate general existence result for control-affine systems from Sec.~11.4.C of the textbook by L.Cesari~\cite{cesari}. Consider optimal control problem of the form:
\begin{align}
&\dx = f(t,x) + \sum_{i=1}^m u_i g_i(t,x), \qquad x \in X \subset \R^n, \quad u = (u_1, \dots, u_m) \in \R^m, \label{gp1} \\
&J = \int_0^{t_1} f_0(t,x,u) \, dt \to \min, \label{gp2} \\
&x(\cdot) \in AC([0,t_1], X), \qquad u(\cdot) \in L_2([0,t_1], \R^m), \label{gp3} \\
&x(0) = x^0, \qquad x(t_1) = x^1, \qquad t_1 \text{ fixed}. \label{gp4}
\end{align}
For such a problem, a general existence theorem  is formulated as follows.

\begin{theorem}[Th.~11.4.VI~\cite{cesari}]
\label{th:cesari}
Assume that the following conditions hold:
\begin{itemize}
\item[$(C')$]
the set $X$ is closed, and the function $f_0$ is continuous on $[0,t_1] \times X \times \R^m$,
\item[$(L_1)$]
there is a real-valued function $\psi(t) \geq 0$, $t \in [0, t_1]$, $\psi \in L_1[0,t_1]$, such that $f_0(t,x,u) \geq -\psi(t)$ for $(t,x,u) \in [0,t_1] \times X \times \R^m$ and almost all $t$,
\item[$(CL)$]
the vector fields $f(t,x)$, $g_1(t,x)$, \dots, $g_m(t,x)$ are continuous on $[0,t_1] \times X$,
\item
the vector fields $f(t,x)$, $g_1(t,x)$, \dots, $g_m(t,x)$ have bounded components on $[0,t_1] \times X$,
\item
the function $f_0(t,x,u)$ is convex in $u$ for all $(t,x) \in [0, t_1] \times X$,
\item
$x^1 \in \CA_{x^0}(t_1).$
\end{itemize}
Then there exists an optimal control $u \in L_2([0,t_1], \R^m)$ for the problem~\eq{gp1}--\eq{gp4}.
\end{theorem}

For Euler's problem embedded into $\R^3$~\eq{pr1}--\eq{pr7}, we have:
\begin{itemize}
\item
$m=1$,
\item
the set $X = \Phi(\tM)$ is compact,
\item
the function $f_0 = u^2$ is continuous, nonnegative, and convex,
\item
the vector fields $f(x) = \ds\frac{x_1^2}{x_1^2 + x_2^2} \pder{}{x_1} + \frac{x_1x_2}{x_1^2+x_2^2} \pder{}{x_2} + \frac{x_2}{\sqrt{x_1^2+x_2^2}}\pder{}{x_3}$, $g_1(x) = \ds -x_2 \pder{}{x_1} + x_1 \pder{}{x_2}$ are continuous and have bounded components on $X$,
\item $x^1 \in \CA_{x^0}(t_1)$ as supposed in~\eq{q1inA}.
\end{itemize}
So all hypotheses of Th.~\ref{th:cesari} are satisfied, and there exists optimal control $u \in L_2[0,t_1]$ for Euler's problem.

\subsection{Boundedness of optimal controls}
\label{subsec:bound_cont}

One can prove essential boundedness of optimal control in Euler's elastic problem by virtue of the following general result due to A.Sarychev and D.Torres.

\begin{theorem}[Th. 1 \cite{sar_torres}]
\label{th:sar_torres}
Consider an optimal control problem of the form \eq{gp1}--\eq{gp4}. Let $f_0 \in C^1([0,t_1] \times X \times \R^m, \R)$, $f, \ g_i \in C^1([0,t_1]\times X; \R^n)$, $i = 1, \dots, m$, and $\f(t,x,u) = f(t,x) + \sum_{i=1}^m u_i g_i(t,x)$.

Under the hypotheses:
\begin{itemize}
\item[$(H1)$]
full rank condition: $\dim \spann(g_1(t,x), \dots, g_m(t,x)) = m$ for all $t \in[0,t_1]$ and $x \in X$;
\item[$(H2)$]
coercivity:
there exists a function $\map{\t}{\R}{\R}$ such that $f_0(t,x,u) \geq \t(\|u\|) > \zeta \quad \forall \ (t,x,u) \in [0, t_1] \times X \times \R^m$,  and $\ds \lim_{r \to + \infty} \dfrac{r}{\t(r)} = 0$;
\item[$(H3)$]
growth condition: there exist constants $\g, \b, \eta$, and $\mu$, with $\g > 0$, $\b < 2$, and $\mu \geq \max \{ \beta - 2, \ -2\}$, such that, for all $t \in [0, t_1]$, $x \in X$, and $u \in \R^m$, it holds that
$$
(|f_{0t}| + |f_{0x_i}| + \|f_0 \f_t - f_{0t} \f\| + \|f_0 \f_{x_i} - f_{0 x_i} \f\|) \|u\|^{\mu} \leq \g f_0^{\b} + \eta, \quad i = 1, \dots, n,
$$
\end{itemize}
all optimal controls $u(\cdot)$ of the problem~\eq{gp1}--\eq{gp4} which are not abnormal extremal controls, are essentially bounded on $[0,t_1]$.
\end{theorem}

It is easy to see that all hypotheses of Th.~\ref{th:sar_torres} hold:
\begin{itemize}
\item[$(H1)$]
$g(x) = -x_2 \pder{}{x_1} + x_1 \pder{}{x_2} \neq 0$ on $X$;
\item[$(H2)$]
 $\t(r) = r^2$;
\item[$(H3)$]
$f_{0t} = f_{0x_i} = \f_t = 0$, $\|\f_{x_i}\| \leq C$ on $X$. The required bound $\|f_0 \f_{x_i} \| \cdot \|u\|^{\mu} \leq \g f_0^{\b} + \eta$ is satisfied for $\b = 1$, $\mu = 1$, $\g = C$, $\eta = 0$.
\end{itemize}

Thus in Euler's elastic problem all optimal controls which are not abnormal extremal controls are essentially bounded: $u(\cdot) \in L_{\infty}[0,t_1]$. In Subsec.~\ref{subsec:abnorm} we describe abnormal extremal controls, and obtain a similar inclusion for {\em all} optimal controls.

Meanwhile we cite one more general result valid for Euler's problem as well.

\begin{corollary}[Cor. 1 \cite{sar_torres}]
\label{cor:sar_torres}
Under conditions of Th.~$\ref{th:sar_torres}$, all minimizers of the problem~\eq{gp1}--\eq{gp4} satisfy the Pontryagin Maximum Principle.
\end{corollary}

We summarize or results for Euler's elastic problem derived in this section. Obviously, we can return back from the problem~\eq{pr1}--\eq{pr7} in $\R^3_{x_1x_2x_3}$  to the initial problem~\eq{sys1}--\eq{J} in $\R^2_{x,y} \times S^1_{\t}$.

\begin{theorem}
\label{th:exist}
Let $q_1 \in \CA_{q_0}(t_1)$.
\begin{itemize}
\item[$(1)$]
Then there exists an optimal control for Euler's problem~\eq{sys1}--\eq{J} in the class $u(\cdot) \in L_2[0,t_1]$.
\item[$(2)$]
If the optimal control $u(\cdot)$ is not an abnormal extremal control, then $u(\cdot) \in L_{\infty}[0,t_1]$. The corresponding optimal trajectory $q(\cdot)$ is Lipschitzian.
\item[$(3)$]
All optimal solutions to Euler's problem satisfy the Pontryagin Maximum Principle.
\end{itemize}
\end{theorem}

Certainly, Th.~\ref{th:exist} is not the best possible statement on regularity of solutions to Euler's problem. We will derive from Pontryagin Maximum Principle that optimal controls and optimal trajectories are analytic, see Th.~\ref{th:opt_analyt}.

\section{Extremals}
\label{sec:extremals}

\subsection{Pontryagin Maximum Principle}
\label{subsec:PMP}
In order to apply Pontryagin Maximum Principle (PMP) in invariant form, we recall the basic notions of the Hamiltonian formalism~\cite{jurd_book, notes}. Notice that the approach and conclusions of this section have much intersection with the book~\cite{jurd_book} by  V.Jurdjevic. 

Let $M$ be a smooth $n$-dimensional manifold, then its cotangent bundle $T^*M$ is a smooth $2n$-dimensional manifold. The canonical projection $\map{\pi}{T^*M}{M}$ maps a covector $\lam \in T_q^*M$ to the base point $q \in M$. The tautological 1-form $s \in \Lam^1(T^*M)$ on the cotangent bundle is defined as follows. Let $\lam \in T^*M$ and $v \in T_{\lam}(T^*M)$, then $\langle s_{\lam}, v \rangle = \langle \lam, \pi_* v \rangle$ (in coordinates $s = p \, dq$). The canonical symplectic structure on the cotangent bundle $\s \in \Lam^2(T^*M)$ is defined as $\s = ds$ (in coordinates $\s = dp \wedge dq$). To any Hamiltonian $h \in C^{\infty}(T^*M)$, there corresponds a Hamiltonian vector field on the cotangent bundle $\vh \in \Vec(T^* M)$ by the rule $\s_{\lam}(\cdot, \vh) = d_{\lam}h$.

Now let $M = \R^2_{x,y} \times S^1_{\t}$ be the state space of Euler's problem. Recall that the vector fields
$$
X_1 = \cos \t \pder{}{x} + \sin \t \pder{}{y}, \quad X_2 = \pder{}{\t},
\quad X_3 = \sin \t \pder{}{x} - \cos \t \pder{}{y}
$$
form a basis in the tangent spaces to $M$. The Lie brackets of these vector fields are given in~\eq{X1X2}, \eq{X2X3}. Introduce the linear on fibers of $T^*M$ Hamiltonians corresponding to these basis vector fields:
$$
h_i(\lam) = \langle \lam, X_i\rangle,
\qquad \lam \in T^*M, \quad i = 1, 2, 3,
$$
and the family of Hamiltonian functions
\begin{multline*}
h_u^{\nu}(\lam) = \langle\lam, X_1+uX_2\rangle + \frac{\nu}{2}u^2 = h_1(\lam) + u h_2(\lam) + \frac{\nu}{2} u^2,
\\
\lam \in T^*M, \quad  u \in \R, \quad  \nu \in \R.
\end{multline*}
the control-dependent Hamiltonian of PMP for Euler's problem~\eq{sys1}--\eq{J}.

By Th.~\ref{th:exist}, all optimal solutions to Euler's problem satisfy Pontryagin Maximum Principle. We write it in the following invariant form.

\begin{theorem}[Th. 12.3 \cite{notes}]
\label{th:PMP}
Let $u(t)$ and $q(t)$, $t \in [0, t_1]$, be an optimal control and the corresponding optimal trajectory in Euler's problem~\eq{sys1}--\eq{J}. Then there exist a curve $\lam_t \in T^*M$, $\pi(\lam_t) = q(t)$, $t \in [0, t_1]$, and a number $\nu \leq 0$ for which the following conditions hold for almost all $t \in [0,t_1]$:
\begin{align}
&\dlam_t = \vh^{\nu}_{u(t)}(\lam_t) = \vh_1(\lam_t) + u(t) \vh_2(\lam_t), \label{PMP1} \\
&h^{\nu}_{u(t)}(\lam_t) = \max_{u \in \R} h^{\nu}_{u}(\lam_t),   \label{PMP2} \\
&(\nu, \lam_t) \neq 0. \label{PMP3}
\end{align}
\end{theorem}

Using the coordinates $(h_1, h_2, h_3, x, y, \t)$, we can write the Hamiltonian system of PMP~\eq{PMP1} as follows:
\begin{align}
&\dh_1 = -uh_3, \label{H1} \\
&\dh_2 = h_3, \label{H2} \\
&\dh_3 = uh_1, \label{H3} \\
&\dx = \cos \t, \label{H4} \\
&\dy = \sin \t, \label{H5} \\
&\dth = u. \label{H6}
\end{align}
Notice that the subsystem for the vertical coordinates $(h_1, h_2, h_3)$~\eq{H1}--\eq{H3} is independent of the horizontal coordinates $(x,y,\t)$, this is a corollary of the left-invariant symmetry of system $\S$ and of appropriate choice of the coordinates $(h_1, h_2, h_3)$, see~\cite{notes}.

As usual, the constant parameter $\nu$ can be either zero (abnormal case), or negative (normal case, then one can normalize $\nu = -1$).

\subsection{Abnormal extremals}
\label{subsec:abnorm}
Consider first the \ddef{abnormal case}:
let
$
\nu = 0.
$
The maximality condition of PMP~\eq{PMP2} reads:
\be{humax}
h_u^{\nu}(\lam) = h_1(\lam) + u h_2(\lam) \to \max_{u\in \R},
\ee
thus $h_2(\lam_t) \equiv 0$ along an abnormal extremal $\lam_t$. 
Then Eq.~\eq{H2} yields $h_3(\lam_t)\equiv 0$, and Eq.~\eq{H3} gives $u(t) h_1(\lam_t) \equiv 0$.  But in view of the nontriviality condition of PMP~\eq{PMP3}, we have $h_1(\lam_t) \neq 0$, thus $u(t) \equiv 0$. So abnormal extremal controls in Euler's problem are identically zero. Notice that these controls are singular since they are not uniquely determined by the maximality condition of PMP~\eq{humax}.

Now we find the abnormal extremal trajectories. For $u \equiv 0$ the horizontal equations~\eq{H4}--\eq{H6} read
$$
\dq = X_1(q) \iff
\begin{cases}
\dx = \cos \t, \\
\dy = \sin \t, \\
\dth = 0,
\end{cases}
$$
and the initial condition $(x,y,\t)(0) = (0,0,0)$ gives
$$
x(t) = t, \qquad y(t) \equiv 0, \qquad \t(t) \equiv 0.
$$
The abnormal extremal trajectory through $q_0 = \Id$ is the one-parameter subgroup of the Lie group $\E(2)$ corresponding to the left-invariant field $X_1$. It is projected to the straight line $(x,y) = (t,0)$ in the plane $(x,y)$. The corresponding elastica is a straight line segment --- the elastic rod without any external forces applied. This is the trajectory connecting $q_0$ to the only attainable point $q_1$ at the boundary of the attainable set $\CA_{q_0}(t_1)$. 

For $u \equiv 0$ the elastic energy is $J = 0$, the absolute minimum. So the abnormal extremal trajectory $q(t)$, $t \in [0, t_1]$, is optimal; it gives an optimal solution for the boundary conditions $q_0 = (0,0,0)$, $q_1 = (t_1, 0, 0)$.

Combining the description of abnormal controls just obtained with Th.~\ref{th:exist}, we get the following statement.

\begin{theorem}
\label{th:exist2}
For any $q_1 \in \CA_{q_0}(t_1)$, the corresponding optimal control for Euler's problem~\eq{sys1}--\eq{J} is essentially bounded.
\end{theorem}

\subsection{Normal case}
\label{subsec:normal}
Now let
$\nu = -1.$
The maximality condition of PMP~\eq{PMP2} reads
$$
h_u^{-1} = h_1 + u h_2 - \frac 12 u^2 \to \max_{u \in \R},
$$
whence $\ds \pder{h_u^{-1}}{u} = h_2 - u = 0$ and
\be{u=h2}
u = h_2.
\ee
The corresponding normal Hamiltonian of PMP is
$$
H = h_1 + \frac12 h_2^2,
$$
and the \ddef{normal Hamiltonian system of PMP} reads
\be{dh1-h2h3}
\dlam = \vH(\lam) \iff
\begin{cases}
\dh_1 = -h_2h_3, \\
\dh_2 = h_3, \\
\dh_3 = h_1 h_3, \\
\dq = X_1 + h_2 X_2.
\end{cases}
\ee
This system is analytic, so we obtain the following statement (taking into account analyticity in the abnormal case, see Subsec.~\ref{subsec:abnorm}).

\begin{theorem}
\label{th:opt_analyt}
All extremal (in particular, optimal) controls and trajectories in Euler's problem are real-analytic.
\end{theorem}

Notice that the vertical subsystem of the Hamiltonian system~\eq{dh1-h2h3} admits a particular solution $(h_1, h_2, h_3) \equiv (0,0,0)$ with the corresponding normal control $u = h_2 \equiv 0$. Thus abnormal extremal trajectories are simultaneously normal, i.e., they are not strictly abnormal.

We define the \ddef{exponential mapping} for the problem:
$$
\map{\Exp_{t_1}}{T_{q_0}^*M}{M},
\qquad
\Exp_{t_1}(\lam_0) = \pi \circ e^{t_1 \vH}(\lam_0) = q(t_1).
$$

The vertical subsystem of system~\eq{dh1-h2h3} has an obvious integral:
$$
h_1^2 + h_3^2 \equiv r^2 = \const \geq 0,
$$
and it is natural to introduce the polar coordinates
$$
h_1 = r \cos \a, \qquad h_3 = r \sin \a.
$$
Then the normal Hamiltonian system~\eq{dh1-h2h3} takes the following form:
\be{dadrdh2}
\begin{cases}
\dot \a = h_2, \\
\dh_2 = r \sin \a, \\
\dot r = 0, \\
\dx = \cos \t, \\
\dy = \sin \t, \\
\dth = h_2.
\end{cases}
\ee

The vertical subsystem of the Hamiltonian system~\eq{dadrdh2} reduces to the 
equation
$$
\ddot \a = r \sin \a.
$$
In the coordinates
$$
c = h_2, \qquad \b = \a + \pi
$$
we obtain the \ddef{equation of pendulum}
$$
\begin{cases}
\dot \b = c, \\
\dc = - r \sin \b
\end{cases}
$$
known as \ddef{Kirchhoff's kinetic analogue} of elasticae. Notice the physical meaning of the constant:
\be{rgl}
r = \frac{g}{L},
\ee
where $g$ is the gravitational acceleration and $L$ is the length of the suspension of the pendulum, see Fig.~\ref{fig:pendulum_pic}.

\onefiglabelsize
{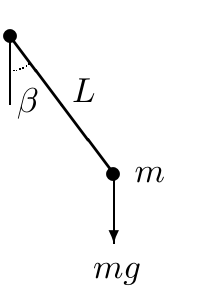}{Pendulum}{fig:pendulum_pic}{0.17}

The total energy of the pendulum is
\be{E_r}
E = \frac{c^2}{2} - r \cos \b \in [-r, + \infty),
\ee
this is just the Hamiltonian $H$.

The geometry of solutions to the vertical subsystem of the normal Hamiltonian system~\eq{dh1-h2h3} 
\be{vert_pend}
\begin{cases}
\dh_1 = -h_2h_3, \\
\dh_2 = h_3, \\
\dh_3 = h_1 h_3
\end{cases}
\iff
\begin{cases}
\dot \b = c, \\
\dc = - r \sin \b, \\
\dot r = 0
\end{cases}
\ee
is visualized by intersections of level surfaces of the integrals
$$
H = h_1 + \frac 12 h_2^2, \qquad r^2 = h_1^2 + h_3^2.
$$
Depending upon the structure of intersection of circular cylinders  $r^2 = \const$ with parabolic cylinders $H = \const$, the following cases are possible, see Figs.~\ref{fig:Hra}--\ref{fig:Hrf}:
\begin{itemize}
\item[(a)]
$H = - r$, $r > 0 \then$ pendulum stays at the stable equilibrium $(\b,c) = (0, 0)$,
\item[(b)]
$H \in(-r,r)$, $r > 0 \then$ pendulum oscillates between extremal values of its angle $\b$,
\item[(c)]
$H =  r$, $r > 0\then$ pendulum either stays at the unstable equilibrium $(\b,c) = (\pi, 0)$ or tends to it for infinite time,
\item[(d)]
$H > r$, $r > 0 \then$ pendulum rotates non-uniformly counterclockwise ($h_2 > 0$) or clockwise ($h_2 < 0$),
\item[(e)]
$H = r = 0 \then$ pendulum is immovable in the absence of gravity, with zero angular velocity $h_2$,
\item[(f)]
$H > r =  0 \then$ pendulum rotates uniformly in the absence of gravity $h_2$ counterclockwise ($h_2 > 0$) or clockwise ($h_2 < 0$).
\end{itemize}

\twofiglabel
{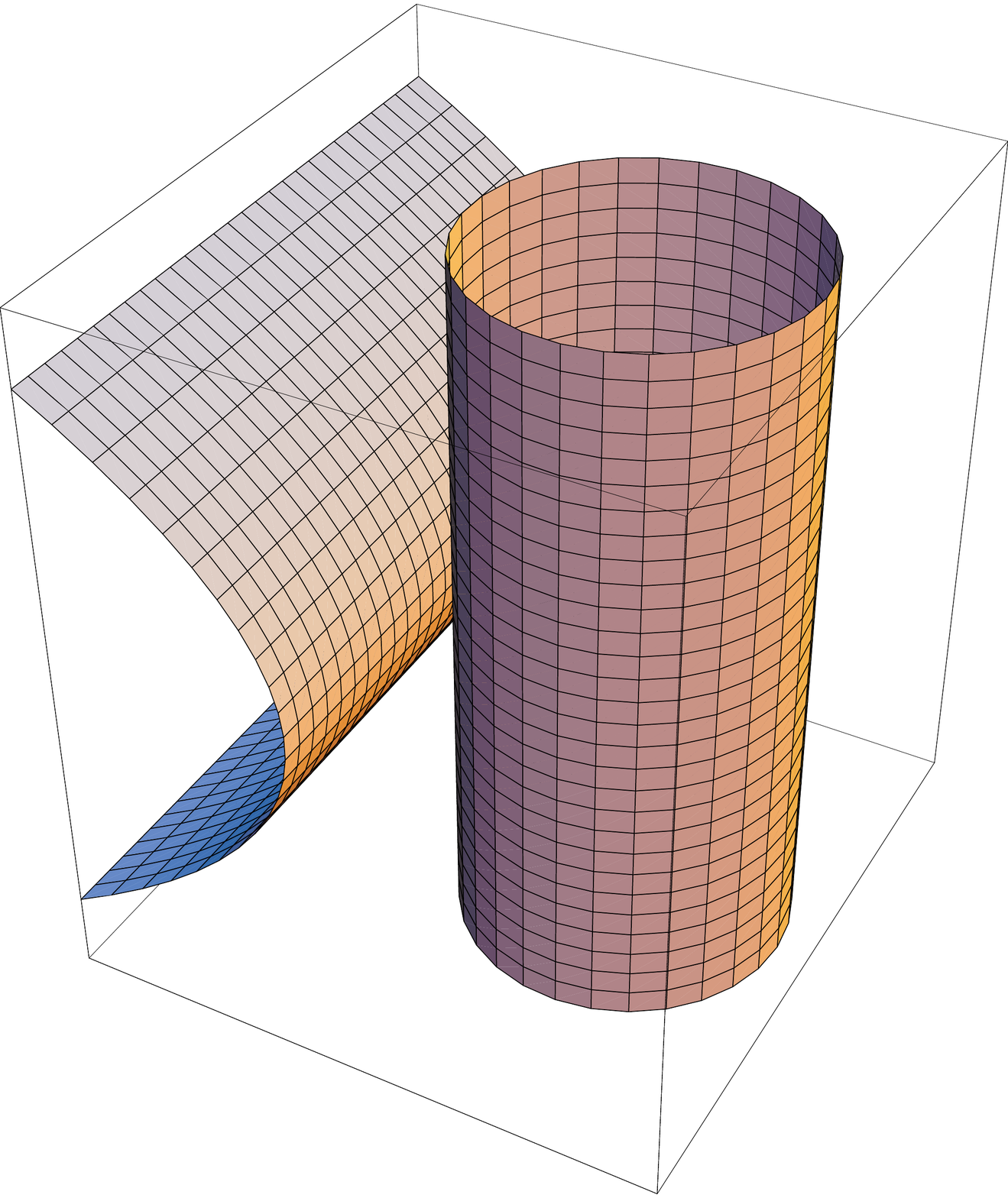}{$H = - r$ , $r > 0$}{fig:Hra}
{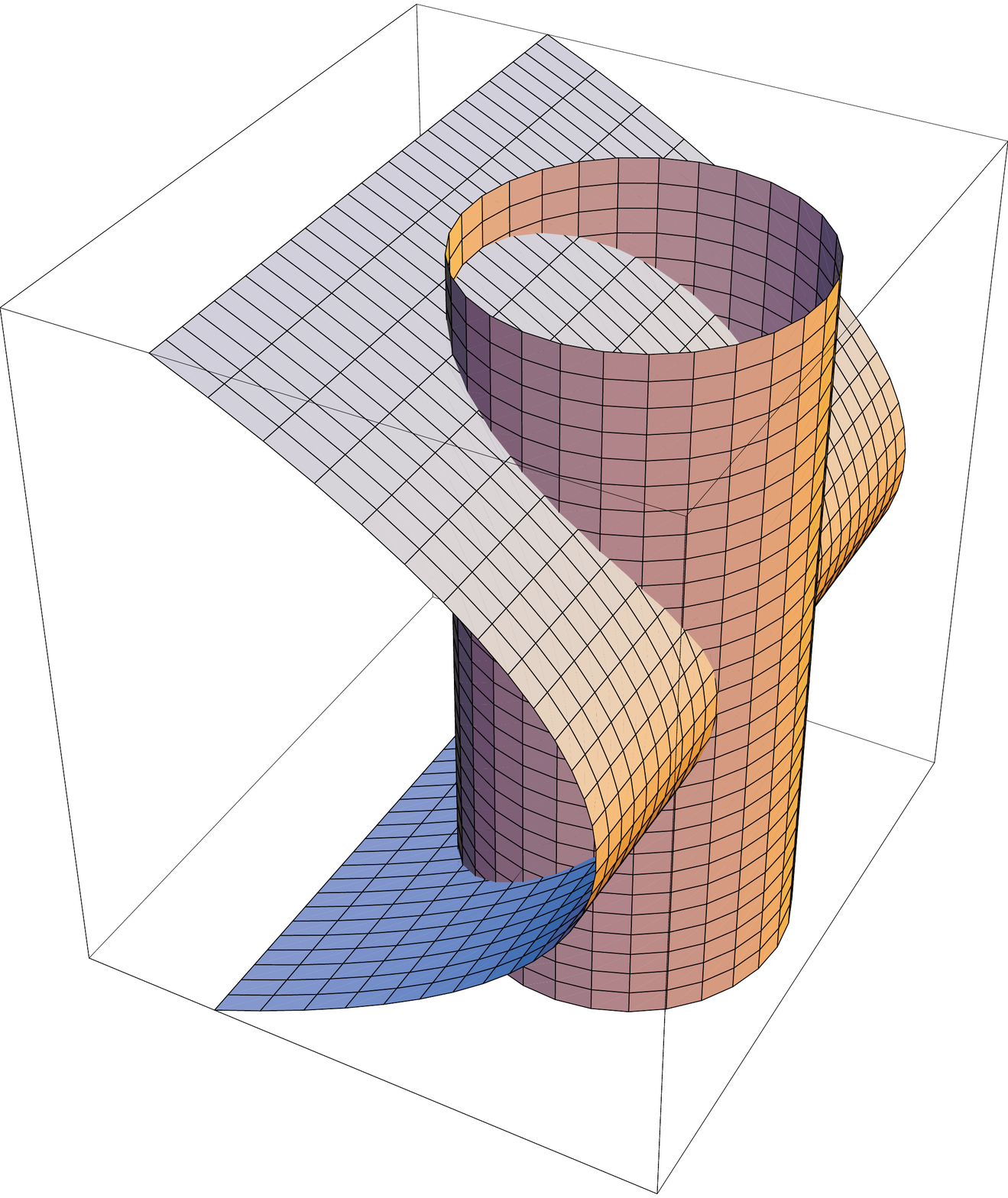}{$H \in(-r,r)$, $r > 0$}{fig:Hrb}

\twofiglabel
{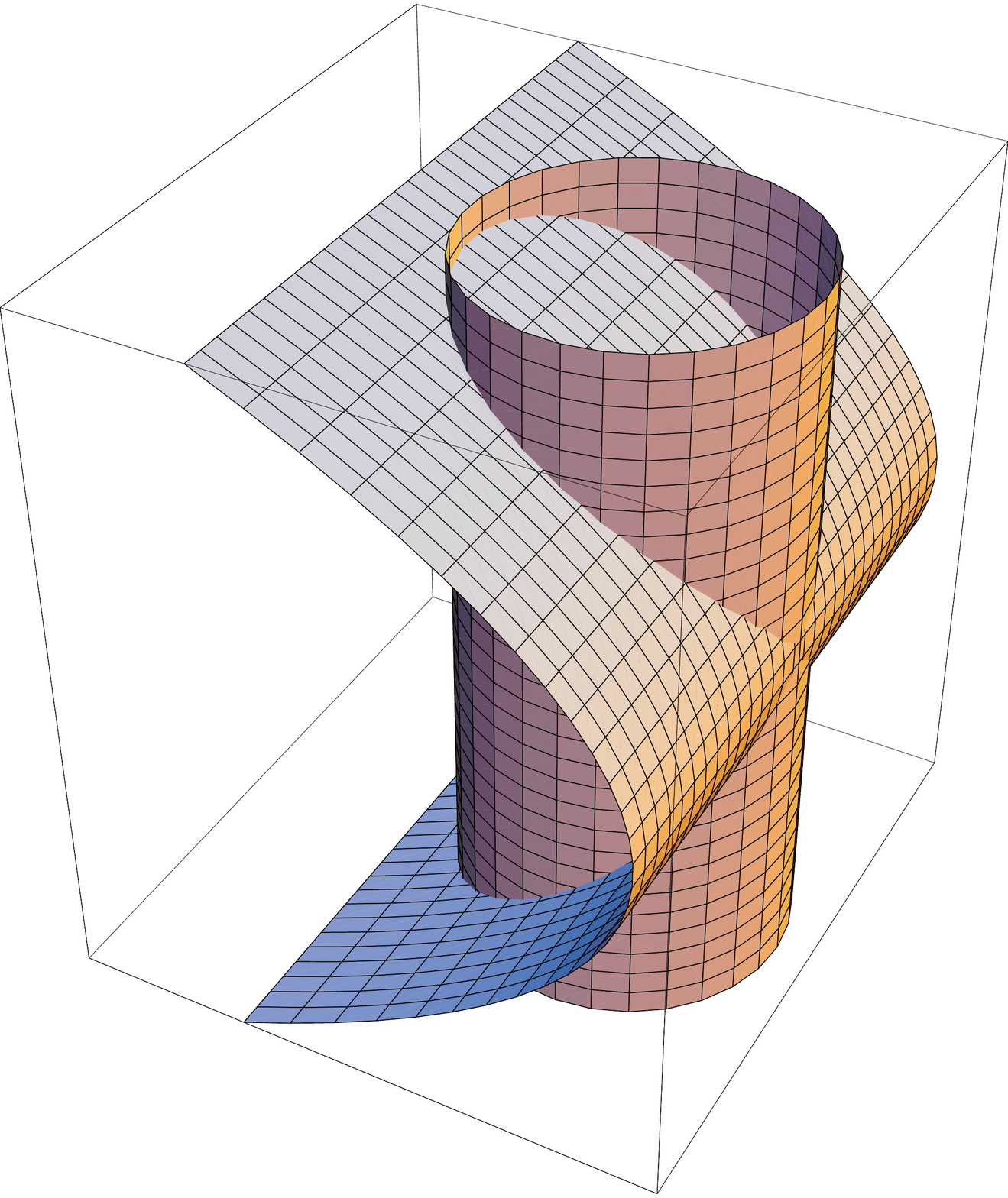}{$H =  r > 0$}{fig:Hrc}
{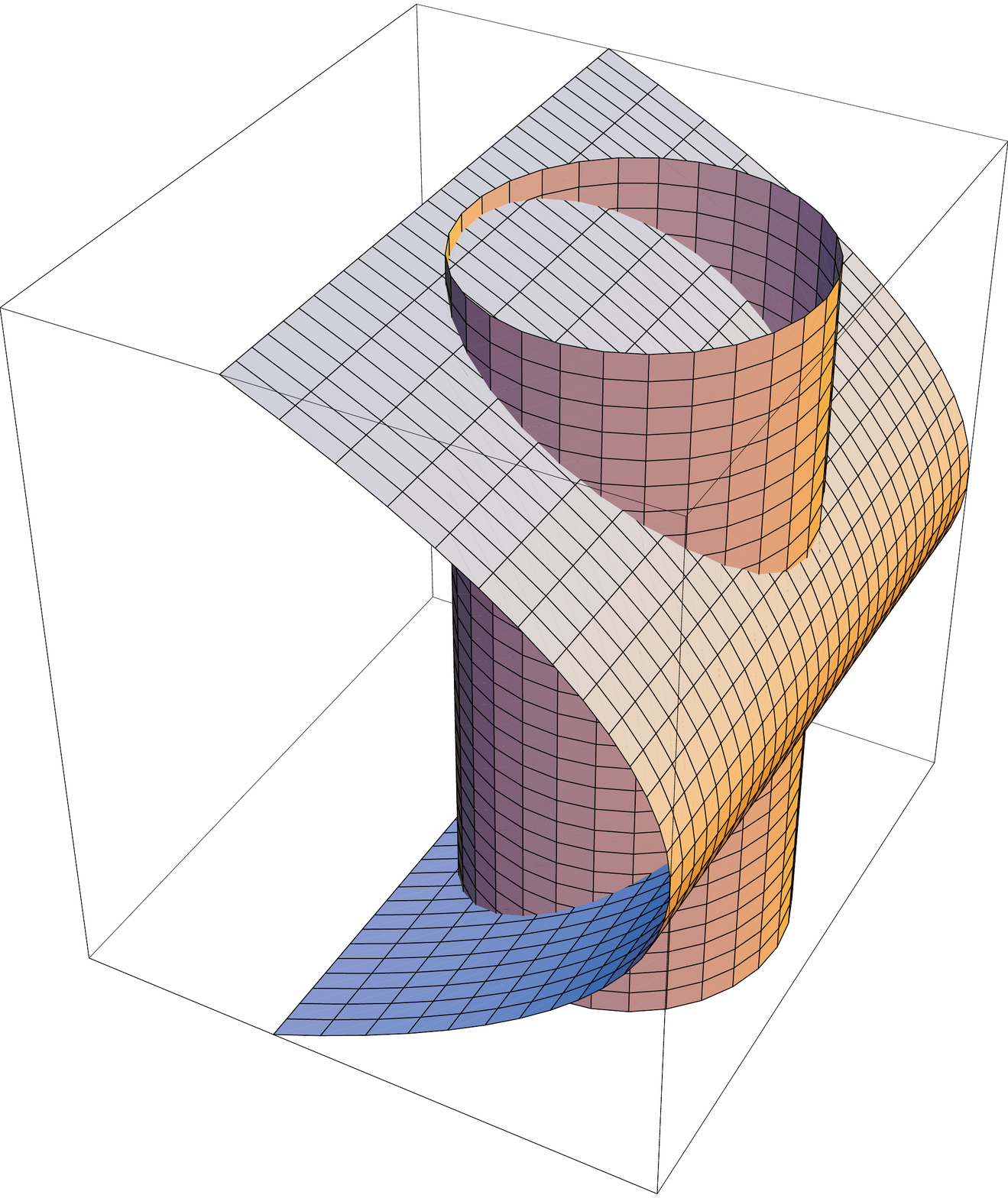}{$H > r > 0$}{fig:Hrd}

\twofiglabel
{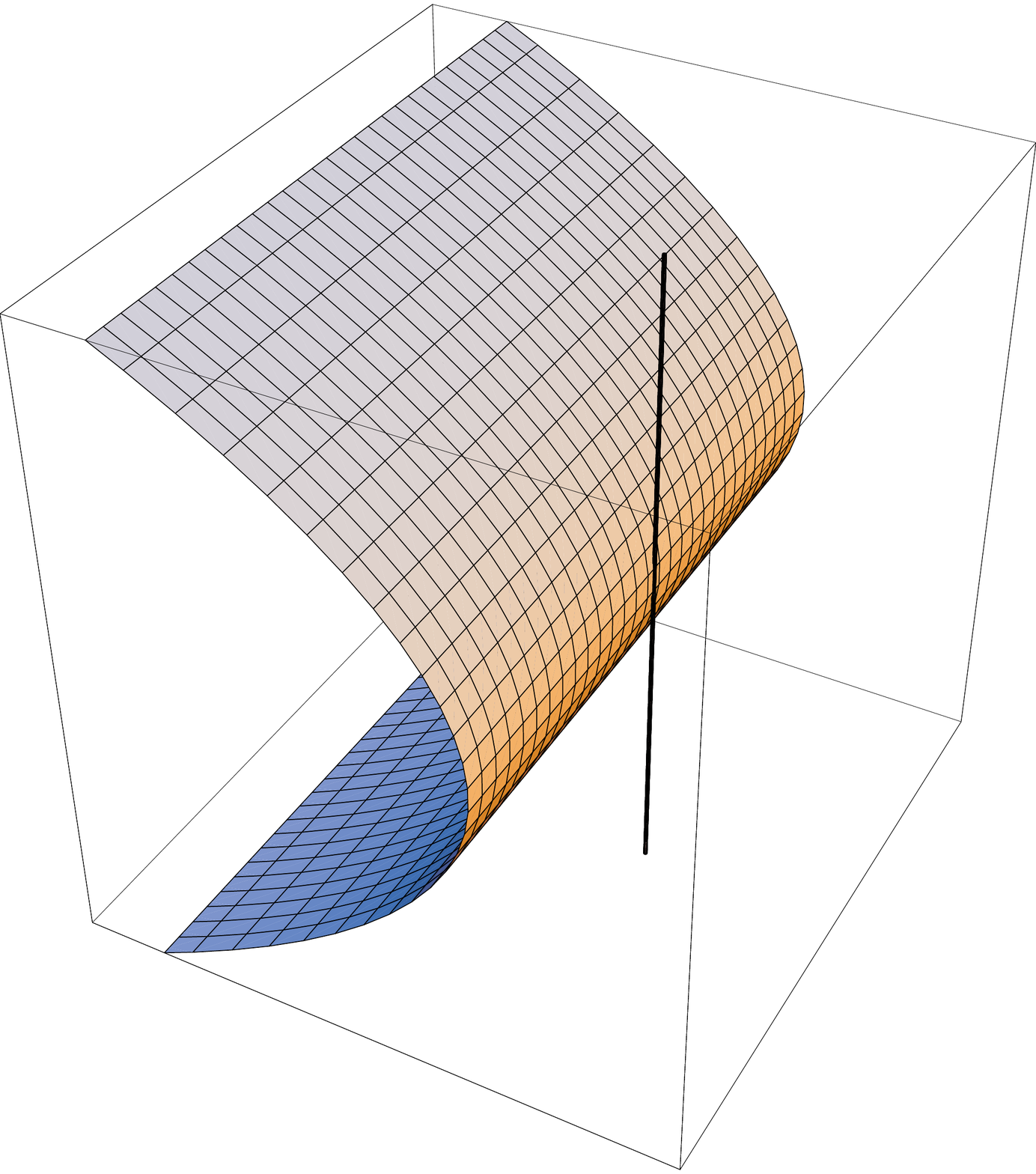}{$H = r = 0$}{fig:Hre}
{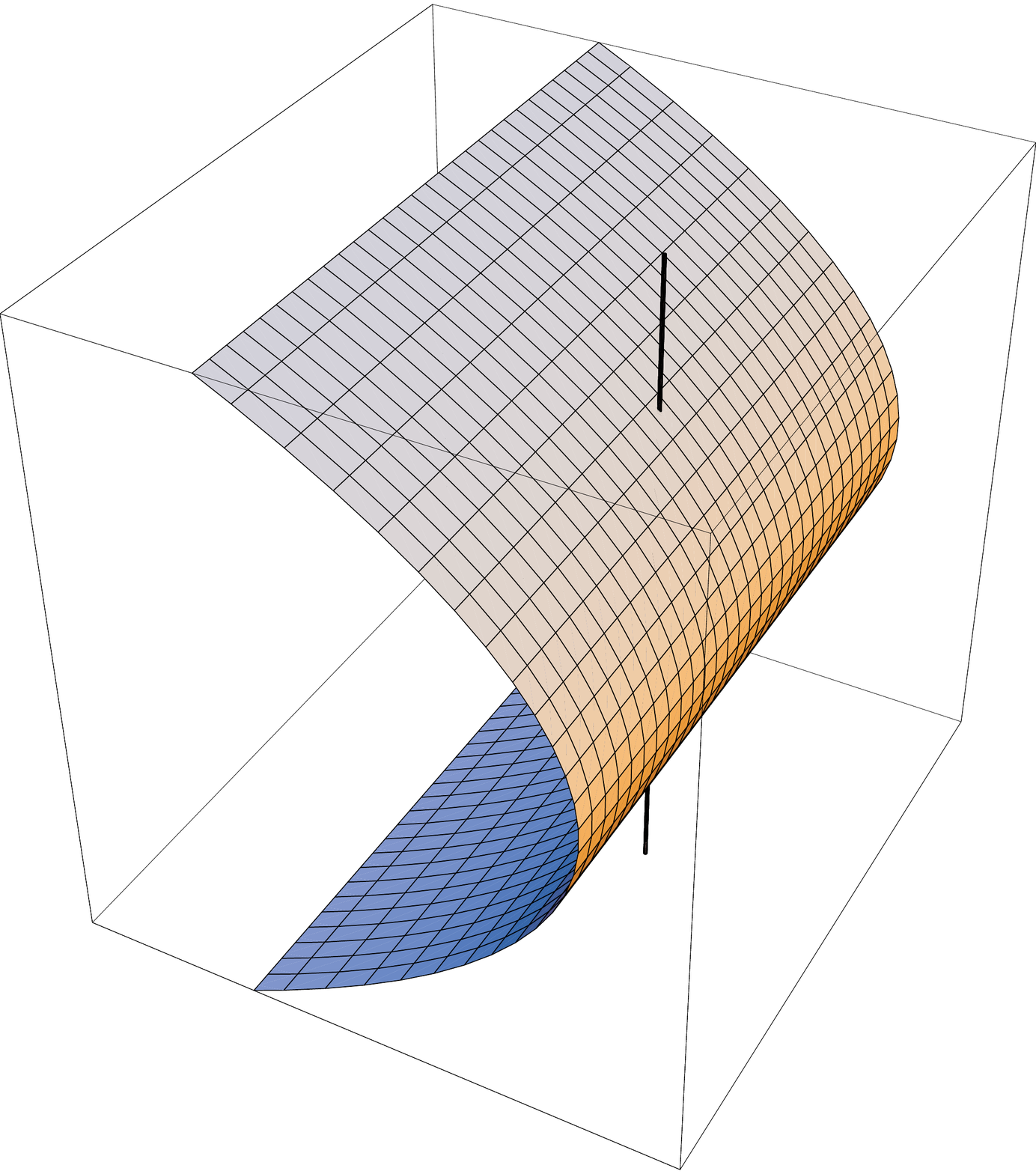}{$H > r = 0$}{fig:Hrf}

The equation of mathematical pendulum is integrable in elliptic functions. Consequently, the whole Hamiltonian system~\eq{dadrdh2} is integrable in quadratures (one can integrate first the vertical subsystem, then the equation for $\t$, and then the equations for $x$, $y$). In Sec.~\ref{sec:integration} we find explicit parametrization of the normal extremals by Jacobi's elliptic functions in terms of natural coordinates in the phase space of pendulum~\eq{vert_pend}.

First we apply continuous symmetries of the problem.
The normal Hamiltonian vector field reads
\begin{align*}
\vH &= -h_2 h_3 \pder{}{h_1} + h_3 \pder{}{h_2} + h_1h_3 \pder{}{h_3} + \cos \t \pder{}{x} + \sin \t \pder{}{y} + h_2 \pder{}{\t}  \\
&= h_2 \pder{}{\a} + r \sin \a \pder{}{h_2} + \cos \t \pder{}{x} + \sin \t \pder{}{y} + h_2 \pder{}{\t}.
\end{align*}
The Hamiltonian system~\eq{dadrdh2} is preserved by the one-parameter group of transformations
\be{arh2}
(\a, r, h_2, x, y, \t, t) \mapsto (\a, r e^{-2s}, h_2 e^{-s}, xe^s, ye^s, \t, t e^s)
\ee
obtained by continuation to the vertical coordinates of the group of dilations of the plane $\R^2_{x,y}$~\eq{group1}.

The one-parameter group~\eq{arh2} is generated by the vector field
$$
Z = -2r \pder{}{r} - h_2 \pder{}{h_2} + x \pder{}{x} + y \pder{}{y}.
$$
We have the Lie bracket and Lie derivatives
\begin{align}
&[Z,\vH] = - \vH, \label{ZH=} \\
&Zr = -2r, \quad Zh_2 = -h_2, \quad \vH r = 0, \quad \vH h_2 = r \sin \a. \label{Zr=}
\end{align}
The infinitesimal symmetry $Z$ of the Hamiltonian field $\vH$ integrates to the symmetry at the level of flows:
$$
e^{t' \vH} \circ e^{s Z}(\lam) =  e^{s Z} \circ   e^{t \vH} (\lam), \qquad t' = e^s t,
\qquad \lam \in T^*M.
$$
The following decomposition of the preimage of the exponential mapping $N$ into invariant subsets of the fields $\vH$ and $Z$ will be very important in the sequel:
\begin{align}
&T_{q_0}^*M = N = \bigcup_{i=1}^7 N_i, \label{N_decomp} \\
&N_1 = \{ \lam \in N \mid r \neq 0, \ E \in (-r, r)\}, \label{N1} \\
&N_2 = \{ \lam \in N \mid r \neq 0, \ E \in (r, + \infty)\} = N_2^+ \cup N_2^-, \label{N2} \\
&N_3 = \{ \lam \in N \mid r \neq 0, \ E  = r, \ \b \neq \pi \} = N_3^+ \cup N_3^-, \label{N3}  \\
&N_4 = \{ \lam \in N \mid r \neq 0, \ E =-r \}, \label{N4} \\
&N_5 = \{ \lam \in N \mid r \neq 0, \ E =r, \ \b = \pi \}, \label{N5} \\
&N_6 = \{ \lam \in N \mid r = 0, \ c \neq 0 \} = N_6^+ \cup N_6^-,\label{N6} \\
&N_7 = \{ \lam \in N \mid r = c =  0\}, \label{N7} \\
&N_i^{\pm} = N_i \cup \{ \lam \in N \mid \sgn c = \pm 1 \}, \qquad i = 2, \ 3, \ 6. \label{Ni+-}
\end{align}

Any cylinder $\{\lam \in N \mid r = \const \neq 0 \}$ can be transformed to the cylinder $C = \{\lam \in N \mid r = 1 \}$ by dilation $Z$; the corresponding decomposition of the phase space of the 
\ddef{standard pendulum}
$$
\begin{cases}
\dot \b = c, \\
\dc = -\sin \b,
\end{cases}
\quad (\b, c) \in C =  S^1_{\b} \times  \R_c,
$$
is shown at Fig.~\ref{fig:decompos_pend}, where
$$
C_i = N_i \cap \{ \, r = 1 \,\}, \qquad i = 1, \dots, 5.
$$

\twofiglabel
{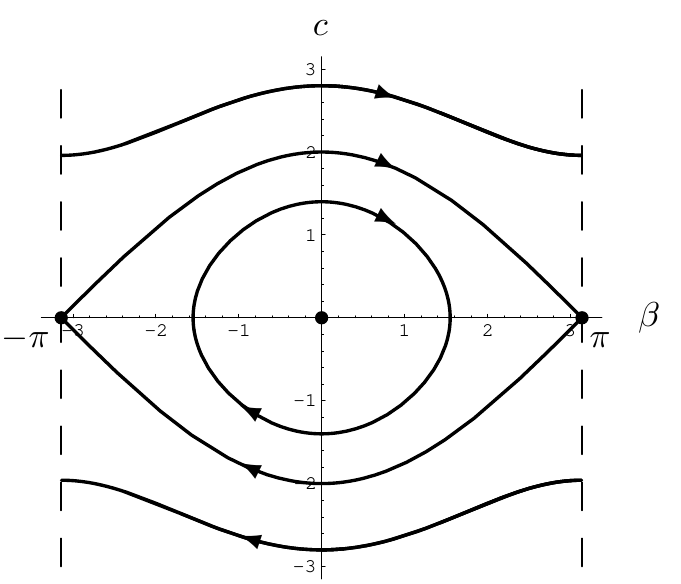}{Phase portrait of pendulum}{fig:pendulum}
{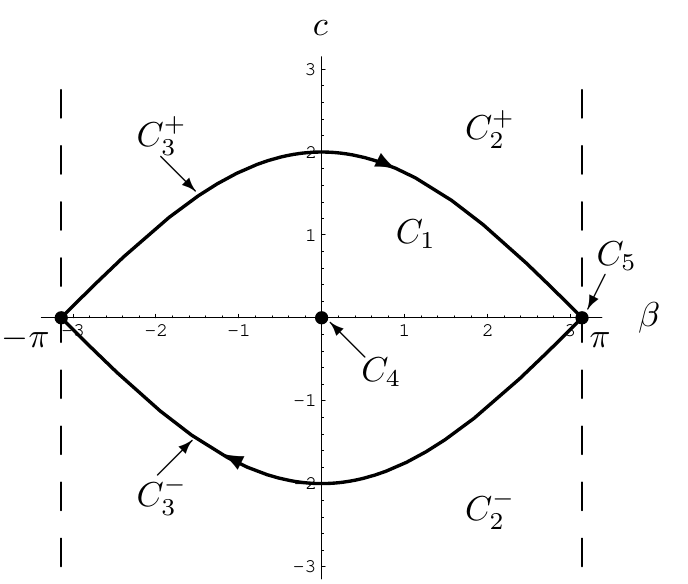}{Decomposition of the phase cylinder of pendulum}{fig:decompos_pend}

In order to integrate the normal Hamiltonian system
\be{Ham_betacr}
\dlam = \vH(\lam) \iff
\begin{cases}
\dot \b = c, \\
\dc = -r \sin \b, \\
\dot r = 0, \\
\dx = \cos \t, \\
\dy = \sin \t, \\
\dth = c,
\end{cases}
\ee
we consider natural coordinates   in the phase space of the pendulum.

\section{Elliptic coordinates}
\label{sec:ell_coords}

The twenty-sixth of the ``Lectures on dynamics'' by 
C.Jacobi~\cite{Jacobi}  is called ``Elliptic coordinates'' and begins by the well-known words: ``The main problem in the integration of these differential equations is the introduction of  convenient variables, there being no  general rule for finding them. Therefore one has to adopt the opposite approach and, finding a remarkable substitution, to seek the problems for which this substitution can be successfully used''. Note that the coordinates introduced below are unrelated to  Jacobi's elliptic coordinates. Moreover, our procedure was opposite to that described by Jacobi: we introduced our elliptic coordinates specifically of parameterizing  extremals and finding Maxwell points in generalized Dido's problem~\cite{dido_exp, max1, max2, max3} and in Euler's problem. Elliptic coordinates lift the veil of complexity over the problems governed by the pendulum equation and open their solution to our eyes (see Fig.~\ref{fig:ell_coordsC}). Here we have an important intersection point with Jacobi: our coordinates are introduced by using Jacobi's elliptic functions, see Sec.~\ref{sec:append} and~\cite{lawden}, \cite{whit_watson}. Another important moment will be the study of conjugate points, that is, solutions to Jacobi equation, along extremals~\cite{elastica_conj}.

\onefiglabelsize{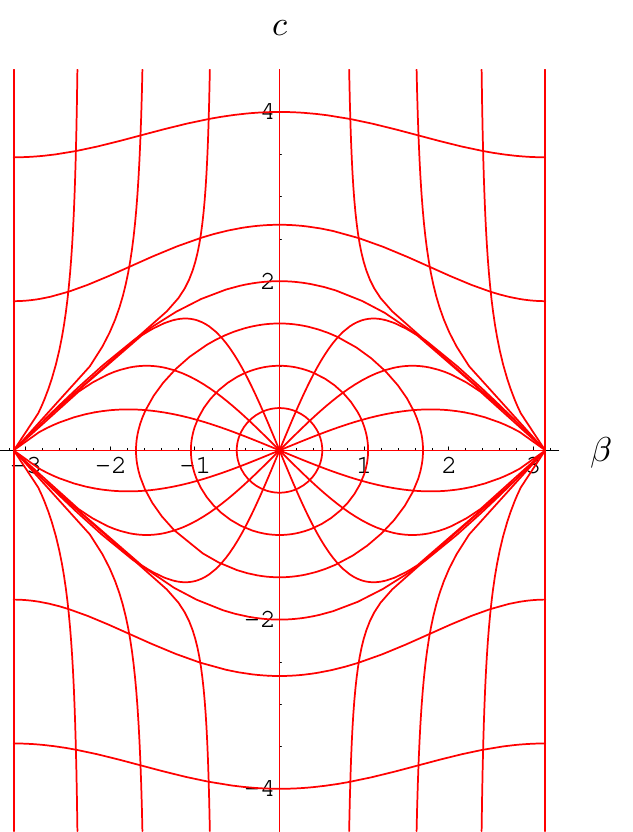}{Elliptic coordinates in the phase cylinder of pendulum}{fig:ell_coordsC}{0.6}

\subsection{Time of motion of the pendulum}
\ddef{Elliptic coordinates} in the phase cylinder of the standard pendulum
\be{pend_st}
\begin{cases}
\dot \beta = c, \\
\dc = - \sin \b,
\end{cases}
\qquad 
(\b, c) \in C = S^1_{\b} \times \R_c,
\ee
were introduced in~\cite{dido_exp} for integration and study of the nilpotent sub-Riemannian problem with the growth vector (2,3,5). Here we propose a more natural and efficient construction of these coordinates.

Denote
$$
P = \R_{+ c} \times \R_t, \qquad \hC = C_1 \cup C_2^+ \cup C_3^+,
$$
and consider the mapping
$$
\map{\Phi}{P}{\hC}, \qquad \Phi \ : \ (c,t) \mapsto (\b_t, c_t),
$$
where $(\b_t, c_t)$ is the solution to the equation of pendulum~\eq{pend_st} with the initial condition
\be{beta0c0}
\b_0 = 0, \qquad c_0 = c.
\ee


The mapping $\map{\Phi}{P}{\hC}$ is real-analytic since the equation of pendulum~\eq{pend_st} is a real-analytic ODE.

First we show that $\Phi$ is a local diffeomorphism, i.e., the Jacobian
$$
\pder{(\b_t, c_t)}{(c,t)}
= \left|
\begin{array}{cc}
\ds\pder{\b_t}{c} & \ds\pder{\b_t}{t} \\
\ds\pder{c_t}{c} & \ds\pder{c_t}{t}
\end{array}
\right|
\neq 0 \qquad \forall (c,t) \in P.
$$
By virtue of system~\eq{pend_st}, we have
$$
\pder{\b_t}{t} = c_t, \qquad \pder{c_t}{t} = -\sin \b_t.
$$
Further, denote
$$
\pder{\b_t}{c} = \zeta(t), \qquad \pder{c_t}{c} = \eta(t).
$$
Since
\begin{align*}
&\pder{}{t}\left(\pder{\b_t}{c}\right) =
\pder{}{c} \pder{\b_t}{t} = \pder{c_t}{c} = \eta(t),
\\
&\pder{}{t} \left( \pder{c_t}{c}\right) = \pder{}{c} \pder{c_t}{t} = \pder{}{c} (-\sin \b_t) = - \cos \b_t\, \zeta(t),
\end{align*}
the pair $(\zeta(t),\eta(t))$ is the solution to the Cauchy problem
\begin{align*}
&\dot\zeta = \eta, &&\zeta(0) = \pder{\b_0}{c} = 0, \\
&\dot\eta = -\cos \b_t \,\zeta, &&\eta(0) = \pder{c_0}{c} = 1.
\end{align*}
Now consider the determinant
\be{d(t)1}
d(t) = \pder{(\b_t,c_t)}{(c,t)} =
\left|
\begin{array}{cc}
\ds \zeta(t)  & \ds c_t \\
\ds \eta(t) & \ds -\sin \b_t
\end{array}
\right|.
\ee
Differentiating by rows, we obtain
$$
\dot d(t) =
\left|
\begin{array}{cc}
\ds \eta(t)  & \ds -\sin \b_t \\
\ds \eta(t) & \ds -\sin \b_t
\end{array}
\right|
+
\left|
\begin{array}{cc}
\ds \zeta(t)  & \ds c_t \\
\ds -\cos \b_t \, \zeta(t) & \ds -\cos \b_t \, \zeta(t)
\end{array}
\right|
= 0,
$$
thus
\be{d(t)2}
d(t) \equiv d(0) =
\left|
\begin{array}{cc}
\ds 0  & \ds c \\
\ds 1 & \ds 0
\end{array}
\right|
= -c \neq 0 \quad \forall \ (c,t) \in P.
\ee

Denote by
\begin{align*}
&k_1 = \sqrt{\frac{E+1}{2}} = \sqrt{\sin^2 \frac{\b}{2} + \frac{c^2}{4}} \in (0,1), \qquad (\b, c ) \in C_1, \\
&k_2 = \sqrt{\frac{2}{E+1}} = \frac{1}{\sqrt{\sin^2 \frac{\b}{2} + \frac{c^2}{4}}} \in (0,1), \qquad (\b, c ) \in C_2, 
\end{align*}
a reparametrized energy ($E = \dfrac{c^2}{2} - \cos \b$) of the standard pendulum; below $k_1$, $k_2$  will play the role of the modulus for Jacobi's elliptic functions, and
$$
K(k) = \int_0^{\pi/2} \frac{dt}{\sqrt{1 - k^2 \sin^2 t}}, \qquad k \in (0,1),
$$
the complete elliptic integral of the first kind, see Sec.~\ref{sec:append} and~\cite{lawden}. It is well known~\cite{lawden} that the standard pendulum~\eq{pend_st} has the following period of motion $T$ depending on its energy $E$:
\begin{align}
&-1< E < 1 \iff (\b,c) \in C_1 \then T = 4 K(k_1), \label{period1}\\
&E = 1, \ \b \neq \pi \iff (\b,c) \in C_3 \then T = \infty, \label{period2}\\
&E > 1 \iff (\b,c) \in C_2 \then T = 2 K(k_2)k_2. \label{period3}
\end{align}
Introduce the equivalence relation $\sim$  in the domain $P$ as follows. For $(c_1, t_1)\in P$, $(c_2, t_2) \in P$, we set $(c_1,t_1) \sim (c_2, t_2)$ iff $c_1 = c_2 = c$ and
\begin{align*}
&t_2 = t_1 \pmod{4 K(k_1)}, \qquad k_1 = \frac{c}{2} && \text{ for } c \in (0,2), \\
&t_2 = t_1  && \text{ for } c =2, \\
&t_2 = t_1 \pmod{2 K(k_2) k_2}, \qquad k_2 = \frac{2}{c} && \text{ for } c \in (2,+ \infty).
\end{align*}
That is, we identify the points $(c,t_1)$,  $(c,t_2)$ iff the corresponding solutions to the equation of pendulum~\eq{pend_st} with the initial condition~\eq{beta0c0} give the same point $(\b_t, c_t)$ in the phase cylinder of the pendulum $S^1_{\b} \times \R_c$ at the instants $t_1$, $t_2$.

Denote the quotient $\tP = P/\sim$. In view of the periodicity properties~\eq{period1}--\eq{period3} 
of the pendulum~\eq{pend_st}, the mapping
$$
\map{\Phi}{\tP}{\hC}, \qquad \Phi(c,t) = (\b_t, c_t),
$$
is a global analytic diffeomorphism. Thus there exists the inverse mapping, an analytic diffeomorphism
\begin{align}
&\map{F}{\hC}{\tP}, \nonumber \\
&F(\b,c) = (c_0, \f), \label{Fbetac}
\end{align}
where $\f$ is the time of motion of the pendulum in the reverse time from a point $(\b,c) \in \hC$ to the semi-axis $\{\b= 0, \ c = 0\}$. In the domains $C_1$ and $C_2^+$, the time $\f$ is defined modulo the period of the pendulum $4K(k_1)$ and $2 K(k_2)k_2$  respectively.

 We summarize the above construction in the following proposition.

\begin{theorem}
\label{th:pend_time}
There is an analytic multi-valued function
$$
\map{\f}{\hC = C_1 \cup C_2^+ \cup C_3^+}{\R}
$$
such that for $\b_0 = 0$, $c_0 > 0$ and the corresponding solution $(\b_t, c_t)$ of the Cauchy problem~\eq{pend_st}, \eq{beta0c0}, there holds the equality
$$
\f(\b_t, c_t) =
\begin{cases}
t \pmod{4 K(k_1)} &\text{ for } (\b_t, c_t) \in C_1, \\
t  &\text{ for } (\b_t, c_t) \in C_2^+, \\
t \pmod{2 K(k_2) k_2} &\text{ for } (\b_t, c_t) \in C_3^+.
\end{cases}
$$
\end{theorem}

In other words, $\f(\b_t, c_t)$ is
 the time of motion of the pendulum in the reverse time from the point $(\b_t, c_t) \in \hC$ to the semi-axis $\{\b=0, \ c > 0\}$.

\subsection{Elliptic coordinates in the phase space of pendulum}
\label{subsec:ell_coordsC123}

In the domain  $C_1\cup C_2\cup C_3$, we introduce \ddef{elliptic coordinates} $(\f,k)$, where $\f$ is the time of motion of the pendulum from the semi-axis $\{\b=0, \ c > 0\}$ (in the domain $\hC = C_1\cup C_2^+\cup C_3^+$) or from the   semi-axis $\{\b=0, \ c < 0\}$ (in the domain $\tC =  C_2^-\cup C_3^-$), and $k \in (0,1)$ is a reparametrized energy of the pendulum --- the modulus of Jacobi's elliptic functions.

\subsubsection{Elliptic coordinates in $C_1$}
If $(\b,c) \in C_1$, then we set
\begin{align}
&\begin{cases}
\ds \sin \frac{\b}{2} = k_1 \sn(\f, k_1), \\
\ds \frac{c}{2} = k_1 \cn(\f, k_1), \\
\ds \cos \frac{\b}{2} = \dn(\f, k_1),
\label{ellC1}
\end{cases} \\
&k_1 = \sqrt{\frac{E+1}{2}} = \sqrt{\sin^2\frac{\b}{2} + \frac{c^2}{4}} \in (0,1),
\label{k1sqrt} \\
&\f \pmod{4 K(k_1)} \in [0, 4 K(k_1)]. \nonumber
\end{align}
Here and below $\cn$, $\sn$, $\dn$ are Jacobi's elliptic functions, see Sec.~\ref{sec:append} and~\cite{lawden}.

The function $\f$ thus defined is indeed the time of motion of the pendulum from the semi-axis $\{\b = 0, \ c > 0\}$ in view of the following:
\begin{align}
&(\b=0, \ c> 0) \then \f = 0, \label{check1}\\
&\restr{\der{\f}{t}}{\text{Eq. }(\ref{pend_st})} = 1,  \label{check2}
\end{align}
the total derivative w.r.t. the equation of pendulum~\eq{pend_st}.

The mapping $(\b,c) \mapsto (k,\f)$ is an analytic diffeomorphism since it decomposes into the chain of analytic diffeomorphisms:
$$
(\b,c) \stackrel{(a)}{\mapsto} (c_0, \f) \stackrel{(b)}{\mapsto} (k_1, \f),
$$
where $(a)$ is defined by $F$~\eq{Fbetac}, while $(b)$ is given by
$$
k_1 = \sqrt{\frac{E+1}{2}} = \frac{c_0}{2},
$$
compare with~\eq{k1sqrt}.

\subsubsection{Elliptic coordinates in $C_2^+$}
\label{subsubsec:ell_coordsC2+}
Let $(\b,c) \in C_2^+$. Elliptic coordinates $(\f, k_1)$ in the domain $C_2^+$ are analytic functions $(\f, k_1)$ defined as follows: $\f$ is the time of motion of the pendulum from the semi-axis $\{\b = 0, \ c > 0\}$, and $k_1 = \frac{E+1}{2}$. By the uniqueness theorem for analytic functions, in the domain $C_2^+$ we have the same formulas as in $C_1$:
\begin{align}
&\sin \frac{\b}{2} = k_1 \sn(\f, k_1), \label{sinbeta2} \\
&\frac{c}{2} = k_1 \cn(\f, k_1), \label{c2}
\\
&\cos \frac{\b}{2} = \dn(\f, k_1), \label{cosbeta2} \\
&k_1 = \sqrt{\frac{E+1}{2}}  \in (1,+ \infty). \nonumber
\end{align}
Here Jacobi's elliptic functions $\sn(u,k_1)$, $\cn(u,k_1)$, $\dn(u,k_1)$ for the modulus $k_1>1$ are obtained from those defined in~\eq{cn}--\eq{dn} by the analytic continuation along the complex modulus $k_1 \in \C$ through the complex plane around the singularity $k_1 = 1$, see Sec.~3.9 and Sec.~8.14~\cite{lawden}. In order to obtain Jacobi's functions with the modulus in the interval $(0,1)$, we apply the transformation of modulus $k \mapsto \frac{1}{k}$, see formulas~\eq{k->1/k1}, \eq{k->1/k2} in Sec.~\ref{sec:append}. Transforming equalities~\eq{sinbeta2}--\eq{cosbeta2} via formulas~\eq{k->1/k1}, \eq{k->1/k2}, we obtain the following expressions for elliptic coordinates $(\f, k_2)$:
\begin{align}
&\begin{cases}
\ds\sin \frac{\b}{2} = \sn\left(\frac{\f}{k_2}, k_2\right), \\
\ds\frac{c}{2} = \frac{1}{k_2} \dn\left(\frac{\f}{k_2}, k_2\right), \\
\ds\cos \frac{\b}{2} = \cn\left(\frac{\f}{k_2}, k_2\right),
\label{ellC2+}
\end{cases} \\
&k_2 =  \frac{1}{k_1} = \sqrt{\frac{2}{E+1}}  \in (0,1).
\nonumber
\end{align}

Certainly, one can verify directly that $\f$ is indeed the time of motion of the standard pendulum from the point $(\b,c)$ to the semi-axis $\{\b = 0, = c > 0 \}$ in the reverse time by checking the conditions~\eq{check1}, \eq{check2} in the domain $C_2^+$, but our idea is to obtain ``for free'' equalities in $C_2^+$ from equalities in $C_1$ via the transformation of the modulus $k \mapsto \frac 1k$.

\subsubsection{Elliptic coordinates in $C_3^+$}
Let $(\b,c) \in C_3^+$. Elliptic coordinated on the set $C_3^+$ are given by $(\f, k = 1)$, where $\f$ is the time of motion of the pendulum from the semi-axis $\{\b=0, \ c > 0\}$, and $k = \sqrt{\frac{E+1}{2}} = 1$. The analytic expressions for $\f$ are obtained by passing to the limit $k_1 \to 1 -0$ in formulas~\eq{ellC1} or to the limit $k_2 \to 1 -0$ in formulas~\eq{ellC2+}, with the use of formulas of degeneration of elliptic functions~\eq{degen1}. As a result of the both limit passages, we obtain the following expression for the elliptic coordinate $\f$ on the set $C_3^+$:
$$
\begin{cases}
\ds \sin \frac{\b}{2} = \tanh \f, \\
\ds \frac{c}{2} = \frac{1}{\cosh \f}, \\
\ds \cos \frac{\b}{2} = \frac{1}{\cosh \f}.
\end{cases}
$$

\subsubsection{Elliptic coordinates in $C_2^-\cup C_3^-$}
\label{subsubsec:ell_coordsC2-C3-}
For a point $(\b,c) \in \tC = C_2^-\cup C_3^-$, elliptic coordinates $(\f,k)$ cannot be defined in the same way as in $\hC = C_1 \cup C_2^+\cup C_3^+$ since such a point is not attainable along the flow of the pendulum~\eq{pend_st} from the semi-axis $\{\b = 0, \ c > 0\}$, see the phase portrait at Fig.~\ref{fig:pendulum}. Now we take the initial semi-axis   $\{\b = 0, \ c < 0\}$, and define $\f$ in $\tC$ equal to the time of motion of the pendulum from this semi-axis to the current point. That is, for points $(\b, c) \in \tC$ we consider the mapping
\begin{align*}
&F(c,t) = (\b_t, c_t), \qquad c < -2, \\
&\b_0 = 0, \quad c_0 = c,
\end{align*}
and construct the inverse mapping
$$
\Phi(\b,c) = (c_0, \f).
$$

The pendulum~\eq{pend_st} has an obvious symmetry --- reflection in the origin $(\b = 0, \ c = 0)$:
\be{betac-beta-c}
i \ : \ (\b, c) \mapsto (-\b, -c).
\ee
In view of this symmetry, we obtain:
\begin{align*}
&\Phi(\b,c) = (c_0, \f), &&(\b,c) \in C_2^-\cup C_3^-, \\
&\Phi(-\b,-c) = (-c_0, \f), &&(-\b,-c) \in C_2^+\cup C_3^+,
\end{align*}
thus
$$
\f(\b,c) = \f(-\b, -c), \qquad (\b,c) \in C_2^-\cup C_3^-.
$$

On the other hand, the energy of the pendulum $E$ and the modulus of elliptic functions $k_2$ are preserved by the reflection~\eq{betac-beta-c}. So we have the following formulas for elliptic functions in $\tC$.
\begin{align*}
&(\b,c) \in C_2^- \then
\begin{cases}
\ds  \sin \frac{\b}{2} = - \sn\left(\frac{\f}{k_2}, k_2\right), \\
\ds \frac{c}{2} = -\frac{1}{k_2} \dn\left(\frac{\f}{k_2}, k_2\right), \\
\ds \cos \frac{\b}{2} = \cn\left(\frac{\f}{k_2}, k_2\right),
\end{cases}
\\
&(\b,c) \in C_3^- \then
\begin{cases}
\ds \sin \frac{\b}{2} = -\tanh \f, \\
\ds \frac{c}{2} = -\frac{1}{\cosh \f}, \\
\ds \cos \frac{\b}{2} = \frac{1}{\cosh \f}.
\end{cases}
\end{align*}

Summing up, in the domain $C_1\cup C_2 \cup C_3$ the elliptic coordinates $(\f,k)$ are defined as follows:
\begin{align*}
&(\b,c) \in C_1 \then
\begin{cases}
\ds \sin \frac{\b}{2} = k_1 \sn(\f, k_1), \\
\ds \frac{c}{2} = k_1 \cn(\f, k_1), \\
\ds \cos \frac{\b}{2} = \dn(\f, k_1),
\end{cases}
\\
&k_1 = \sqrt{\frac{E+1}{2}} \in (0,1), \qquad \f \pmod{ 4 K(k_1)} \in [0, 4 K(k_1)],
\end{align*}
\begin{align*}
&(\b,c) \in C_2^{\pm} \then
\begin{cases}
\ds \sin \frac{\b}{2} = \pm \sn\left(\frac{\f}{k_2}, k_2\right), \\
\ds \frac{c}{2} = \pm \frac{1}{k_2} \dn\left(\frac{\f}{k_2}, k_2\right), \\
\ds \cos \frac{\b}{2} = \cn\left(\frac{\f}{k_2}, k_2\right),
\end{cases}
\\
&k_2 = \sqrt{\frac{2}{E+1}} \in (0,1), \qquad \f \pmod{ 2 K(k_2)k_2} \in [0, 2 K(k_2)k_2],
\qquad \pm = \sgn c,
\end{align*}
\begin{align*}
&(\b,c) \in C_3^{\pm} \then
\begin{cases}
\ds \sin \frac{\b}{2} = \pm \tanh \f, \\
\ds \frac{c}{2} = \pm \frac{1}{\cosh \f}, \\
\ds \cos \frac{\b}{2} = \frac{1}{\cosh \f},
\end{cases}
\\
&k = 1, \qquad \f \in \R,
\qquad \pm = \sgn c.
\end{align*}

\begin{remark}
In such a definition of elliptic coordinates, the domains $C_2^+$ and $C_2^-$ ($C_3^+$ and $C_3^-$) have different status: the coordinate $\f$ is discontinuous when crossing the separatrix $C_3^-$, and analytic on the separatrix $C_3^+$. This is a consequence of the fact that in $C_1$, $C_2^+$, $C_3^+$ the coordinate $\f$ is defined uniformly --- as the time of motion of the pendulum from the semi-axis $\{\b = 0, \ c > 0\}$, while in $C_2^-$, $C_3^-$ this is the time of motion from another semi-axis $\{\b=0, \ c < 0\}$. This different status is reflected in the fact that elliptic coordinates in $C_2^+$ are obtained from elliptic coordinated in $C_1$ by analytic continuation (with the use of the transformation $k \mapsto \frac{1}{k}$ of Jacobi's functions), after which elliptic coordinates in $C_2^-$ are obtained from $C_2^+$ via the symmetry $i$ of the pendulum~\eq{betac-beta-c}.

The use of analytic continuation from $C_1$ to $C_2^+$ allows us to obtain ``gratis'' all formulas in $C_2^+$ from the corresponding formulas in $C_1$ via the transformation $k \mapsto \frac{1}{k}$~\eq{k->1/k1}, \eq{k->1/k2}. As usual for analytic functions, analytic continuation respects only equalities; inequalities are not continued in such a way, in particular, we will have to obtain bounds for roots of equations independently in $C_1$ and $C_2$. But in order to obtain equalities in the cylinder $C$ (and in the preimage of the exponential mapping $N$) we will make use of the following chain:
\be{chainCi}
\xymatrix{
C_4 & C_1 \ar[l]_{k \to 0} \ar[r]^{k \mapsto \frac{1}{k}} \ar[d]^{k \to 1 -0} & C_2^+
\ar[d]^{i} \ar[r]^{k \to 1 -0}& C_3^+ \ar[d]^{i} \\
    & C_5 & C_2^- & C_3^-
}
\ee
Such a chain will be useful not only in Euler's problem, but in all problems governed by the pendulum~\eq{pend_st}, e.g. in the nilpotent (2,3,5) sub-Riemannian problem~\cite{dido_exp, max1, max2, max3}, in the plate-ball problem~\cite{jurd_book}, in the sub-Riemannian problem on the group of motions of the plane.
\end{remark}

At Fig.~\ref{fig:ell_coordsC} we present the grid of elliptic coordinates in the phase cylinder of the standard pendulum
($\R_{+\, c} \times S^1_{\b}$). In the domain $C_1$ (oscillations of pendulum with low energy  $E < 1$) we plot the curves   $k = \const$, $\f = \const$; in the domain   $C_2$ (rotations of pendulum with high energy   $E>1$) we plot the curves    $k = \const$, $\p = \const$; these domains are separated by the set   $C_2$ (motions of pendulum with critical energy   $E = 1$), consisting of two separatrices  $k = 1$ and the unstable equilibrium.

\subsection[Elliptic coordinates in the preimage of the exponential mapping]
{Elliptic coordinates \\ in the preimage of the exponential mapping}
\label{subsec:ell_coordN123}
In the domain $\hN = N_1 \cup N_2\cup N_3$ (recall decomposition~\eq{N_decomp}--\eq{Ni+-}), the vertical subsystem of the Hamiltonian system~\eq{dh1-h2h3} has the form of the \ddef{generalized pendulum}
\be{pend_r}
\begin{cases}
\dot \b = c, \\
\dc = - r \sin \b, \\
\dot r = 0.
\end{cases}
\ee
Elliptic coordinates in the domain $\hN$ have the form $(\f, k, r)$. On the set $N_1 \cup N_2^+ \cup N_3^+$, the coordinate $\f$ is equal to the time of motion of the generalized pendulum~\eq{pend_r} from a point $(\b = 0, c = c_0 > 0, r)$ to a point $(\b, c, r)$, while on the set $N_2^- \cup N_3^-$ the time of motion is taken from a point $(\b = 0, c = c_0 < 0, r)$.

The one-parameter group of symmetries
$$
(\b, c, r, t) \mapsto (\b, c e^{-s}, r e^{-2s}, t e^s)
$$
of the generalized pendulum~\eq{pend_r} is a restriction of action of the group~\eq{arh2}. We apply this group to transform the generalized pendulum~\eq{pend_r} in the domain $\{r>0\}$ to the standard pendulum~\eq{pend_st} for $r=1$. This transformation preserves the integral of the generalized pendulum
$$
k_1 = \sqrt{\frac{E+r}{2r}} = \sqrt{\sin^2 \frac{\b}{2} + \frac{c^2}{4 r}}.
$$

Thus we obtain the following expressions for elliptic coordinates in the domain $\hN$ from similar expressions in the domain $\hC$, see Subsec.~\ref{subsec:ell_coordsC123}.
\begin{align*}
&\lam = (\b,c, r) \in N_1 \then
\begin{cases}
\ds \sin \frac{\b}{2} = k_1 \sn(\sqrt{r} \f, k_1), \\
\ds \frac{c}{2} = k_1 \sqrt{r} \cn(\sqrt{r} \f, k_1), \\
\ds \cos \frac{\b}{2} = \dn(\sqrt{r} \f, k_1),
\end{cases}
\\
&k_1 = \sqrt{\frac{E+r}{2r}} \in (0,1), \qquad \sqrt{r} \f \pmod{ 4 K(k_1)} \in [0, 4 K(k_1)], 
\end{align*}
\begin{align*}
&\lam = (\b,c,r) \in N_2^{\pm} \then
\begin{cases}
\ds \sin \frac{\b}{2} = \pm \sn\left(\frac{\sqrt{r} \f}{k_2}, k_2\right), \\
\ds \frac{c}{2} = \pm \frac{\sqrt{r}}{k_2} \dn\left(\frac{\sqrt{r} \f}{k_2}, k_2\right), \\
\ds \cos \frac{\b}{2} = \cn\left(\frac{\sqrt{r} \f}{k_2}, k_2\right),
\end{cases}
\\
&k_2 = \sqrt{\frac{2 r}{E+r}} \in (0,1), \qquad \sqrt{r} \f \pmod{ 2 K(k_2)k_2} \in [0, 2 K(k_2)k_2],
\qquad \pm = \sgn c,
\end{align*}
\begin{align*}
&\lam = (\b,c, r) \in N_3^{\pm} \then
\begin{cases}
\ds \sin \frac{\b}{2} = \pm \tanh (\sqrt{r} \f), \\
\ds \frac{c}{2} = \pm \frac{\sqrt{r}}{\cosh(\sqrt{r} \f)}, \\
\ds \cos \frac{\b}{2} = \frac{1}{\cosh(\sqrt{r} \f)},
\end{cases}
\\
&k = 1, \qquad \f \in \R,
\qquad \pm = \sgn c.
\end{align*}

In the domain $N_2$ it will also be convenient to use the coordinates $(k_2, \psi, r)$, where
$$
\psi = \frac{\f}{k_2}, \qquad \sqrt r \psi \pmod{2 K(k_2)} \in [0, 2 K(k_2)].
$$
In computations, if this does not lead to ambiguity, we denote the both moduli of Jacobi's functions $k_1$ and $k_2$ by $k$, notice that $k \in (0,1)$, this is the normal case in the theory of Jacobi's elliptic functions, see~\cite{lawden}.

\section{Integration of the normal Hamiltonian system}
\label{sec:integration}

\subsection{Integration of the vertical subsystem}
\label{subsec:integr_vert}
In the elliptic coordinates $(\f,k,r)$ in the domain $\hN$, the vertical subsystem~\eq{pend_r} of the normal Hamiltonian system $\dlam = \vH(\lam)$ rectifies:
$$
\dot \f = 1, \qquad \dot k = 0, \qquad \dot r = 0,
$$
thus it has solutions
$$
\f_t = \f+t, \qquad k = \const, \qquad r = \const.
$$
Then expressions for the vertical coordinates $(\b,c,r)$ are immediately given by the formulas for elliptic coordinates derived in Subsec.~\ref{subsec:ell_coordN123}. For $\lam \in N \setminus \hN$, the vertical subsystem degenerates and is easily integrated.
So we obtain the following description of the solution $(\b_t, c_t, r)$ to the vertical subsystem~\eq{pend_r} with the initial condition $\restr{(\b_t, c_t, r)}{t=0} = (\b, c, r)$.

\begin{align*}
&\lam  \in N_1 \then
\begin{cases}
\ds \sin \frac{\b_t}{2} = k_1 \sn(\sqrt{r} \f_t), \\
\ds \cos \frac{\b_t}{2} = \dn(\sqrt{r} \f_t), \\
\ds \frac{c_t}{2} = k_1 \sqrt{r} \cn(\sqrt{r} \f_t),
\end{cases}\\
&\lam  \in N_2^{\pm} \then
\begin{cases}
\ds \sin \frac{\b_t}{2} = \pm \sn\left(\frac{\sqrt{r} \f_t}{k}\right), \\
\ds \cos \frac{\b_t}{2} = \cn\left(\frac{\sqrt{r} \f_t}{k}\right), \\
\ds \frac{c_t}{2} = \pm \frac{\sqrt{r}}{k} \dn\left(\frac{\sqrt{r} \f_t}{k}\right),
\end{cases}\\
&\lam  \in N_3^{\pm} \then
\begin{cases}
\ds \sin \frac{\b_t}{2} = \pm \tanh (\sqrt{r} \f_t), \\
\ds \cos \frac{\b_t}{2} = \frac{1}{\cosh(\sqrt{r} \f_t)}, \\
\ds \frac{c_t}{2} = \pm \frac{\sqrt{r}}{\cosh(\sqrt{r} \f_t)}.
\end{cases}\\
&\lam  \in N_4 \then \b_t \equiv 0, \quad c_t \equiv 0. \\
&\lam  \in N_5 \then \b_t \equiv \pi, \quad c_t \equiv 0. \\
&\lam  \in N_6 \then \b_t  = c t + \b, \quad c_t \equiv c. \\
&\lam  \in N_7 \then  c_t \equiv 0, \quad r \equiv 0.
\end{align*}

\subsection{Integration of the horizontal subsystem}
\label{subsec:integr_horiz_subs}

The Cauchy problem for the horizontal variables $(x,y,\t)$ of the normal Hamiltonian system~\eq{Ham_betacr}
has the form
\begin{align*}
&\dx = \cos \t = 2 \cos^2 \frac{\t}{2} - 1, && x_0 = 0, \\
&\dy = \sin \t = 2 \sin \frac{\t}{2} \cos\frac{\t}{2}, && y_0 = 0, \\
&\dth = c = \dot \b, && \t_0 = 0,
\end{align*}
thus
\be{thetatbetat}
\t_t = \b_t - \b.
\ee
We apply known formulas for integrals of Jacobi's elliptic functions, see Sec.~\ref{sec:append}, and obtain the following parametrization of normal extremal trajectories.

If $\lam \in N_1$, then
\begin{align*}
\sin \frac{\t_t}{2} &= k \dn (\sqrt{r} \f) \sn (\sqrt{r} \f_t)  - k \sn (\sqrt{r} \f) \dn (\sqrt{r} \f_t), \\
\cos \frac{\t_t}{2} &= \dn (\sqrt{r} \f) \dn (\sqrt{r} \f_t)  + k^2 \sn (\sqrt{r} \f) \sn (\sqrt{r} \f_t), \\
 x_t &=
\frac{2}{\sqrt r} \dn^2 (\sqrt r \f) (\E(\sqrt r \f_t) - \E(\sqrt r \f)) \\
&\qquad + \frac{4k^2}{\sqrt r}  \dn (\sqrt{r} \f) \sn (\sqrt{r} \f)
(\cn \sqrt{r} \f) - \cn (\sqrt{r} \f_t)) \\
&\qquad +
\frac{2k^2 }{\sqrt r}  \sn^2 (\sqrt{r} \f) (\sqrt r t + \E(\sqrt r \f) - \E(\sqrt r \f_t)) - t, \\
y_t &=
\frac{2k}{\sqrt r}(2 \dn^2 (\sqrt{r} \f) -1)(\cn (\sqrt{r} \f) - \cn (\sqrt{r} \f_t)) \\
&\qquad -
\frac{2k}{\sqrt r} \sn (\sqrt{r} \f) \dn (\sqrt{r} \f)(2(\E(\sqrt r \f_t) - \E(\sqrt r \f)) - \sqrt r t).
\end{align*}

Here $\E(u,k)$ is Jacobi's epsilon function, see Sec.~\ref{sec:append} and~\cite{lawden}.

The parametrization of trajectories in $N_2^+$ is obtained from the above parametrization in $N_1$ via the transformation $k \mapsto \frac{1}{k}$ described in Subsubsec.~\ref{subsubsec:ell_coordsC2-C3-};  after that, trajectories in $N_2^-$ are obtained via the reflection $i$~\eq{betac-beta-c}, see the chain~\eq{chainCi}.  In the domain $N_2$, we will use the coordinate
$$
\psi_t = \frac{\f_t}{k}.
$$
Then we obtain the following.

If $\lam \in N_2^{\pm}$, then
\begin{align*}
\sin \frac{\t_t}{2} &=
\pm( \cn(\sqrt r \psi) \sn(\sqrt r \psi_t) - \sn(\sqrt r \psi) \cn(\sqrt r \psi_t)), \\
\cos \frac{\t_t}{2} &=
\cn(\sqrt r \psi) \cn(\sqrt r \psi_t) + \sn(\sqrt r \psi) \sn(\sqrt r \psi_t), \\
x_t &= \frac{1}{\sqrt r} (1 - 2 \sn^2(\sqrt r \psi))
\left( \frac 2k (\E(\sqrt r \psi_t) - \E(\sqrt r \psi)) - \frac{2-k^2}{k^2} \sqrt r t \right)
\\
&\qquad +
\frac{4}{k \sqrt r} \cn(\sqrt r \psi) \sn(\sqrt r \psi) (\dn(\sqrt r \psi) - \dn(\sqrt r \psi_t)), \\
y_t &= \pm
\left(
\frac{2}{k \sqrt r}(2 \cn^2(\sqrt r \psi) - 1)(\dn(\sqrt r \psi) - \dn(\sqrt r \psi_t)) \right. \\
&\qquad \ \left. -
\frac{2}{\sqrt r} \sn(\sqrt r \psi) \cn(\sqrt r \psi)
\left( \frac{2}{k}(\E(\sqrt r \psi_t) - \E(\sqrt r \psi)) -
\frac{2-k^2}{k^2} \sqrt r t
\right)
\right).
\end{align*}

The formulas in $N_3^{\pm}$ are obtained from the above formulas in $N_2^{\pm}$ via the limit $k \to 1 - 0$, see the formulas of degeneration of Jacobi's functions~\eq{degen1}, and compare with chain~\eq{chainCi}.

Consequently, if $\lam \in N_3^{\pm}$, then
\begin{align*}
\sin \frac{\t_t}{2}
&=
\pm \left( \frac{\tanh (\sqrt r \f_t)}{\cosh (\sqrt r \f)} -
\frac{\tanh \sqrt r \f)}{\cosh (\sqrt r \f_t)}\right), \\
\cos \frac{\t_t}{2}
&=
\frac{1}{\cosh(\sqrt r \f) \cosh (\sqrt r \f_t)} + \tanh (\sqrt r \f) \tanh (\sqrt r \f_t), \\
x_t &=  (1 - 2 \tanh^2 (\sqrt r \f)) t +
\frac{4 \tanh (\sqrt r \f)}{\sqrt r \cosh (\sqrt r \f)}
\left( \frac{1}{\cosh (\sqrt r \f)} -  \frac{1}{\cosh (\sqrt r \f_t)}\right), \\
y_t &=
\pm \left(
\frac{2}{\sqrt r}
\left( \frac{2}{\cosh^2\sqrt r \f)}-1\right)
\left( \frac{1}{\cosh (\sqrt r \f)} - \frac{1}{\cosh (\sqrt r \f_t)}\right) \right. \\
&\qquad\qquad \left. - 2
\frac{\tanh (\sqrt r \f)}{\cosh (\sqrt r \f)} t
\right).
\end{align*}

Now we consider the special cases.

If $\lam \in N_4 \cup N_5 \cup N_7$, then
$$
\t_t = 0, \qquad x_t = t, \qquad y_t = 0.
$$

If $\lam \in N_6$, then
$$
\t_t = ct, \qquad
x_t = \frac{\sin ct}{c}, \qquad
y_t = \frac{1-\cos ct}{c}.
$$

So we parametrized the exponential mapping of Euler's elastic problem
$$
\Exp_t \ : \ \lam = (\b,c,r) \mapsto q_t = (\t_t, x_t, y_t), \qquad \lam \in N = T_{q_0}^* M, \quad q_t \in M,
$$
by Jacobi's elliptic functions.

\subsection{Euler elasticae}

Projections of extremal trajectories to the plane $(x,y)$ are stationary configurations of the elastic rod in the plane --- Euler elasticae. These curves satisfy the system of ODEs
\begin{align}
&\dx = \cos \t, \nonumber \\
&\dy = \sin \t, \nonumber \\
&\ddot \t = - r \sin(\t - \b), \qquad r, \b = \const. \label{pend2}
\end{align}
Depending on the value of energy $\ds E = \frac{\dth^2}{2} - r \cos(\t - \b) \in [-r, + \infty)$ and the constants of motion $r \in [0, + \infty)$, $\b \in S^1$,  of the generalized pendulum~\eq{pend2}, elasticae have different forms discovered by Euler.

\twofiglabel
{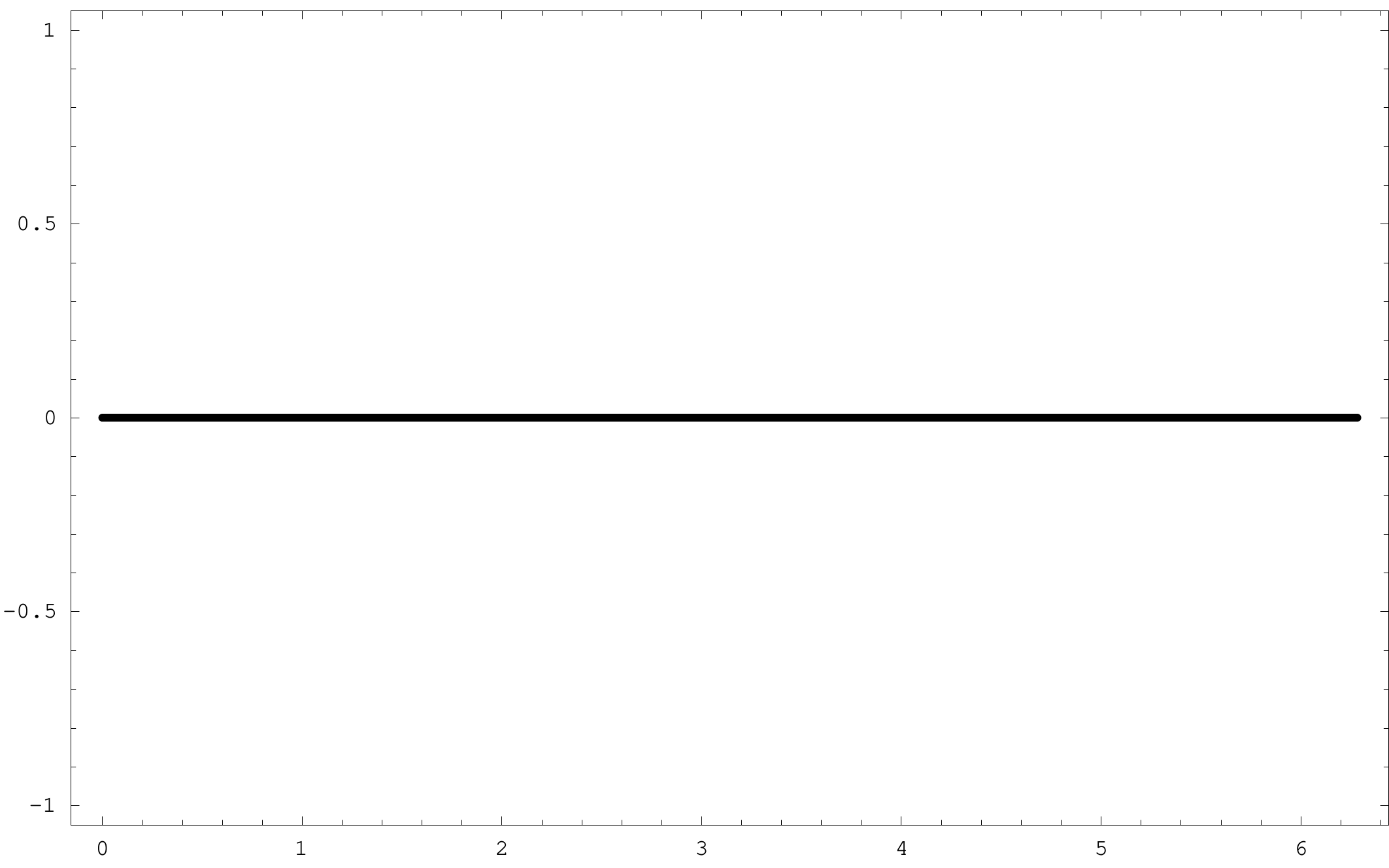}{$E = \pm r$, $r >  0$, $c = 0$}{fig:elastica1}
{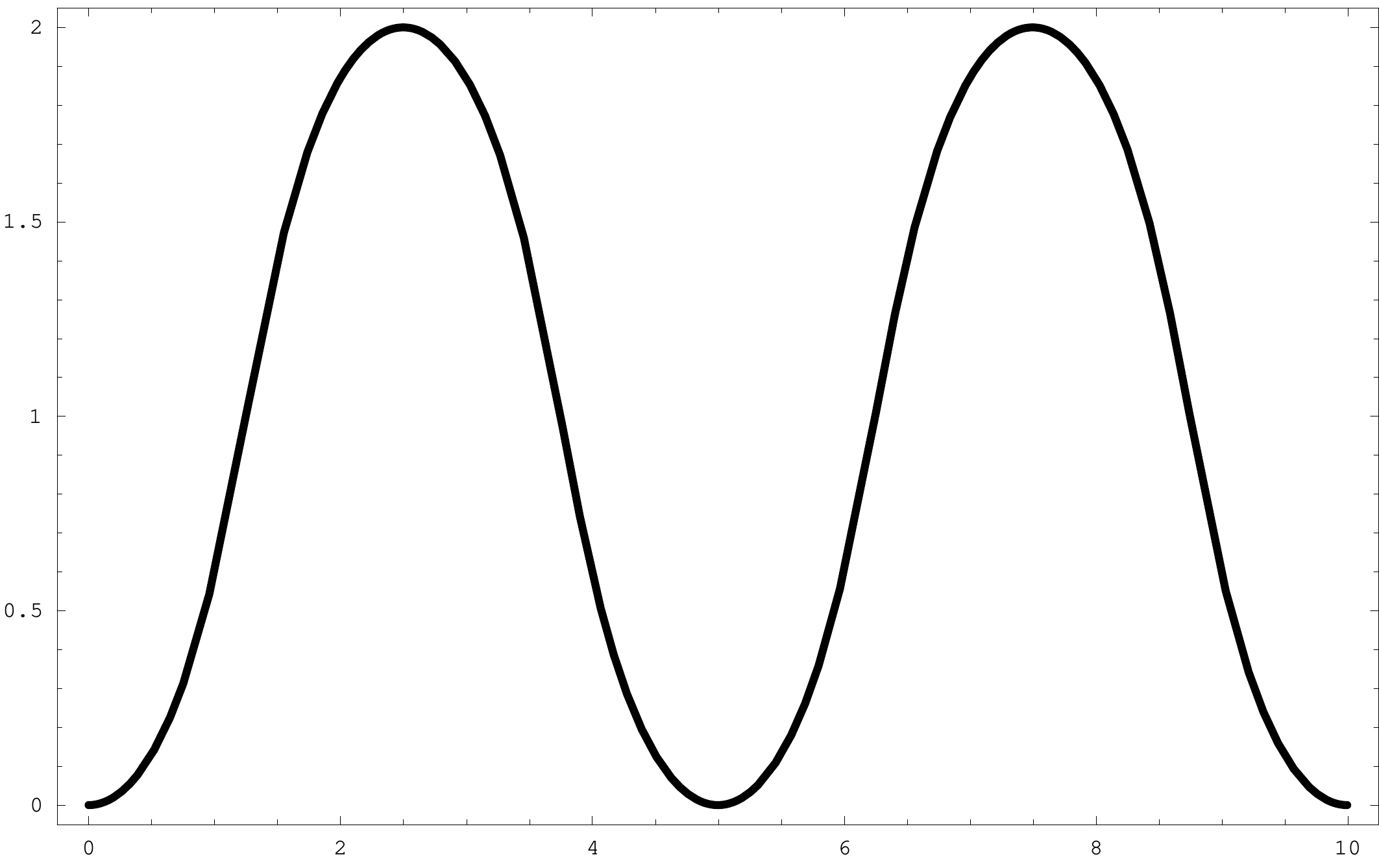}{$E \in (-r, r)$, $r >  0$, $k \in (0, \sqq)$}{fig:elastica2}

\twofiglabel
{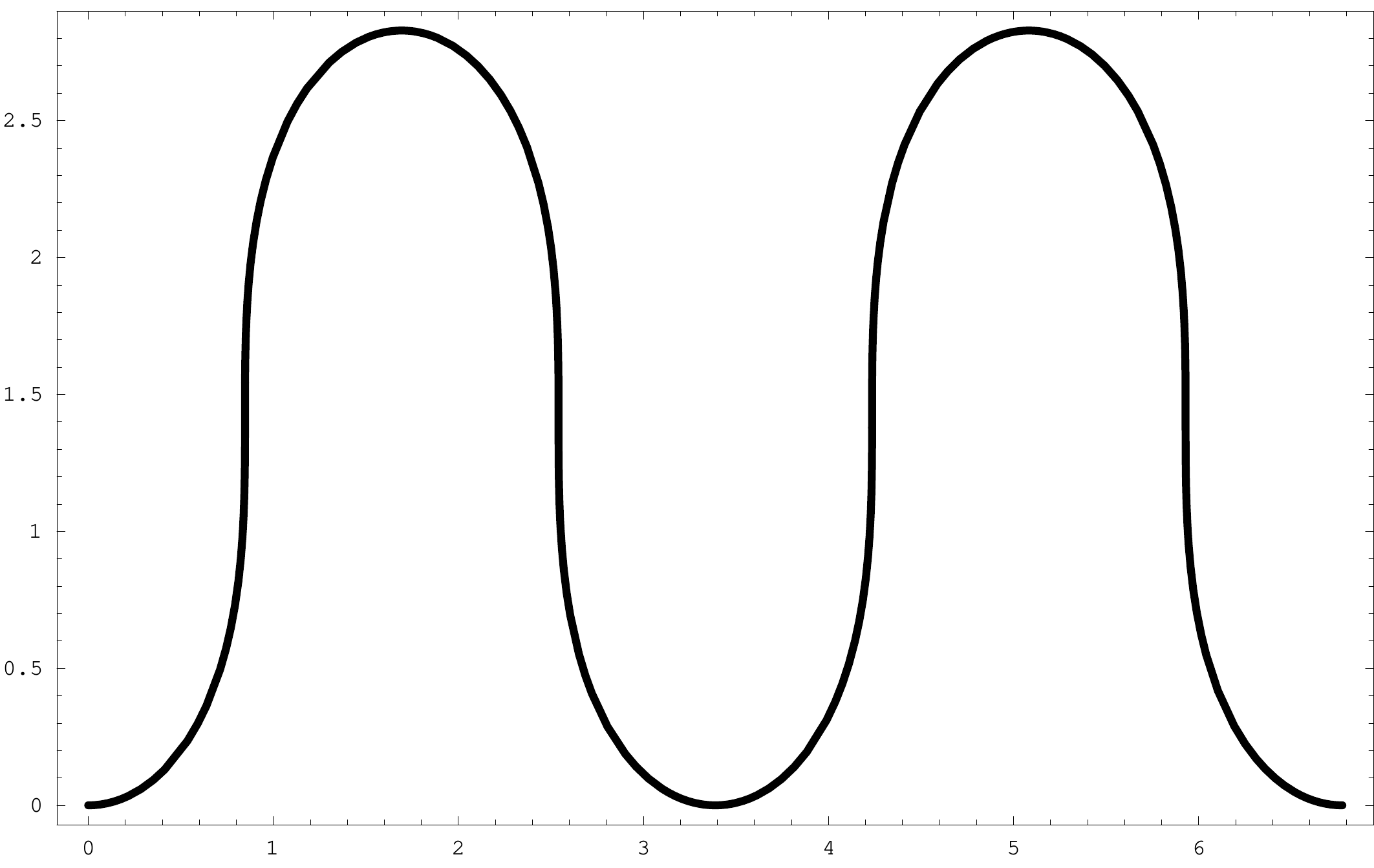}{$E \in (-r, r)$, $r >  0$, $k = \sqq$}{fig:elastica3}
{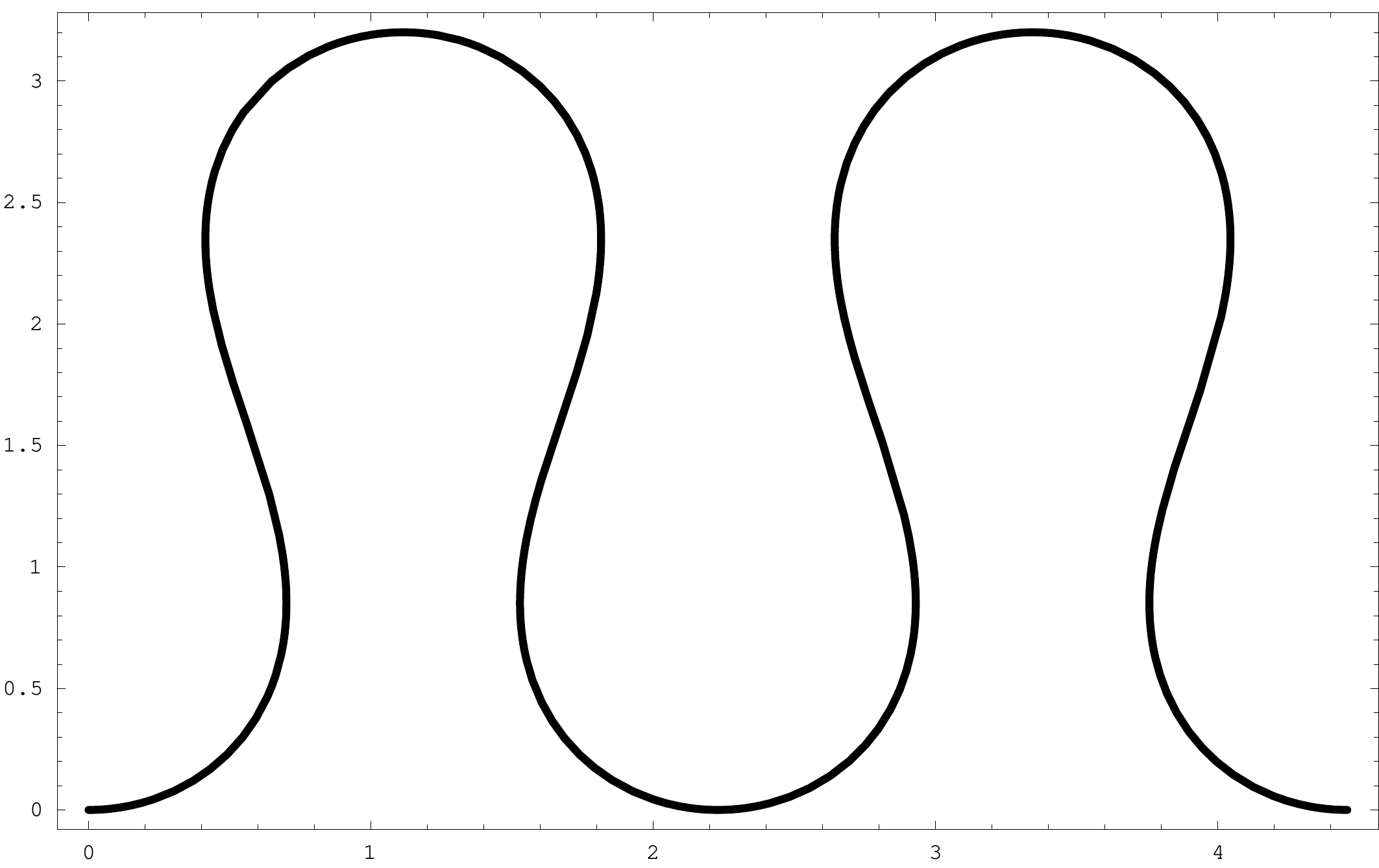}{$E \in (-r, r)$, $r >  0$, $k \in (\sqq, k_0)$}{fig:elastica4}

\twofiglabel
{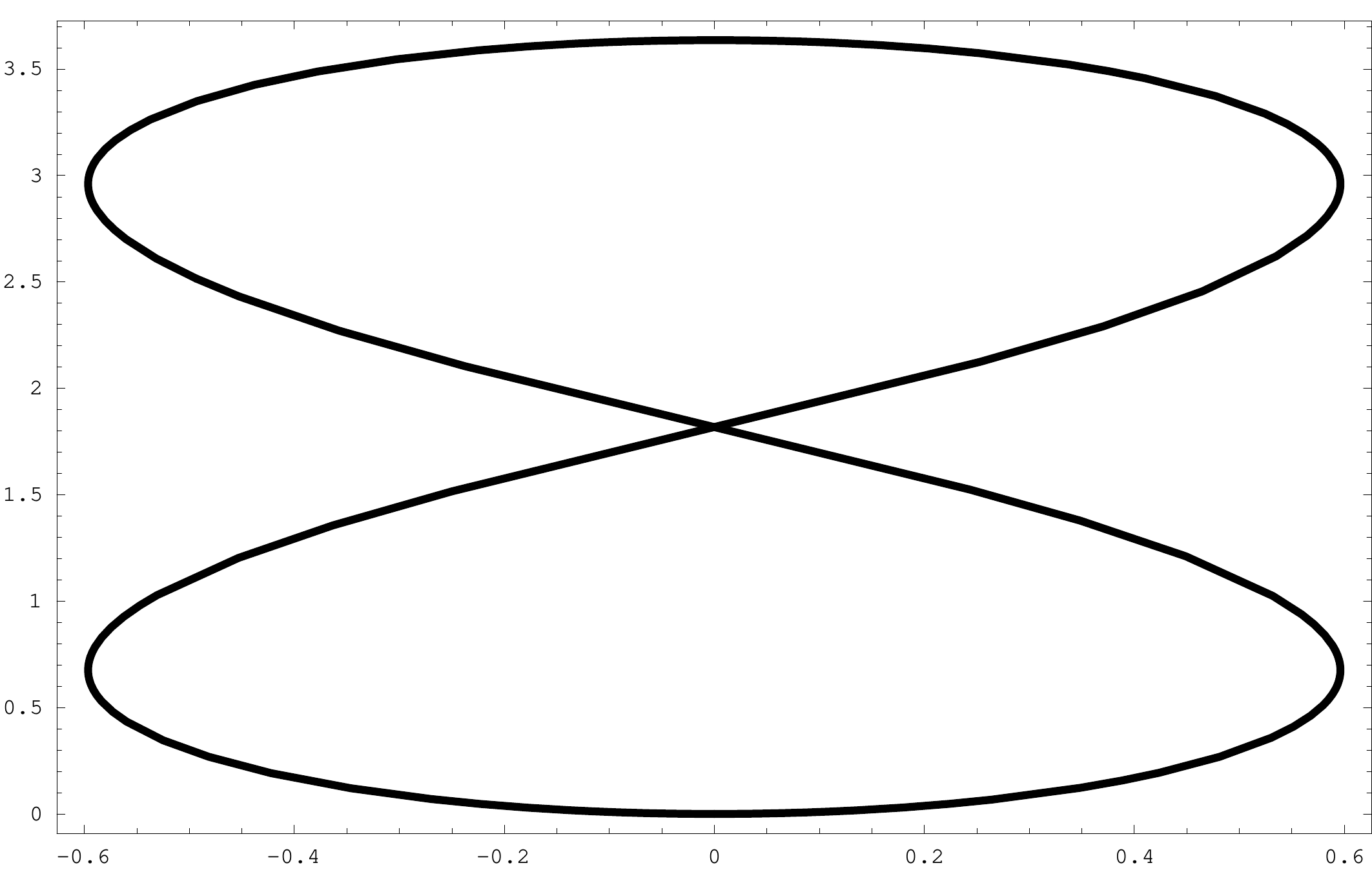}{$E \in (-r, r)$, $r >  0$, $k = k_0$}{fig:elastica5}
{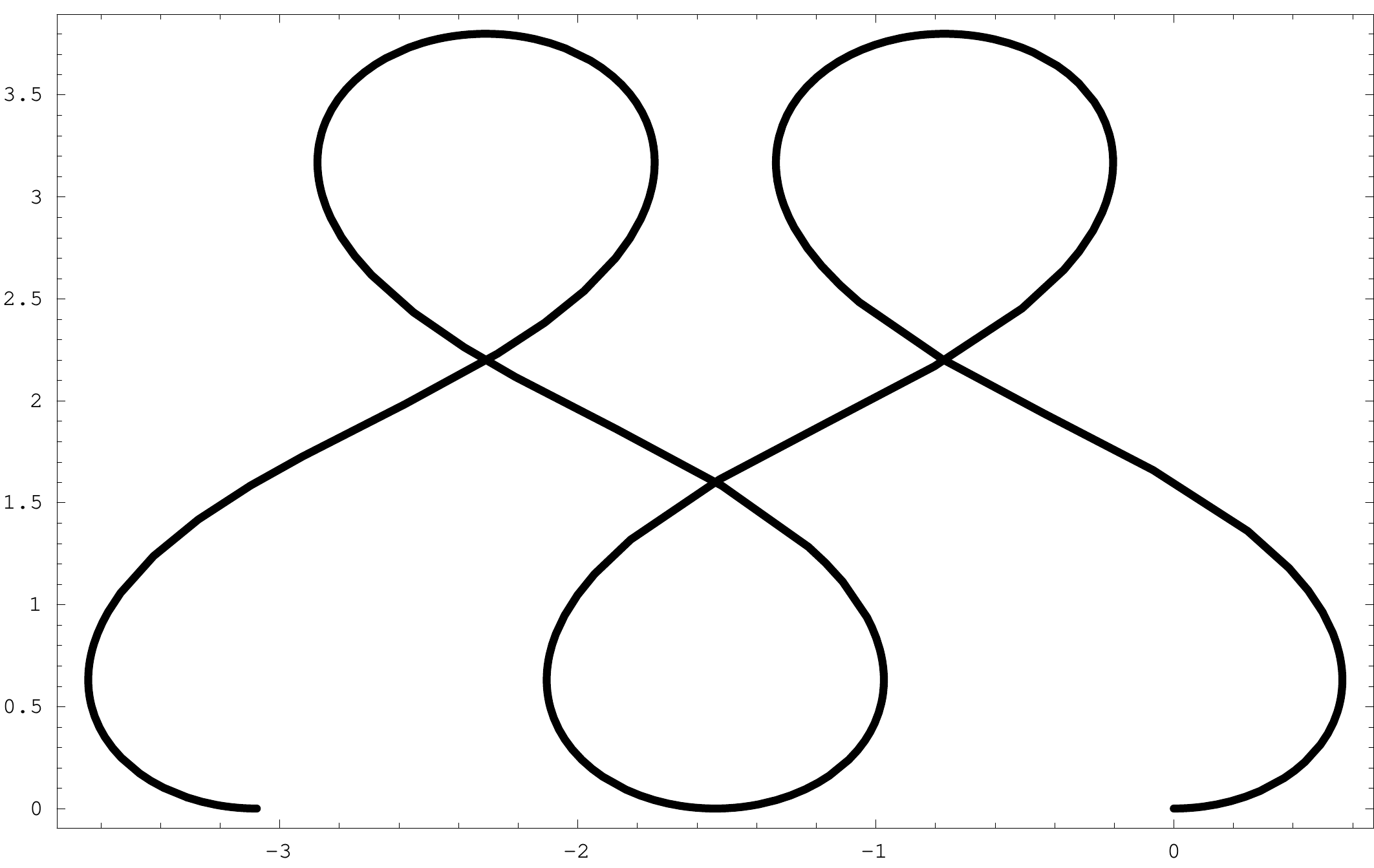}{$E \in (-r, r)$, $r >  0$, $k \in (k_0, 1)$}{fig:elastica6}

\twofiglabel
{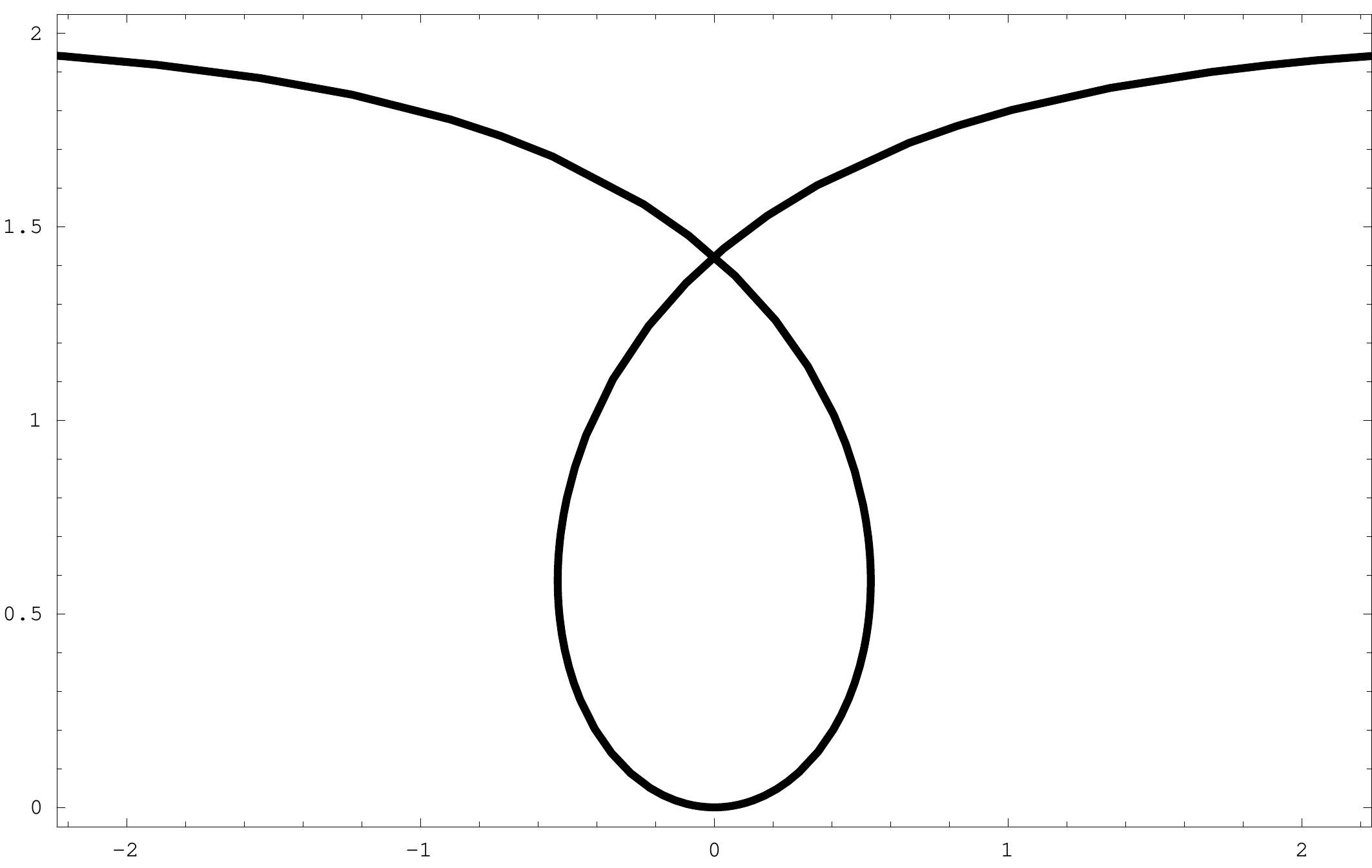}{$E = r > 0$, $\b \neq \pi$}{fig:elastica7}
{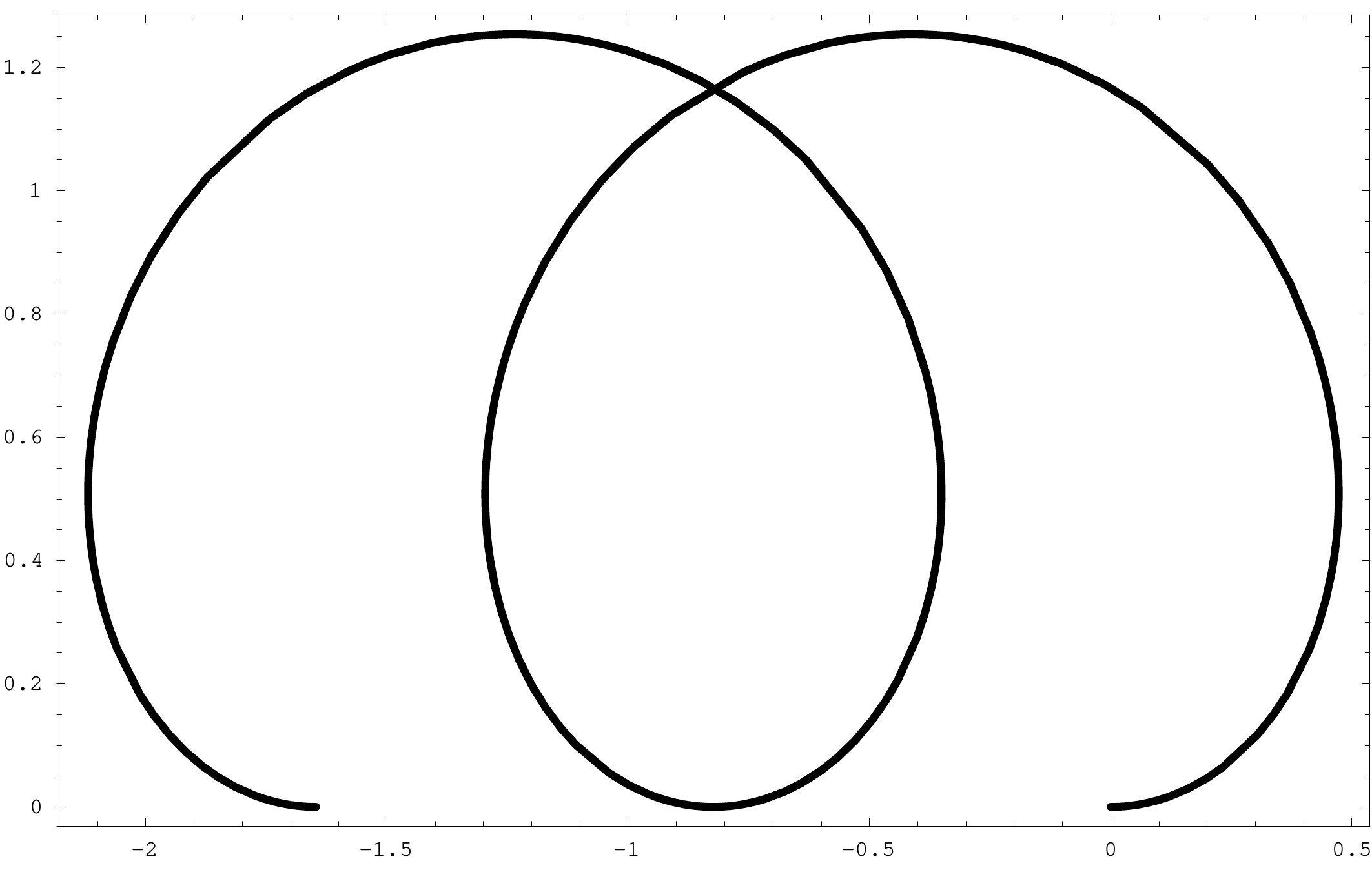}{$E > r > 0$}{fig:elastica8}

\onefiglabel
{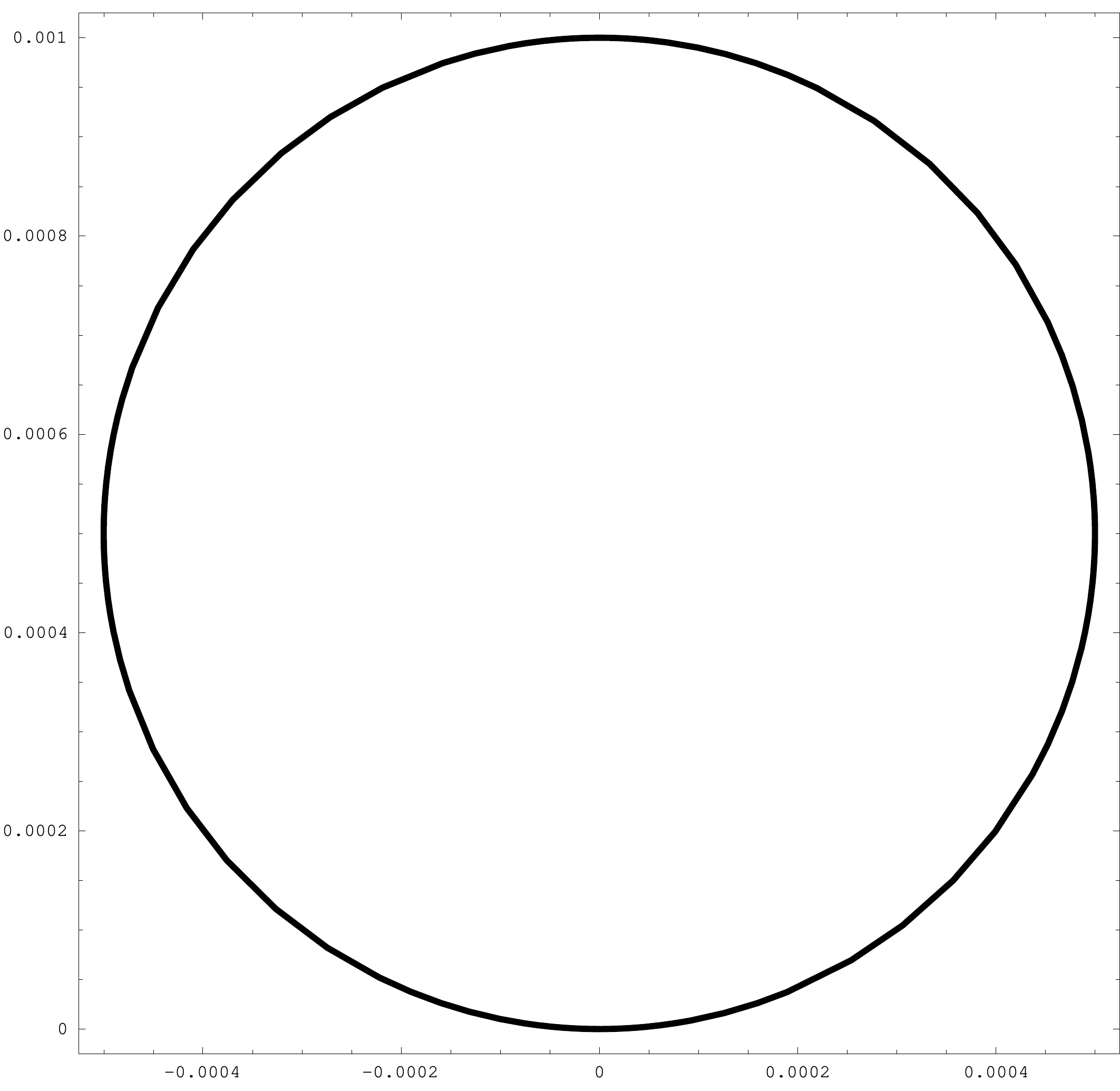}{$r = 0$, $c \neq 0$}{fig:elastica9}

If the energy $E$ takes the absolute minimum $-r \neq 0$, thus $\lam \in N_4$, then the corresponding elastica $(x_t, y_t)$ is a straight line (Fig.~\ref{fig:elastica1}). The corresponding motion of the generalized pendulum (Kirchoff's kinetic analogue) is the stable equilibrium.

If $E \in (-r, r)$, $r \neq 0$, thus $\lam \in N_1$, then the pendulum oscillates between extremal values of the angle, and the angular velocity $\dth$ changes its sign. The corresponding elasticae have inflections at the points where $\dth = 0$, and vertices at the points where $|\dth| = \max$ since $\dth$ is the curvature of an elastica $(x_t, y_t)$. Such elasticae are called inflectional. See the plots of different classes of inflectional elasticae at Figs.~\ref{fig:elastica2}--\ref{fig:elastica6}. The correspondence between the values of the modulus of elliptic functions $k = \ds \sqrt{\frac{E+r}{2r}} \in (0,1)$ and these figures is as follows:
\begin{align*}
&k \in \left(0, \frac{1}{\sqrt 2}\right) &&\then \text{Fig.~\ref{fig:elastica2}}, \\
&k = \frac{1}{\sqrt 2}                   &&\then \text{Fig.~\ref{fig:elastica3}}, \\
&k \in \left(\frac{1}{\sqrt 2}, k_0\right) &&\then \text{Fig.~\ref{fig:elastica4}}, \\
&k = k_0                                   &&\then \text{Fig.~\ref{fig:elastica5}}, \\
&k \in \left(k_0, 1\right)                 &&\then \text{Fig.~\ref{fig:elastica6}}.
\end{align*}
The value $k = \frac{1}{\sqrt 2}$ corresponds to the rectangular elastica studied by James Bernoulli (see Sec.~\ref{sec:history}). The value $k_0 \approx 0.909$ corresponds to the periodic elastica in the form of figure 8 and is described below in Propos.~\ref{propos:lem21max3}. 
As it was mentioned by Euler, when $k \to 0$, the inflectional elasticae tend to sinusoids.
The corresponding Kirchhoff's kinetic analogue is provided by the harmonic oscillator
$\ddot \t = - r (\t - \b)$.

If $E = r \neq 0$ and $\t - \b \neq \pi$, thus $\lam \in N_3$, then the pendulum approaches its unstable equilibrium $(\t - \b = \pi, \dth = 0)$ along the saddle separatrix, and the corresponding critical elastica has one loop, see Fig.~\ref{fig:elastica7}.

If $E = r \neq 0$ and $\t - \b = \pi$, thus $\lam \in N_5$, then the pendulum stays at its unstable equilibrium  $(\t - \b = \pi, \dth = 0)$, and the elastica is a straight line (Fig.~\ref{fig:elastica1}). 

If $E > r \neq 0$, thus $\lam \in N_2$, then the Kirchhoff's kinetic analogue is the pendulum rotating counterclockwise ($\dth > 0 \iff \lam \in N_2^+$) or clockwise ($\dth < 0 \iff \lam \in N_2^-$). The corresponding elasticae have nonvanishing curvature $\dth$, thus they have no inflection points and are called non-inflectional, see Fig.~\ref{fig:elastica8}. The points where $|\dth|$ has local maxima or minima are vertices of inflectional elasticae.

If $r = 0$ and $\dth \neq 0$, thus $\lam \in N_6$, then the pendulum rotates uniformly: one may think that  the gravitational acceleration is $g=0$ (see the physical meaning of the constant $r$~\eq{rgl}), while the angular velocity $\dth$ is nonzero.  The  corresponding elastica is a circle, see Fig.~\ref{fig:elastica9}.

Finally, if $r = 0$ and $\dth = 0$, thus $\lam \in N_7$, then the pendulum is stationary (no gravity with zero angular velocity $\dth$), and the elastica is a straight line, see Fig.~\ref{fig:elastica1}. 

Notice that the plots of elasticae at Figs.~\ref{fig:elastica2}--\ref{fig:elastica8} do not preserve the real ratio $\frac{y}{x}$ for the sake of saving space.

\section{Discrete symmetries of Euler's problem}
\label{sec:discr_sym}
In this section we lift discrete symmetries of the standard pendulum~\eq{pend_st} to discrete pendulum of the normal Hamiltonian system
\be{Ham2}
\begin{cases}
\dot \b = c, \\
\dot c = - r \sin \b,\\
\dot r = 0, \\
\dth = c, \\
\dx = \cos \t, \\
\dy = \sin \t.
\end{cases}
\ee

\subsection[Reflections in the phase cylinder of the standard pendulum]
{Reflections in the phase cylinder \\ of standard pendulum}
It is obvious that the following reflections of the phase cylinder of the standard pendulum $C = S^1_{\b} \times \R_c$ preserve the field of directions (although, not the vector field) determined by the ODE of the standard pendulum~\eq{pend_st}:
\begin{align*}
&\eps^1 \ : \ (\b,c) \mapsto (\b, -c), \\
&\eps^2 \ : \ (\b,c) \mapsto (-\b, c), \\
&\eps^3 \ : \ (\b,c) \mapsto (-\b, -c),
\end{align*}
see Fig.~\ref{fig:epsi1}.

\twofiglabel
{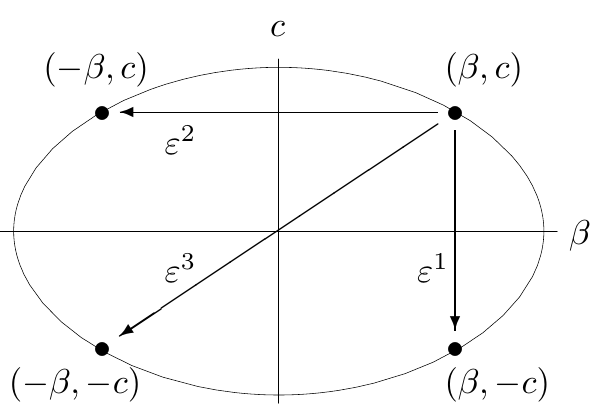}{Reflections in the phase cylinder of pendulum}{fig:epsi1}
{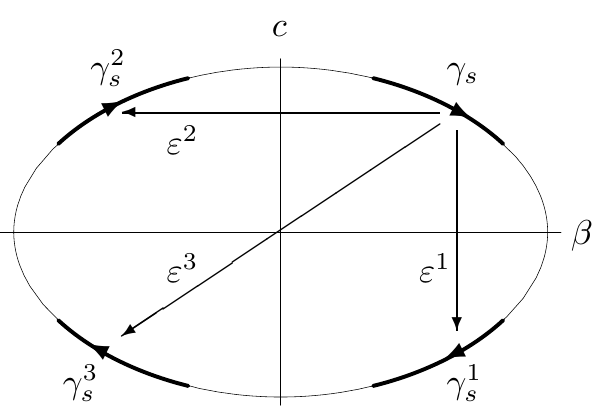}{Reflections of trajectories of pendulum}{fig:refl_pend_st}

These reflections generate the dihedral group --- the group of symmetries of the rectangle
$$
D_2 = \{ \Id, \eps^1, \eps^2, \eps^3 \}
$$
with the multiplication table
$$
\begin{array}{|c|c|c|c|}
\hline
  & \eps^1 & \eps^2 & \eps^3 \\
  \hline
\eps^1 & \Id & \eps^3 & \eps^2 \\
  \hline
\eps^2  & \eps^3 & \Id & \eps^1 \\
  \hline
\eps^3  & \eps^2 & \eps^1 & \Id \\
  \hline
  \end{array}
$$

Notice that the reflections $\eps^1$, $\eps^2$ reverse direction of time on trajectories of the pendulum, while $\eps^3$ preserves the direction of time (in fact, $\eps^3$ is the inversion $i$ defined in~\eq{betac-beta-c}).

All reflections $\eps^i$ preserve the energy of the pendulum $\ds E = \frac{c^2}{2} - \cos \b$.

\subsection{Reflections of trajectories of the standard pendulum}
We can define the action of reflections on trajectories of the standard pendulum as follows:
$$
\eps^i \ : \ \g_s = \{(\b_s, c_s) \mid s \in [0, t]\} \mapsto \g_s^i = \{(\b_s^i, c_s^i) \mid s \in [0, t]\},
$$
where
\begin{align}
&(\b_s^1, c_s^1) = (\b_{t-s}, -c_{t-s}), \label{beta1s} \\
&(\b_s^2, c_s^2) = (-\b_{t-s}, c_{t-s}), \label{beta2s} \\
&(\b_s^3, c_s^3) = (-\b_{s}, -c_{s}), \label{beta3s}
\end{align}
see Fig.~\ref{fig:refl_pend_st}.

All reflections $\eps^i$ map trajectories $\g_s$ to trajectories $\g_s^i$; they preserve both the total time of motion $t$ and the energy $\ds E = \frac{c^2}{2} - \cos \b$.

\subsection[Reflections of trajectories of the generalized pendulum]
{Reflections of trajectories \\ of the generalized pendulum}
\label{subsec:refl_pend_gen}
The action of reflections is obviously continued to trajectories of the generalized pendulum~\eq{pend_r} --- the vertical subsystem of the normal Hamiltonian system~\eq{Ham2} as follows:
\be{epsibetas}
\eps^i \ : \{(\b_s, c_s, r) \mid s \in [0, t]\} \mapsto \{(\b_s^i, c_s^i, r) \mid s \in [0, t]\},
\quad i = 1, 2, 3,
\ee
where the functions $\b_s^i$, $c_s^i$ are given by~\eq{beta1s}--\eq{beta3s}. Then the reflections $\eps^i$ preserve both the total time of motion $t$, the energy of the generalized pendulum $\ds E = \frac{c^2}{2} - r \cos \b$, and the elastic energy of the rod
$$
J = \frac 12 \int_0^t \dth_s^2 \, ds = \frac 12 \int_0^t c_s^2 \, ds.
$$

\subsection{Reflections of normal extremals}
Now we define action of the reflections $\eps^i$ on the normal extremals 
$$
\lam_s = e^{s \vH}(\lam_0) \in T^*M, \qquad s \in [0, t],
$$
i.e., solutions to the Hamiltonian system
\be{Ham3}
\begin{cases}
\dot \b_s = c_s, \\
\dot c_s = -r \sin \b_s, \\
\dot r = 0, \\
\dq_s = X_1(q_s) + c_s X_2(q_s),
\end{cases}
\ee
as follows:
\begin{align}
&\eps^i \ : \ \{\lam_s \mid s \in [0, t]\} \mapsto \{\lam_s^i \mid s \in [0, t]\},
\qquad i = 1, 2, 3, \label{epsilams} \\
&\lam_s = (\nu_s, q_s) = (\b_s, c_s, r, q_s), \qquad
\lam_s^i = (\nu_s^i, q_s^i) = (\b_s^i, c_s^i, r, q_s^i).
\label{lamsnusqs}
\end{align}
Here $\lam^i_s$ is a solution to the Hamiltonian system~\eq{Ham3}, and the action on the vertical coordinates
$$
\eps^i \ : \ \{\nu_s = (\b_s, c_s, r)\} \mapsto \{\nu_s^i = (\b_s^i, c_s^i, r)\}
$$
was defined in Subsec.~\ref{subsec:refl_pend_gen}. The action of reflections on horizontal coordinates $(\t, x, y)$ is described in the following subsection.

\subsection{Reflections of Euler elasticae}
Here we describe the action of reflections on the normal extremal trajectories
$$
\eps^i \ : \ \{q_s = (\t_s, x_s,y_s)\mid s \in [0, t]\}
\mapsto
\{q_s^i = (\t_s^i, x_s^i,y_s^i)\mid s \in [0, t]\}.
$$

\begin{proposition}
\label{propos:epsiqs}
Let $\lam_s = (\b_s, c_s, r, q_s)$ and $\lam^i_s = \eps^i(\lam_s) = (\b_s^i, c_s^i, r, q_s^i)$, $s \in [0, t]$,  be normal extremals defined in~\eq{epsilams}, \eq{lamsnusqs}. Then the following equalities hold:
\begin{itemize}
\item[$(1)$]
$\t_s^1 = \t_{t -s} - \t_t$,
\quad
$\vect{x_s^1 \\ y_s^1} =
\left( \begin{array}{cc}
\cos \t_{t} & \sin \t_t \\
-\sin \t_{t} & \cos \t_t
\end{array}\right)
\vect{x_t - x_{t-s} \\ y_t - y_{t -s}}$,
\item[$(2)$]
$\t_s^2 = \t_t - \t_{t -s}$,
\quad
$\vect{x_s^2 \\ y_s^2} =
\left( \begin{array}{cc}
\cos \t_{t} & -\sin \t_t \\
\sin \t_{t} & \cos \t_t
\end{array}\right)
\vect{x_t - x_{t-s} \\ y_{t -s} - y_t}$,
\item[$(3)$]
$\t_s^3 = - \t_{s}$,
\quad
$\vect{x_s^3 \\ y_s^3} =
\vect{x_s  \\ - y_s}$.
\end{itemize}
\end{proposition}
\begin{proof}
We prove only the formulas in item (1), the next two items are studied similarly. By virtue of~\eq{thetatbetat} and~\eq{beta1s}, we have:
$$
\t_s^1 = \b_s^1 - \b_0^1 = \b_{t-s} - \b_t = \t_{t-s} - \t_t.
$$
Further,
\begin{align*}
x_s^1
&=
\int_0^s \cos \t_r^1 \, d r = \int_0^s \cos (\t_{t-r} - \t_t)\,dr \\
&= \cos \t_t \int_0^s \cos \t_{t-r} \,dr + \sin \t_t \int_0^s \sin \t_{t-r} \,dr \\
&=
\cos \t_t (x_t - x_{t-s}) + \sin \t_t (y_t - y_{t-s}),
\end{align*}
and similarly
\begin{align*}
y_s^1 &= \int_0^s \sin \t_r^1 \, dr = \int_0^s \sin(\t_{t-r} - \t_t)\, dr \\
&= \cos \t_t(y_t - y_{t-s}) - \sin \t_t (x_t - x_{t-s}).
\end{align*}
\end{proof}

\begin{remark}
Notice the visual meaning of the action of the reflections $\eps^i$ on elastica $\{(x_s, y_s) \mid s \in [0, t]\}$ in the case $(x_t, y_t) \neq (x_0, y_0)$.

By virtue of the equality
\begin{align*}
\eps^1 \ : \
\vect{x_s \\ y_s}
&\stackrel{(1)}{\mapsto}
\vect{x_{t-s} \\y_{t-s}}
\stackrel{(2)}{\mapsto}
\vect{x_t - x_{t-s} \\ y_t - y_{t-s}}\\
&\stackrel{(3)}{\mapsto}
\left(
\begin{array}{cc}
\cos \t_t & \sin \t_t \\
-\sin \t_t & \cos \t_t
\end{array}
\right)
\vect{x_t - x_{t-s} \\ y_t - y_{t-s}}
=
\vect{x_s^1 \\ y_s^1},
\end{align*}
reflection $\eps^1$ is a composition of the following transformations: (1) inversion of time on elastica; (2) reflection of the plane $(x,y)$ in the center $p_c = (x_t/2, y_t/2)$ of the elastic chord $l$, i.e., the segment connecting its initial point $(x_0,y_0) = (0,0)$ and the endpoint $(x_t, y_t)$, and (3) rotation by the angle $(-\t_t)$; see Fig.~\ref{fig:eps1_el}.

\twofiglabel
{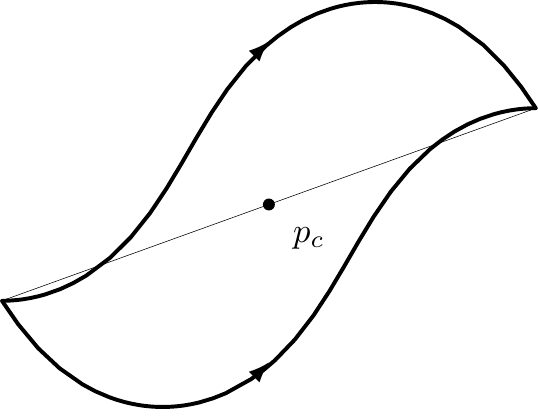}{Reflection of elastica in the center of chord $p_c$}{fig:eps1_el}
{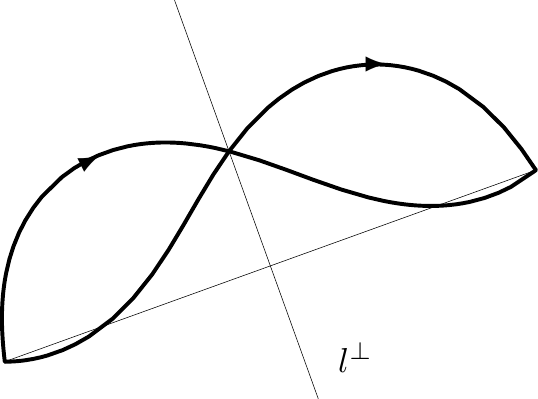}{Reflection of elastica in the middle perpendicular $l^{\perp}$}{fig:eps2_el}

For reflection $\eps^2$ we have the decomposition
\begin{align*}
&\eps^2 \ : \
\vect{x_s \\ y_s}
\stackrel{(1)}{\mapsto}
\vect{x_{t-s} \\y_{t-s}}
\stackrel{(2)}{\mapsto}
\vect{x_t \\ y_t }
+
\left(
\begin{array}{cc}
-\cos 2 \chi & -\sin 2 \chi \\
-\sin 2 \chi & \cos 2 \chi
\end{array}
\right)
\vect{x_{t - s} \\ y_{t - s}}
\\
&\stackrel{(3)}{\mapsto}
\left(
\begin{array}{cc}
\cos 2 \chi & \sin 2 \chi \\
-\sin 2 \chi & \cos 2 \chi
\end{array}
\right)
\left[
\vect{x_t \\ y_t }
+
\left(
\begin{array}{cc}
-\cos 2 \chi & -\sin 2 \chi \\
-\sin 2 \chi & \cos 2 \chi
\end{array}
\right)
\vect{x_{t - s} \\ y_{t - s}}
\right]\\
&=
\vect{x_t - x_{t-s} \\ y_{t-s} - y_{t}}
\stackrel{(4)}{\mapsto}
\left(
\begin{array}{cc}
\cos \t_t & -\sin \t_t \\
\sin \t_t & \cos \t_t
\end{array}
\right)
\vect{x_t - x_{t-s} \\ y_{t-s} - y_{t}}
=
\vect{x_s^2 \\ y_s^2},
\end{align*}
where $\chi$ is the polar angle of the point $(x_t, y_t)$:
$$
\cos \chi = \frac{x_t}{\sqrt{x_t^2 + y_t^2}}, \qquad
\sin \chi = \frac{y_t}{\sqrt{x_t^2 + y_t^2}}.
$$
Thus $\eps^2$ acts on elasticae as a composition of 4 transformations: (1) inversion of time on elastica; (2) reflection of the plane $(x,y)$ in the middle perpendicular $l^{\perp}$ to the elastic chord $l$; (3) and (4) rotations by the angles $(-2 \chi)$ and $\t_t$ respectively; see Fig.~\ref{fig:eps2_el}.

The symmetry $\eps^3$ acts on elasticae as reflection in the axis $x$. On the other hand, we have the following chain:
\begin{align*}
\eps^3 \ : \
\vect{x_s \\ y_s}
&\stackrel{(1)}{\mapsto}
\left(
\begin{array}{cc}
\cos 2 \chi & \sin 2 \chi \\
\sin 2 \chi & -\cos 2 \chi
\end{array}
\right)
\vect{x_{s} \\y_{s}}
\\
&\stackrel{(2)}{\mapsto}
\left(
\begin{array}{cc}
\cos 2 \chi & \sin 2 \chi \\
-\sin 2 \chi & \cos 2 \chi
\end{array}
\right)
\left[
\left(
\begin{array}{cc}
\cos 2 \chi & \sin 2 \chi \\
\sin 2 \chi & -\cos 2 \chi
\end{array}
\right)
\vect{x_{s} \\y_{s}}
\right]
\\
&=
\vect{x_s \\ - y_s} =
\vect{x_s^3 \\ y_s^3},
\end{align*}
this is a composition of: (1) reflection of the plane $(x,y)$ in the elastic chord $l$; and (2) rotation by the angle $(-2 \chi)$; see Fig.~\ref{fig:eps3_el}.

\onefiglabel{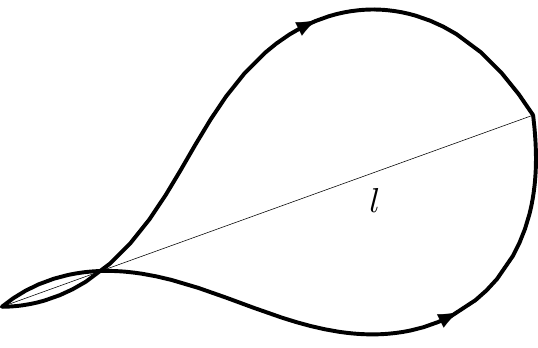}{Reflection of elastica in chord $l$}{fig:eps3_el}

So, modulo inversion of time on elasticae and rotations of the plane $(x,y)$, we have:
\begin{itemize}
\item
$\eps^1$ is the reflection of elastica in the center of its chord;
\item
$\eps^2$ is the reflection of elastica in the middle perpendicular to its chord;
\item
$\eps^3$ is the reflection of elastica in its chord.
\end{itemize}
\end{remark}

\subsection{Reflections of endpoints of extremal trajectories}
\label{subsec:epsiM}
Now we can define the action of reflections in the state space $M = \R^2_{x,y} \times S^1_{\t}$ as the action on endpoints of extremal trajectories:
\be{epsiM}
\map{\eps^i}{M}{M},
\qquad
\mapto{\eps^i}{q_t}{q_t^i},
\ee
as follows:
\begin{align}
\label{eps1M}
&\mapto{\eps^1}{\vect{\t_t \\ x_t \\ y_t}}{\vect{-\t_t \\ x_t \cos \t_t + y_t \sin \t_t \\ -x_t \sin \t_t + y_t \cos \t_t}}, \\
\label{eps2M}
&\mapto{\eps^2}{\vect{\t_t \\ x_t \\ y_t}}{\vect{\t_t \\ x_t \cos \t_t + y_t \sin \t_t \\ x_t \sin \t_t - y_t \cos \t_t}}, \\
\label{eps3M}
&\mapto{\eps^3}{\vect{\t_t \\ x_t \\ y_t}}{\vect{-\t_t \\ x_t  \\ - y_t}}.
\end{align}
These formulas follow directly from Propos.~\ref{propos:epsiqs}. Notice that the action of reflections $\map{\eps^i}{M}{M}$ is well-defined in the sense that the image $\eps^i(q_t)$ depends only on the point $q_t$, but not on the whole trajectory $\{q_s \mid s \in [0, t]\}$.

\subsection{Reflections as symmetries of the exponential mapping}
The action of reflections $\eps^i$ on the vertical subsystem of the normal Hamiltonian system~\eq{epsibetas} defines the action of $\eps^i$ in the preimage of the exponential mapping by restriction to the initial instant $s = 0$:
$$
\mapto{\eps^i}{\nu = (\b,c,r)}{\nu^i = (\b^i,c^i,r)},
$$
where $(\b,c,r) = (\b_0, c_0, r)$, $(\b^i,c^i,r) = (\b_0^i,c_0^i,r)$ are the initial points of the curves $\nu_s = (\b_s,c_s,r)$ and $\nu_s^i = (\b_s^i,c_s^i,r)$. The explicit formulas for $(\b^i,c^i)$ are derived from formulas~\eq{beta1s}--\eq{beta3s}:
\begin{align*}
&(\b^1,c^1) = (\b_t, -c_t), \\
&(\b^2,c^2) = (-\b_t, c_t), \\
&(\b^3,c^3) = (-\b_0, -c_0).
\end{align*}

So we have the action of reflections in the preimage of the exponential mapping:
$$
\map{\eps^i}{N}{N},
\qquad
\eps^i(\nu) = \nu^i,
\qquad \nu, \ \nu^i \in N = T_{q_0}^*M.
$$

Since the both actions of $\eps^i$ in $N$ and $M$ are induced by the action of $\eps^i$ on extremals $\lam_s$~\eq{epsilams}, we obtain the following statement.

\begin{proposition}
Reflections $\eps^i$ are symmetries of the exponential mapping $\map{\Exp_t}{N}{M}$, i.e., the following diagram is commutative:
$$
\xymatrix{
N \ar[r]^{\Exp_t} \ar[d]^{\eps^i} & M  \ar[d]^{\eps^i}\\
N \ar[r]^{\Exp_t} & M
}
\qquad\qquad\qquad
\xymatrix{
\nu \ar@{|->}[r]^{\Exp_t} \ar@{|->}[d]^{\eps^i} & q_t  \ar@{|->}[d]^{\eps^i}\\
\nu^i  \ar@{|->}[r]^{\Exp_t} & q_t^i
}
$$
\end{proposition}

\subsection[Action of reflections in the preimage of the exponential mapping]
{Action of reflections \\ in the preimage of the exponential mapping}
In this subsection we describe the action of reflections
$$
\map{\eps^i}{N}{N}, \qquad \eps^i(\nu) = \nu^i,
$$
in elliptic coordinates (Subsec.~\ref{subsec:ell_coordN123}) in the preimage of the exponential mapping~$N$.

\begin{proposition}
\label{propos:eps123N}
\begin{itemize}
\item[$(1)$]
If $\nu = (k,\f,r) \in N_1$, then $\nu^i = (k,\f^i,r) \in N_1$, and
\begin{align*}
&\f^1 + \f_t = \frac{2K}{\sqrt r} \left(\bmod{\frac{4K}{\sqrt r}}\right), \\
&\f^2 + \f_t = 0 \left(\bmod{\frac{4K}{\sqrt r}}\right), \\
&\f^3 - \f = \frac{2K}{\sqrt r} \left(\bmod{\frac{4K}{\sqrt r}}\right).
\end{align*}
\item[$(2)$]
If $\nu = (k,\psi,r) \in N_2$, then $\nu^i = (k,\psi^i,r) \in N_2$, moreover,
\be{nuinN2+-}
\nu \in N_2^{\pm} \then
\nu^1 \in N_2^{\mp}, \quad \nu^2 \in N_2^{\pm}, \quad \nu^3 \in N_2^{\mp},
\ee
and
\begin{align*}
&\psi^1 + \psi_t = 0 \left(\bmod{\frac{2K}{\sqrt r}}\right), \\
&\psi^2 + \psi_t = 0 \left(\bmod{\frac{2K}{\sqrt r}}\right), \\
&\psi^3 - \psi = 0 \left(\bmod{\frac{2K}{\sqrt r}}\right).
\end{align*}
\item[$(3)$]
If $\nu = (\f,r) \in N_3$, then $\nu^i = (\f^i,r) \in N_3$, moreover,
$$
\nu \in N_3^{\pm} \then
\nu^1 \in N_3^{\mp}, \quad \nu^2 \in N_3^{\pm}, \quad \nu^3 \in N_3^{\mp},
$$
and
\begin{align*}
&\f^1 + \f_t = 0, \\
&\f^2 + \f_t = 0, \\
&\f^3 - \f = 0.
\end{align*}
\item[$(4)$]
If $\nu = (\b,c,r) \in N_6$, then $\nu^i = (\b^i,c^i,r) \in N_6$, moreover,
\be{N6+--+}
\nu \in N_6^{\pm} \then
\nu^1 \in N_6^{\mp}, \quad \nu^2 \in N_6^{\pm}, \quad \nu^3 \in N_6^{\mp},
\ee
and
\begin{align}
&(\b^1,c^1) = (\b_t, -c), \nonumber \\
&(\b^2,c^2) = (-\b_t, c), \label{beta2c2N6}\\
&(\b^3,c^3) = (-\b, -c).    \nonumber
\end{align}
\end{itemize}
\end{proposition}
\begin{proof}
We prove only item (1) since the other items are proved similarly.

The reflections $\eps^i$ preserve the domain $N_1$ since
\begin{gather*}
\mapto{\eps^i}{E}{E}, \qquad \mapto{\eps^i}{r}{r}, \\
\mapto{\eps^1, \ \eps^3}{c}{-c}, \qquad \mapto{\eps^2}{c}{c},
\end{gather*}
this follows from equalities~\eq{beta1s}--\eq{beta2s}. Further, we obtain from~\eq{beta1s} that
$$
\t^1 = \t_t, \qquad c^1 = -c_t,
$$
whence by virtue of the construction of elliptic coordinates (Subsec.~\ref{subsec:ell_coordN123}) it follows that
$$
\sn(\sqrt r \f^1) = \sn(\sqrt r \f_t), \qquad
\cn(\sqrt r \f^1) = -\cn(\sqrt r \f_t),
$$
thus $\f^1 + \f_t =  \frac{2K}{\sqrt r} \left(\bmod{\frac{4K}{\sqrt r}}\right)$. The expressions for action of the rest reflections in elliptic coordinates are obtained in a similar way.
\end{proof}

\section{Maxwell strata}
\label{sec:Maxwell}

\subsection{Optimality of normal extremal trajectories}
Consider an analytic optimal control problem of the form:
\begin{align}
&\dq = f(q,u), \qquad q \in M, \quad u \in U, \label{gp11} \\
&q(0) = q_0, \qquad q(t_1) = q_1, \qquad t_1 \text{ fixed}, \label{gp21} \\
&J_{t_1}[q,u] = \int_0^{t_1} \f(q(t),u(t)) \, dt \to \min. \label{gp31}
\end{align}
Here $M$ and $U$ are finite-dimensional analytic manifolds, and $f(q,u)$, $\f(q,u)$ are respectively an analytic vector field and a function depending on the control parameter $u$. Let
$$
h_u(\lam) = \langle \lam, f(q,u)\rangle - \f(q,u), \qquad \lam \in T^* M, \quad q = \pi(\lam) \in M, \quad u \in U
$$
be the normal Hamiltonian of Pontryagin Maximum Principle for this problem, see Subsec.~\ref{subsec:PMP} and~\cite{notes}. Suppose that all normal extremals $\lam_t$ of the problem are regular, i.e., the strong Legendre condition is satisfied:
\be{Legendre}
\restr{\frac{\partial^2}{\partial u^2}}{u(t)} h_u(\lam_t) < - \d, \qquad \d > 0,
\ee
for the corresponding extremal control $u(t)$. Then the maximized Hamiltonian $\ds H(\lam) = \max_{u \in U} h_u(\lam)$ is analytic, and there is defined the exponential mapping for time $t$:
$$
\map{\Exp_t}{N = T_{q_0}^* M}{M},
\qquad
\Exp_t(\lam) = \pi \circ e^{t \vH}(\lam) = q(t).
$$

Suppose that the control $u$ maximizing the Hamiltonian $h_u(\lam)$ is an analytic function $u = u(\lam)$, $\lam \in T^*M$.

For covectors $\lam, \tlam \in T_{q_0}^* M$, we denote the corresponding extremal trajectories as
$$
q_s = \Exp_s(\lam), \qquad \tq_s = \Exp_s(\tlam)
$$
and the extremal controls as
\begin{align*}
&u(s) = u(\lam_s), \qquad \lam_s = e^{s \vH}(\lam), \\
&\tu(s) = u(\tlam_s), \qquad \tlam_s = e^{s \vH}(\tlam).
\end{align*}

The time $t$ \ddef{Maxwell set} in the preimage of the exponential mapping 
$N = T_{q_0}^*M$
is defined as follows:
\be{MAXt}
\MAX_t
= \left\{ \lam \in N \mid \exists \ \tlam \in N \ : \ \tq_s \not\equiv q_s, \ s \in [0, t], \quad \tq_t = q_t, \ J_t[q,u] = J_t[\tq,\tu]\right\}.
\ee
The inclusion $\lam \in \MAX_t$ means that two distinct extremal trajectories $\tq_s \not\equiv q_s$ with the same value of the cost functional $J_t[q,u] = J_t[\tq,\tu]$ intersect one another at the point  $\tq_t = q_t$, see Fig.~\ref{fig:Maxwell}.

\onefiglabelsize{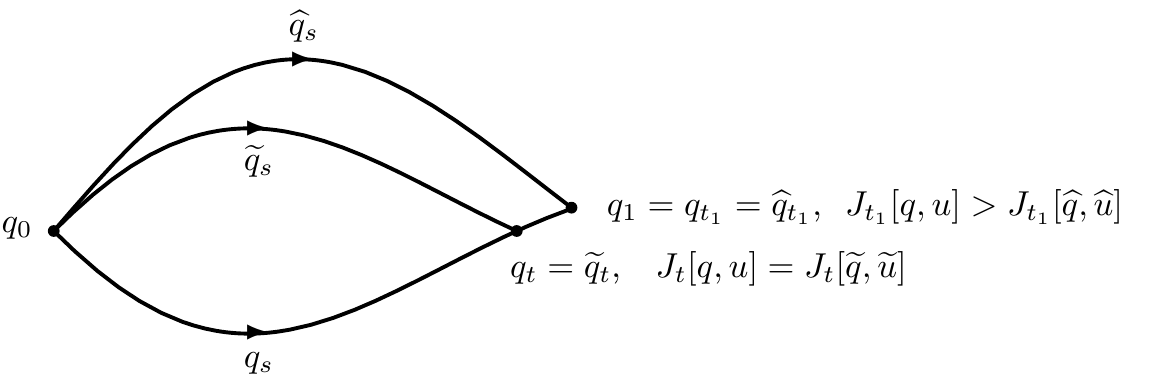}{Maxwell point $q_t$}{fig:Maxwell}{1}

The point $q_t$ is called a \ddef{Maxwell point} of the trajectory $q_s$, $s \in [0, t_1]$, and the instant $t$ is called a \ddef{Maxwell time}.

Maxwell set is closely related to optimality of extremal trajectories: such a trajectory cannot be optimal after a Maxwell point.
The following statement is a modification of a similar proposition proved by S.Jacquet~\cite{jacquet} in the context of sub-Riemannian problems.

\begin{proposition}
\label{propos:Maxwell}
If a normal extremal trajectory $q_s$, $s \in [0, t_1]$, admits a Maxwell point $q_t$, $t \in (0, t_1)$, then $q_s$ is not optimal in the problem~\eq{gp11}--\eq{gp31}.
\end{proposition}
\begin{proof}
By contradiction, assume that the trajectory $q_s$, $s \in [0, t_1]$, is optimal. Then the broken curve
$$
q_s' =
\begin{cases}
\tq_s, & s \in [0, t], \\
q_s, & s \in [t, t_1]
\end{cases}
$$
is an admissible trajectory of system~\eq{gp11} with the control
$$
u_s' =
\begin{cases}
\tu(s), & s \in [0, t], \\
u(s), & s \in [t, t_1].
\end{cases}
$$
Moreover, the trajectory $q_s'$ is optimal in the problem~\eq{gp11}--\eq{gp31} since
\begin{align*}
J_{t_1}[q',u']
&= \int_0^{t_1} \f(q_s', u'(s))\, ds =
\int_0^{t} \f(q_s', u'(s))\, ds + \int_t^{t_1} \f(q_s', u'(s))\, ds \\
&=   \int_0^{t} \f(\tq_s, \tu(s))\, ds + \int_t^{t_1} \f(q_s, u(s))\, ds \\
&= J_t[\tq, \tu] +  \int_t^{t_1} \f(q_s, u(s))\, ds =
J_t[q, u] +  \int_t^{t_1} \f(q_s, u(s))\, ds\\
&= J_{t_1}[q, u],
\end{align*}
which is minimal since $q_s$ is optimal.

So the trajectory $q_s'$ is extremal, in particular, it is analytic. Thus the analytic curves $q_s$ and $q_s'$ coincide one with another at the segment $s \in [t, t_1]$. By the uniqueness theorem for analytic functions, these curves must coincide everywhere: $q_s \equiv q_s'$, $s \in [0, t_1]$, thus  $q_s \equiv \tq_s$, $s \in [0, t_1]$, which contradicts to definition of Maxwell point $q_t$.
\end{proof}

Maxwell points were successfully applied for the study of optimality of geodesics in several sub-Riemannian problems~\cite{ABCK, myasnich}. We will apply this notion in order to obtain an upper bound on cut time, i.e., time where the normal extremals lose optimality, see~\cite{max1, max2, max3} for a similar result for the nilpotent sub-Riemannian problem with the growth vector (2,3,5).

As  noted in the book by V.I.Arnold~\cite{arnold_osob_kaustic},  the term {\em Maxwell point} originates ``in connection with the Maxwell rule of the van der Waals theory, according to  which phase transition takes place at a  value of the parameter for which two maxima of a certain smooth function are equal to each other''.

\subsection{Maxwell strata generated by reflections}
We return to Euler's elastic problem~\eq{sys1}--\eq{J}. It is easy to see that this problem has form~\eq{gp11}--\eq{gp31} and satisfies all assumptions stated in the previous subsection, so Propos.~\ref{propos:Maxwell} holds for Euler's problem.

Consider the action of reflections in the preimage of the exponential mapping:
$$
\map{\eps^i}{N}{N}, \qquad \eps^i(\lam) = \lam^i,
$$
and denote the corresponding extremal trajectories
$$
q_s = \Exp_s(\lam), \qquad q_s^i = \Exp_s(\lam^i)
$$
and extremal controls~\eq{u=h2}
$$
u(s) = c_s, \qquad u^i(s) = c_s^i.
$$

The \ddef{Maxwell strata corresponding to reflections $\eps^i$} are defined as follows:
\be{MAXit}
\MAX_t^i = \{\lam \in N \mid
q_s^i \not\equiv q_s, \ q_t^i = q_t, \ J_t[q,u] = J_t[q^i,u^i]\},
\quad i = 1, 2, 3, \quad t > 0.
\ee
It is obvious that
$$
\MAX_t^i \subset \MAX_t, \qquad \qquad i = 1, 2, 3.
$$

\begin{remark}
Along normal extremals we have
$$
J_t[q,u] = \frac 12 \int_0^t c_s^2 \, ds.
$$
In view of the expression for the action of reflections $\eps^i$ on trajectories of the pendulum~\eq{beta1s}--\eq{beta3s}, we have
$$
J_t[q^i,u^i] = J_t[q,u],
\qquad i = 1, 2, 3,
$$
i.e., the last condition in the definition of the Maxwell stratum $\MAX_t$ is always satisfied.
\end{remark}

\subsection{Extremal trajectories preserved by reflections}
In this subsection we describe the normal extremal trajectories $q_s$ such that $q_s^i \equiv q_s$. This identity appears in the definition of Maxwell strata $\MAX_t^i$~\eq{MAXit}.

\begin{proposition}
\label{propos:qisequivqs}
\begin{itemize}
\item[$(1)$]
$q_s^1 \equiv q_s \iff \lam^1 = \lam$.
\item[$(2)$]
$q_s^2 \equiv q_s \iff \lam^2 = \lam \text{ or } \lam \in N_6$.
\item[$(3)$]
$q_s^3 \equiv q_s \iff \lam^3 = \lam$.
\end{itemize}
\end{proposition}
\begin{proof}
First of all notice the chain
\be{qis=qs}
q_s^i \equiv q_s \then \t_s^i \equiv \t_s \then \b^i_s - \b^i_0 \equiv \b_s - \b_0, \qquad i = 1, 2, 3.
\ee

(1) Let $q^1_s \equiv q_s$. By equality~\eq{beta1s}, $\b^1_s = \b_{t-s}$, thus we obtain from~\eq{qis=qs} that
$$
\b_{t-s} - \b_t \equiv \b_s - \b_0.
$$
For $s = t$ we have $\b_t = \b_0$, thus
$$
\b_{t-s} \equiv \b_s.
$$
Differentiating w.r.t. $s$ and taking into account the equation of generalized pendulum~\eq{pend_r}, we obtain
$$
c_{t-s} \equiv -c_s.
$$
In view of equality~\eq{beta1s},
$$
(\b^1_s,c^1_s) \equiv (\b_s,c_s) \then
(\b^1,c^1) = (\b,c) \then
\lam = \lam^1.
$$
Conversely, if $\lam^1 = \lam$, then $q^1_s \equiv q_s$.

(2) Let $q^2_s \equiv q_s$. In view of~\eq{beta2s}, $\b^2_s = -\b_{t-s}$, then~\eq{qis=qs} gives the identity
$$
-\b_{t-s} + \b_t \equiv \b_s - \b_0.
$$
Differentiating twice w.r.t. the equation of generalized pendulum~\eq{pend_r}, we obtain
$$
c_{t-s} \equiv c_s \then - r \sin \b_{t-s} \equiv r \sin \b_r \then
(\b_s \equiv \b_0 \text{ or } \b_{t-s} \equiv - \b_s \text{ or } r = 0).
$$
If $\b_s \equiv \b_0$, then $c_s \equiv 0$, which means that $\lam \in N_4 \cup N_5 \cup N_7$. If $\b_{t-s} \equiv - \b_s$, then $(\b^2_s, c^2_s) \equiv (\b_s,c_s)$, thus $\lam^2 = \lam$. Finally, the equality $r=0$ means that $\lam \in N_6 \cup N_7$. So we proved that
$$
q^2_s \equiv q_s \then (\lam^2 = \lam \text{ or } \lam \in \cup_{i=4}^7 N_i).
$$
But if $\lam \in N_4 \cup N_5 \cup N_7$, then $\b_s \equiv 0 \text{ or } \pi$, $c_s \equiv 0$ (see Subsec.~\ref{subsec:integr_vert}), and equality~\eq{beta2s} implies that $(\b^2_s, c^2_s) = (\b_s,c_s)$, thus $\lam^2 = \lam$. The implication $\then$ in item (2) follows. The reverse implication is checked directly.

(3) Let $q_s^3 \equiv q_s$. Equality~\eq{beta3s} gives $\b^3_s = -\b_s$, and condition~\eq{qis=qs} implies that $\b_s \equiv \b_0$. Then $c_s \equiv 0$. Consequently, $\lam \in N_4 \cup N_5 \cup N_7$. But if $\lam \in N_4 \cup N_5 \cup N_7$, then $\lam^3 = \lam$ by the argument used above in the proof of item (2). The implication $\then$ in item (3) follows. The reverse implication in item (3) is checked directly.
\end{proof}

Proposition~\ref{propos:qisequivqs} means that the identity $q_s^i \equiv q_s$ is satisfied in the following cases:
\begin{itemize}
\item[(a)]
$\lam^i = \lam$, the trivial case, or
\item[(b)]
$\lam \in N_6$ for $i=2$.
\end{itemize}

\subsection{Multiple points of the exponential mapping}
In this subsection we study solutions to the equations $q^i_t = q_t$ related to the Maxwell strata $\MAX_t^i$~\eq{MAXit}.

Recall that in Subsec.~\ref{subsec:epsiM} we defined the action of reflections $\eps^i$ in the state space $M$. We denote $q^i = \eps^i(q)$, $q, q^i \in M$.

The following functions are defined on $M = \R^2_{x,y} \times S^1_{\t}$ up to sign:
\begin{align*}
&P = x \sin \frac{\t}{2} - y \cos \frac{\t}{2}, \\
&Q = x \cos \frac{\t}{2} + y \sin \frac{\t}{2},
\end{align*}
although their zero sets $\{P = 0\}$ and $\{Q = 0\}$ are well-defined.

\begin{proposition}
\label{propos:MAXMgen}
\begin{itemize}
\item[$(1)$]
$q^1 = q \iff \t = 0 \pmod{2 \pi}$,
\item[$(2)$]
$q^2 = q \iff P = 0$,
\item[$(3)$]
$q^3 = q \iff (y = 0 \text{ and } \t = 0 \pmod{\pi})$.
\end{itemize}
\end{proposition}
\begin{proof}
We apply the formulas for action of reflections $\eps^i$ in $M$ obtained in Subsec.~\ref{subsec:epsiM}.

(1) Formula~\eq{eps1M} means that
$$
\mapto{\eps^1}{(\t,x,y)}{(-\t, x \cos \t + y \sin \t, -x \sin \t + y \cos \t)},
$$
which gives statement (1).

(2) Formula~\eq{eps2M} reads
$$
\mapto{\eps^2}{q = (\t, x, y)}{q^2 = (\t, x \cos \t + y \sin \t, x \sin \t - y \cos \t)}.
$$
If $(x,y) = (0,0)$, then $q^2 = (\t, 0, 0) = q$ and $P = 0$, thus statement (2) follows.

Suppose that $(x,y) \neq (0,0)$, then we can introduce polar coordinates:
$$
x = \rho \cos \chi, \qquad y = \rho \sin \chi,
$$
with $\rho > 0$. We have:
\begin{align*}
q^2 = q
&\iff
\begin{cases}
x \cos \t + y \sin \t = x \\
x \sin \t - y \cos \t = y
\end{cases}
\iff
\begin{cases}
\cos \chi \cos \t + \sin \chi \sin \t = \cos \chi \\
\cos \chi \sin \t - \sin \chi \cos \t = \sin \chi
\end{cases}
\\
&\iff
\begin{cases}
\cos (\t-\chi)  = \cos \chi \\
\sin (\t - \chi)  = \sin \chi
\end{cases}
\iff \t - \chi = \chi
\!\!\! \iff \!\!\! \sin\left(\chi - \frac{\t}{2}\right) = 0
\\
&\iff \cos \chi \sin \frac{\t}{2} - \sin \chi \cos \frac{\t}{2} = 0
\iff P = 0,
\end{align*}
and statement (2) is proved also in the case $(x,y) \neq (0,0)$.

(3) Formula~\eq{eps3M} reads
$$
\mapto{\eps^3}{q = (\t,x,y)}{q^3 = (-\t,x,-y)},
$$
thus
$$
q^3 = q
\iff
\begin{cases}
\t = -\t \\
y = -y
\end{cases}
\iff
\begin{cases}
\t = 0 \pmod{\pi} \\
y = 0.
\end{cases}
$$
\end{proof}

Notice the visual meaning of the conditions $q_t^i = q_t$ for the corresponding arcs of Euler elasticae $(x_s,y_s)$, $s \in [0, t]$ in the case $x_t^2 + y_t^2 \neq 0$. As above, introduce the polar coordinates
$$
x_t = \rho_t \cos \chi_t, \qquad y_t = \rho_t \sin \chi_t
$$
with $\rho_t > 0$.

The condition
$$
q^1_t = q_t
\iff
\t_t = 0 = \t_0
\iff
\chi_t - \t_0 = \chi_t - \t_t
$$
means that the elastic arc has the same slope at its endpoints. The configuration corresponding to the inclusion $q_t \in M^1$, where
$$
M^1 = \{q\in M \mid q^1 = q \}
$$
is shown at Fig.~\ref{fig:eps1Mvisual}.

\twofiglabel
{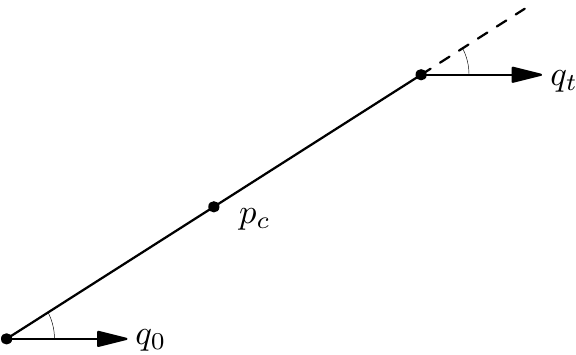}{$q_t \in M^1$}{fig:eps1Mvisual}
{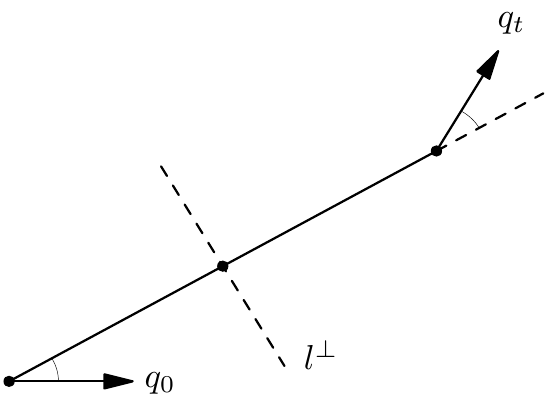}{$q_t \in M^2$}{fig:eps2Mvisual}

The condition
$$
q^2_t = q_t
\iff
P_t = 0 
\iff
\chi_t - \t_0 = \t_t - \chi_t
$$
means that the angle between the elastic arc and the elastic chord connecting $(x_0,y_0)$ to $(x_t,y_t)$ reverses its sign. The configuration corresponding to the inclusion $q_t \in M^2$, where
$$
M^2 = \{q\in M \mid q^2 = q \}
$$
is shown at Fig.~\ref{fig:eps2Mvisual}.

Finally, we have 
$$
q_t^3 = q_t
\iff
\orr{
y=0 \text{ and } \t = 0 \pmod{2 \pi}}
{y=0 \text{ and } \t = \pi \pmod{2 \pi}.}
$$
Thus the set
$$
M^3 = \{ q \in M \mid q^3 = q\}
$$
has two connected components
\begin{align*}
&M^{3+} = \{ q \in M \mid y = 0, \ \t = 0 \pmod{2 \pi} \}, \\
&M^{3-} = \{ q \in M \mid y = 0, \ \t = \pi \pmod{2 \pi} \}, \\
&M^3 = M^{3+} \cup M^{3-}, \qquad M^{3+} \cup M^{3-} = \emptyset.
\end{align*}
See the illustrations to the inclusions $q_t \in M^{3+}$, $q_t \in M^{3-}$ at Figs.~\ref{fig:eps3+Mvisual},  \ref{fig:eps3-Mvisual} respectively.

\bigskip\bigskip

\twofiglabelsize
{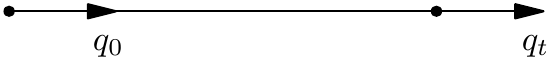}{$q_t \in M^{3+}$}{fig:eps3+Mvisual}
{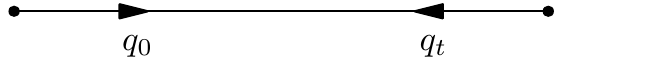}{$q_t \in M^{3-}$}{fig:eps3-Mvisual}
{0.44}{0.5}

It is easy to describe the global structure of the sets $M^i$. The set $M^1 = \{q = (x,y,\t) \mid \t = 0\}$  is a two-dimensional plane. It is the unique 2-dimensional Lie subgroup in the Lie group $\E(2)$ --- the group of parallel translations of the two-dimensional plane $\R^2_{x,y}$. The set $M^2 = \{q = (x,y,\t) \mid P = 0\}$ is the M\"obius strip. Finally, $M^{3+}$ and $M^{3-}$ are straight lines. Notice that
\begin{gather*}
M^{3+} = M^1 \cap M^2, \\
M^{3-} \cap M^1 = \emptyset, \qquad
M^{3-} \cap M^2 = \{(\t = \pi, x = 0, y = 0)\}.
\end{gather*}


\subsection[Fixed points of reflections in the preimage of the exponential mapping]
{Fixed points of reflections \\ in the preimage of the exponential mapping}
In order to describe fixed points of the reflections $\map{\eps^i}{N}{N}$, we use elliptic coordinates $(k,\f,r)$ in $N$
introduced in Subsec.~\ref{subsec:ell_coordN123}. Moreover, the following coordinate will prove very convenient:
$$
\tau = \frac{\sqrt r (\f_t + \f)}{2}.
$$
While the values $\sqrt r \f$ and $\sqrt r    \f_t$ correspond to the initial and terminal points of an elastic arc, their arithmetic mean $\tau$ corresponds to the midpoint of the elastic arc.

\begin{proposition}
\label{propos:nui=nuN1}
Let $\nu = (k,\f,r) \in N_1$, then $\nu^i = \eps^i(\nu) = (k,\f^i,r) \in N_1$. Moreover:
\begin{itemize}
\item[$(1)$]
$\nu^1 = \nu \iff \cn \tau = 0$,
\item[$(2)$]
$\nu^2 = \nu \iff \sn \tau = 0$,
\item[$(3)$]
$\nu^3 = \nu$ is impossible.
\end{itemize}
\end{proposition}
\begin{proof}
We apply Propos.~\ref{propos:eps123N}. The inclusion $\nu^i \in N_1$  holds. Further,
\begin{align*}
\nu^1 = \nu
&\iff \f^1 = \f \iff \f + \f_t = \frac{2 K}{\sqrt r} \left( \mod{\frac{4K}{\sqrt r}}\right)\\
&\iff \tau = K \pmod{ 2 K} \iff \cn \tau = 0,
\end{align*}
\begin{align*}
\nu^2 = \nu
&\iff \f^2 = \f \iff \f + \f_t = 0 \left( \mod{\frac{4K}{\sqrt r}}\right)\\
&\iff \tau = 0 \pmod{ 2 K} \iff \sn \tau = 0,
\end{align*}
\begin{align*}
\nu^3 = \nu
&\iff \f^3 = \f \iff 0 = \frac{2 K}{\sqrt r} \left( \mod{\frac{4K}{\sqrt r}}\right)
\text{ which is impossible.}
\end{align*}
\end{proof}

Notice the visual meaning of the fixed points of the reflections $\map{\eps^i}{N_1}{N_1}$ for the standard pendulum~\eq{pend_st} in the cylinder $(\b,c)$, and for the corresponding inflectional elasticae.

The equality $\cn \tau = 0$ is equivalent to $c = 0$, these are inflection points of elasticae (zeros of their curvature $c$), see Fig.~\ref{fig:cntau0cb}, \ref{fig:cntau0xy}.

\twofiglabel
{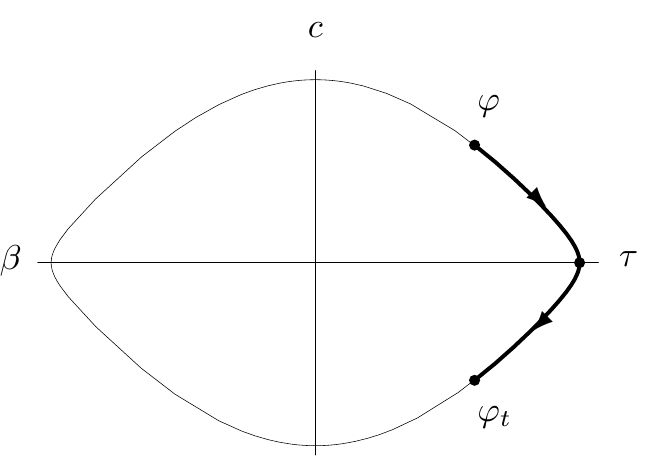}{$\cn \tau = 0$, $\nu \in N_1$}{fig:cntau0cb}
{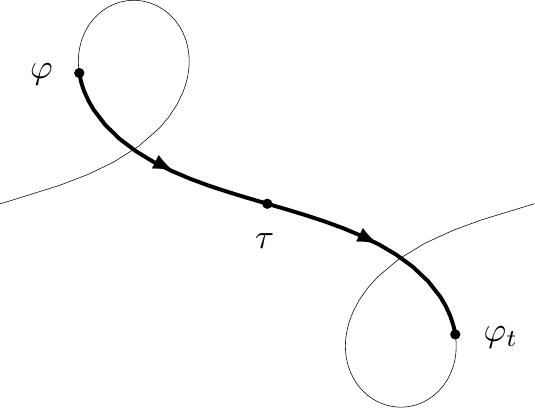}{Inflectional elastica centered at inflection point}{fig:cntau0xy}

The equality $\sn \tau = 0$ is equivalent to $\b = 0$, these are vertices of elasticae (extrema of their curvature $c$), see Fig.~\ref{fig:sntau0cb}, \ref{fig:sntau0xy}.

\twofiglabel
{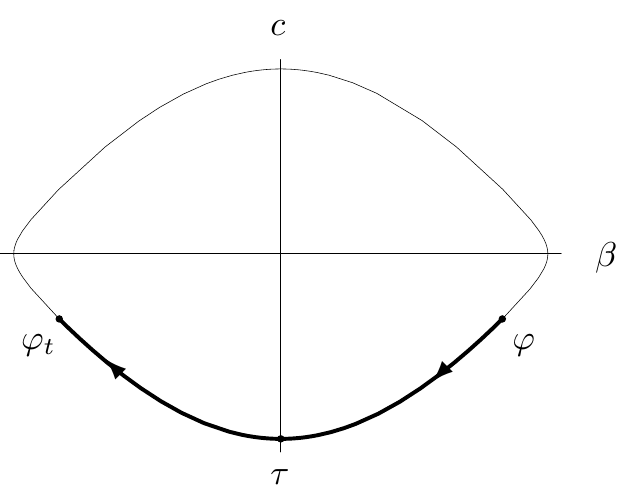}{$\sn \tau = 0$, $\nu \in N_1$}{fig:sntau0cb}
{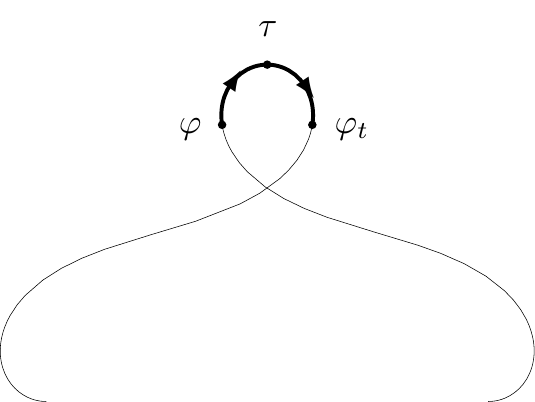}{Inflectional elastica centered at vertex}{fig:sntau0xy}

In the domain $N_2$, we use the convenient coordinate
$$
\tau = \frac{\sqrt r (\psi + \psi_t)}{2}
$$
corresponding to the midpoint of a non-inflectional elastic arc.

\begin{proposition}
\label{propos:nui=nuN2}
Let $\nu = (k,\p,r) \in N_2$, then $\nu^i = \eps^i(\nu) = (k,\p^i,r) \in N_2$. Moreover:
\begin{itemize}
\item[$(1)$]
$\nu^1 = \nu$ is impossible,
\item[$(2)$]
$\nu^2 = \nu \iff \sn \tau \cn \tau = 0$,
\item[$(3)$]
$\nu^3 = \nu$ is impossible.
\end{itemize}
\end{proposition}
\begin{proof}
We apply Propos.~\ref{propos:eps123N}. The inclusion $\nu^i \in N_2$  holds.
Implication~\eq{nuinN2+-} yields statements (1) and (3). We prove statement (2):
\begin{align*}
\nu^2 = \nu
&\iff \p^2 = \p \iff \p + \p_t = 0 \left( \mod{\frac{2K}{\sqrt r}}\right)\\
&\iff \tau = 0 \pmod{ K} \iff \ts \tc = 0.
\end{align*}
\end{proof}

Notice the visual meaning of the fixed points of the reflections $\map{\eps^i}{N_2}{N_2}$.
The equality $\sn \tau \cn \tau = 0$ is equivalent to the equalities $\b = 0 \pmod{\pi}$, $|c| = \max, \ \min$,  these are vertices of non-inflectional elasticae (local extrema of their curvature $c$), see Figs.~\ref{fig:sntau0cbmax}--\ref{fig:cntau0xymin}.

\twofiglabel
{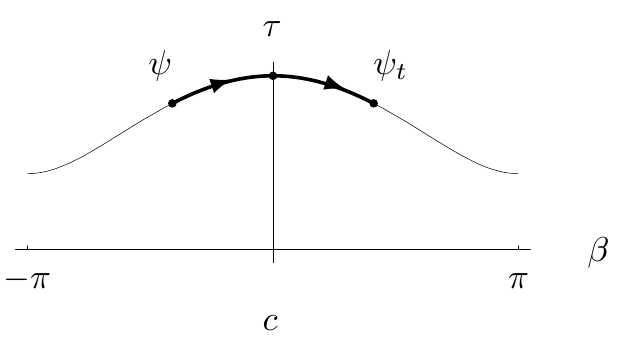}{$\sn {\tau} = 0$, $|c| = \max$, $\nu \in N_2$}{fig:sntau0cbmax}
{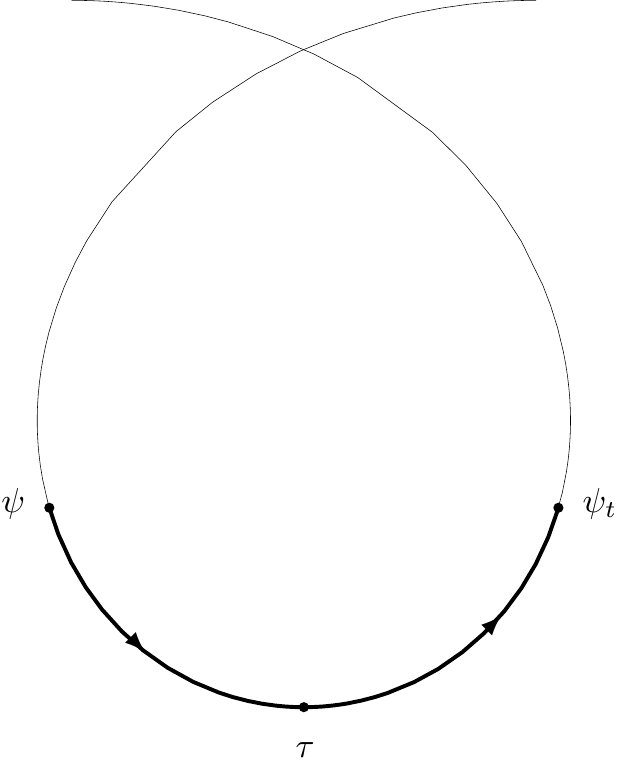}{Non-inflectional elastica centered at vertex}{fig:sntau0xymax}

\twofiglabel
{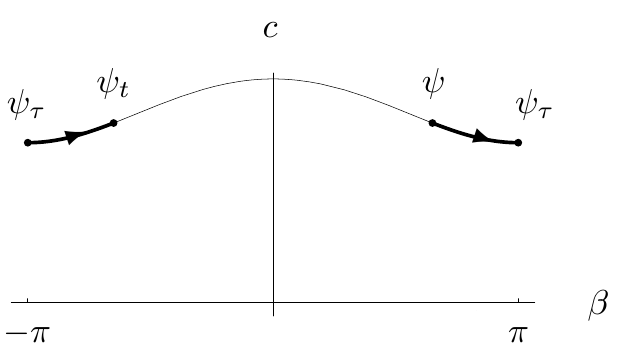}{$\sn {\tau} = 0$, $|c| = \min$, $\nu \in N_2$}{fig:cntau0cbmin}
{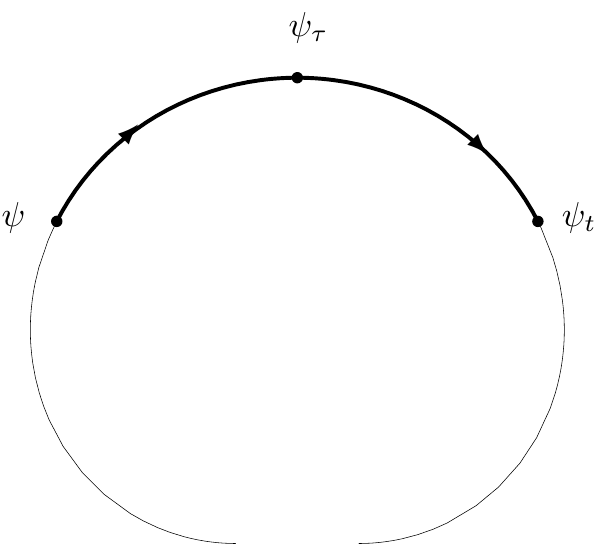}{Non-inflectional elastica centered at vertex}{fig:cntau0xymin}

Similarly to the previous cases, in the set $N_3$ we use the parameter
$$
\tau = \frac{\sqrt r (\f_t + \f)}{2}.
$$

\begin{proposition}
\label{propos:nui=nuN3}
Let $\nu = (\f,r) \in N_3$, then $\nu^i = \eps^i(\nu) = (\f^i,r) \in N_3$. Moreover:
\begin{itemize}
\item[$(1)$]
$\nu^1 = \nu$ is impossible,
\item[$(2)$]
$\nu^2 = \nu \iff \tau = 0$,
\item[$(3)$]
$\nu^3 = \nu$ is impossible.
\end{itemize}
\end{proposition}
\begin{proof}
Follows exactly as in Propos.~\ref{propos:nui=nuN2}.
\end{proof}

The visual meaning of fixed points of reflection $\map{\eps^2}{N_3}{N_3}$: the equality $\tau = 0$ means that $\b = 0$, $|c| = \max$, these are vertices of critical elasticae, see Figs.~\ref{fig:tau0bc}, \ref{fig:tau0xy}.

\twofiglabel
{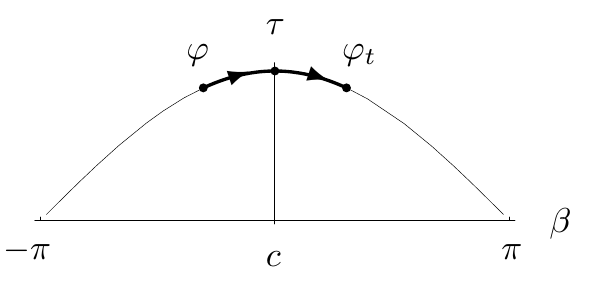}{${\tau} = 0$, $\nu \in N_3$}{fig:tau0bc}
{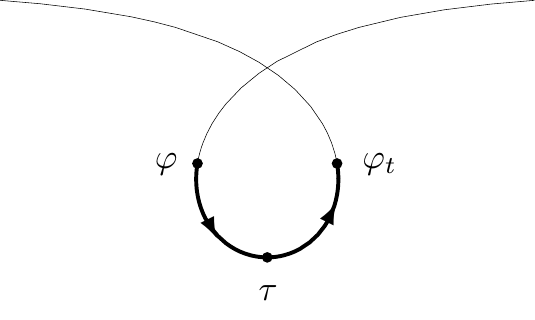}{Critical elastica centered at vertex }{fig:tau0xy}

\begin{proposition}
\label{propos:nui=nuN6}
Let $\nu = (\b,c,r) \in N_6$, then $\nu^i = \eps^i(\nu) = (\b^i,c^i,r) \in N_6$. Moreover:
\begin{itemize}
\item[$(1)$]
$\nu^1 = \nu$ is impossible,
\item[$(2)$]
$\nu^2 = \nu \iff 2 \b + c t = 0 \pmod{2 \pi}$,
\item[$(3)$]
$\nu^3 = \nu$ is impossible.
\end{itemize}
\end{proposition}
\begin{proof}
Items (1), (3) follow from implication~\eq{N6+--+}. Item (2) follows from~\eq{beta2c2N6} and the formula $\b_t = \b_0 + ct$, see Subsec.~\ref{subsec:integr_vert}.
\end{proof}

\subsection{General description of the Maxwell strata generated by reflections}
Now we summarize our computations of Maxwell strata corresponding to reflections.

\begin{theorem}
\label{th:MAX_gen}
\begin{itemize}
\item[$(1)$]
Let $\nu = (k,\f,r) \in N_1$. Then:
\begin{itemize}
\item[$(1.1)$]
$
\nu \in \MAX_t^1 \iff
\begin{cases}
\nu^1 \neq \nu \\
q^1_t = q_t
\end{cases}
\iff
\begin{cases}
\cn \tau  \neq 0 \\
\t_t = 0
\end{cases}
$
\item[$(1.2)$]
$
\nu \in \MAX_t^2 \iff
\begin{cases}
\nu^2 \neq \nu \\
q^2_t = q_t
\end{cases}
\iff
\begin{cases}
\sn \tau  \neq 0 \\
P_t = 0
\end{cases}
$
\item[$(1.3)$]
$
\nu \in \MAX_t^3 \iff
\begin{cases}
\nu^3 \neq \nu \\
q^3_t = q_t
\end{cases}
\iff
\begin{cases}
y_t  = 0 \\
\t_t = 0 \text{ or } \pi.
\end{cases}
$
\end{itemize}
\item[$(2)$]
Let $\nu = (k,\p,r) \in N_2$. Then:
\begin{itemize}
\item[$(2.1)$]
$
\nu \in \MAX_t^1 \iff
\begin{cases}
\nu^1 \neq \nu \\
q^1_t = q_t
\end{cases}
\iff
\t_t = 0
$
\item[$(2.2)$]
$
\nu \in \MAX_t^2 \iff
\begin{cases}
\nu^2 \neq \nu \\
q^2_t = q_t
\end{cases}
\iff
\begin{cases}
\sn \tau \cn \tau  \neq 0 \\
P_t = 0
\end{cases}
$
\item[$(2.3)$]
$
\nu \in \MAX_t^3 \iff
\begin{cases}
\nu^3 \neq \nu \\
q^3_t = q_t
\end{cases}
\iff
\begin{cases}
y_t  = 0 \\
\t_t = 0 \text{ or } \pi.
\end{cases}
$
\end{itemize}
\item[$(3)$]
Let $\nu = (\f,r) \in N_3$. Then:
\begin{itemize}
\item[$(3.1)$]
$
\nu \in \MAX_t^1 \iff
\begin{cases}
\nu^1 \neq \nu \\
q^1_t = q_t
\end{cases}
\iff
\t_t = 0
$
\item[$(3.2)$]
$
\nu \in \MAX_t^2 \iff
\begin{cases}
\nu^2 \neq \nu \\
q^2_t = q_t
\end{cases}
\iff
\begin{cases}
 \tau  \neq 0 \\
P_t = 0
\end{cases}
$
\item[$(3.3)$]
$
\nu \in \MAX_t^3 \iff
\begin{cases}
\nu^3 \neq \nu \\
q^3_t = q_t
\end{cases}
\iff
\begin{cases}
y_t  = 0 \\
\t_t = 0 \text{ or } \pi.
\end{cases}
$
\end{itemize}
\item[$(4)$]
$\MAX_t^i \cap N_j = \emptyset$ for $i = 1, 2, 3$, $j = 4, 5, 7$.
\item[$(6)$]
Let $\nu \in N_6$. Then:
\begin{itemize}
\item[$(6.1)$]
$
\nu \in \MAX_t^1 \iff
\begin{cases}
\nu^1 \neq \nu \\
q^1_t = q_t
\end{cases}
\iff
\t_t = 0
$
\item[$(6.2)$]
$
\nu \in \MAX_t^2
$ is impossible
\item[$(6.3)$]
$
\nu \in \MAX_t^3 \iff
\begin{cases}
\nu^3 \neq \nu \\
q^3_t = q_t
\end{cases}
\iff
\begin{cases}
y_t  = 0 \\
\t_t = 0 \text{ or } \pi.
\end{cases}
$
\end{itemize}
\end{itemize}
\end{theorem}
\begin{proof}
In view of the remark after definition of the Maxwell strata~\eq{MAXit} and Propos.~\ref{propos:qisequivqs}, we have
\begin{align}
&\MAX^i_t = \{ \nu \in N \mid \nu^i \neq \nu, \ q^i_t = q_t\}, \quad i = 1, 3, \nonumber \\
&\MAX^2_t \cap N_j = \{ \nu \in N_j \mid \nu^2 \neq \nu, \ q^2_t = q_t\}, \quad j \neq 6, \nonumber \\
&\MAX^2_t \cap N_6 = \emptyset. \label{MAX2tN6}
\end{align}
This proves the first implication in items (1.1)--(3.3). The second implication in these items follows directly by combination of Propositions~\ref{propos:nui=nuN1}, \ref{propos:nui=nuN2}, \ref{propos:nui=nuN3} with Proposition~\ref{propos:MAXMgen}. So items (1)--(3) follow.

In the case $\nu \in N_4 \cup N_5 \cup N_7$ the corresponding extremal trajectory is $(x_s,y_s,\t_s) = (s,0,0)$, which is globally optimal since elastic energy of the straight segment is $J = 0$. By Propos.~\ref{propos:Maxwell}, there are no Maxwell points in this case.

Finally, let $\nu \in N_6$. Items (6.1) and (6.2) follow by combination of Propos.~\ref{propos:nui=nuN6} with Propos.~\ref{propos:MAXMgen}. Item (6.2) was already obtained~\eq{MAX2tN6} from item (2) of Propos.~\ref{propos:qisequivqs}.
\end{proof}

\begin{remark}
Items (1.3), (2.3), (3.3.), (4), (6.3) of Th.~\ref{th:MAX_gen} show that the Maxwell stratum $\MAX_t^3$ admits a decomposition into two disjoint subsets:
\begin{align*}
&\MAX_t^3 = \MAX_t^{3+} \cup \MAX_t^{3-}, 
\qquad
\MAX_t^{3+} \cap \MAX_t^{3-} = \emptyset, \\
&\nu \in \MAX_t^{3+} \iff
\begin{cases}
y_t =0 \\
\t_t = 0,
\end{cases}\\
&\nu \in \MAX_t^{3-} \iff
\begin{cases}
y_t =0 \\
\t_t = \pi.
\end{cases}
\end{align*}
\end{remark}

In order to obtain a complete description of the Maxwell strata $\MAX_t^i$, in the next section we solve the equations that determine these strata and appear in Th.~\ref{th:MAX_gen}.

\section{Complete description of Maxwell strata}
\label{sec:MAX_complete}

\subsection{Roots of equation $\t = 0$}
In this subsection we solve the equation $\t_t = 0$ that determines the Maxwell stratum $\MAX^1_t$, see Th.~\ref{th:MAX_gen}.

We denote by $\orr{A}{B}$ the condition $A \vee B$ contrary to $\begin{cases}A \\ B\end{cases}$, which denotes the condition $A \wedge B$.

\begin{proposition}
\label{propos:theta=0N1}
Let $\nu = (k,\f, r) \in N_1$, then
$$
\t_t = 0 \iff 
\orr{p = 2 K n, \ n \in \Z}{\cn \tau = 0,}
$$
where $p = \ds\frac{\sqrt r (\f_t - \f)}{2}$, $\tau = \ds\frac{\sqrt r (\f_t + \f)}{2}$.
\end{proposition}
\begin{proof}
We have
\begin{align*}
\t_t = 0
&\iff
\b_t = \b_0 \pmod{2 \pi}
\iff \frac{\b_t}{2} = \frac{\b_0}{2} \pmod{\pi}
\\
&\iff
\begin{cases}
\sn(\sqrt r \f_t) = \pm \sn(\sqrt r \f) \\
\dn(\sqrt r \f_t) = \pm \dn(\sqrt r \f)
\end{cases}
\iff
\begin{cases}
\sn(\sqrt r \f_t) =  \sn(\sqrt r \f) \\
\dn(\sqrt r \f_t) =  \dn(\sqrt r \f)
\end{cases}
\\
&\iff\!\!\!
\sn(\sqrt r \f_t) =  \sn(\sqrt r \f)
\!\!\! \iff \!\!\!
\orr{\sqrt r \f_t = \sqrt r \f \pmod{4 \sqrt r K}}
{\sqrt r \f_t = 2 \sqrt r K - \sqrt r \f \pmod{4 \sqrt r K}}
\\
&\iff
\orr{\sn p = 0}{\cn \tau = 0}
\iff
\orr{p = 2 Kn, \quad n \in \Z}{\cn\tau = 0.}
\end{align*}
\end{proof}

\begin{proposition}
\label{propos:theta=0N2}
Let $\nu = (k,\p, r) \in N_2$, then
$$
\t_t = 0 \iff p =  K n, \ n \in \Z,
$$
where $\ds p = \frac{\sqrt r (\p_t - \p)}{2}$.
\end{proposition}
\begin{proof}
Let $\nu \in N_2^+$, then
\begin{align*}
\t_t = 0
&\iff
\frac{\b_t}{2} = \frac{\b_0}{2} \pmod{\pi}
\\
&\iff
\begin{cases}
\sn(\sqrt r \p_t) = \pm \sn(\sqrt r \p) \\
\cn(\sqrt r \p_t) = \pm \cn(\sqrt r \p)
\end{cases} \\
&\iff
\orr{\sqrt r \p_t = \sqrt r \p \pmod{4 \sqrt r K}}
{\sqrt r \p_t = \sqrt r \p + 2 K \pmod{4 \sqrt r K}}
\\
&\iff
\orr{p = 0 \pmod{2K}}{p = K \pmod{2K}}
\iff p = K n , \ n \in \Z.
\end{align*}
If $\nu \in N_2^-$, then the same result is obtained by the inversion $\map{i}{N_2^+}{N_2^-}$.
\end{proof}

\begin{proposition}
\label{propos:theta=0N3}
Let $\nu  \in N_3$, then
$$
\t_t = 0 \iff t = 0.
$$
\end{proposition}
\begin{proof}
Let $\nu \in N_3^+$, then
\begin{align*}
\t_t = 0
&\iff
\frac{\b_t}{2} = \frac{\b_0}{2} \pmod{\pi}
\iff
\begin{cases}
\ds \tanh(\sqrt r \f_t) = \pm \tanh(\sqrt r \f) \\
\ds \frac{1}{\cosh(\sqrt r \f_t)} = \pm \frac{1}{\cosh(\sqrt r \f)}
\end{cases}
\\
&\iff
\begin{cases}
\ds \tanh(\sqrt r \f_t) = \tanh(\sqrt r \f) \\
\ds \frac{1}{\cosh(\sqrt r \f_t)} = \frac{1}{\cosh(\sqrt r \f)}
\end{cases}
\iff
\sqrt r \f_t = \sqrt r \f \iff t = 0.
\end{align*}
The same result is obtained for  $\nu \in N_3^-$ via the inversion $\map{i}{N_3^+}{N_3^-}$.
\end{proof}

\begin{proposition}
\label{propos:theta=0N6}
Let $\nu \in N_6$. Then
$$
\t_t = 0 \iff
ct = 2 \pi n, \quad n \in \Z.
$$
\end{proposition}
\begin{proof}
We have $\t_t = ct$ in the case $\nu \in N_6$.
\end{proof}

\subsection{Roots of equation $P=0$ for $\nu \in N_1$}
\label{subsec:P=0N1}

Using the coordinates
\be{taupN1}
\tau = \frac{\sqrt r (\f_t + \f)}{2} = \sqrt r \left( \f + \frac t2\right), 
\qquad  p = \frac{\sqrt r (\f_t - \f)}{2} = \frac{\sqrt r t}{2},
\ee
or, equivalently,
$$
\sqrt r \f = \tau - p, \qquad \sqrt r \f_t = \tau + p,
$$
and addition formulas for Jacobi's functions (see Sec.~\ref{sec:append}), we obtain the following in the case $\nu \in N_1$:
\begin{align}
\sin \frac{\t_t}{2}
&=
2 k \ss \dd \tc(\tdp + k^2 \ccp \tsp)/\D^2, \nonumber \\
\cos \frac{\t_t}{2}
&=
(\ddp - k^2 \ssp \tcp)(\tdp + k^2 \ccp \tsp)/\D^2, \nonumber
\end{align}
\begin{align}
x_t &=
[2 (2 \E(p) - p)(1 - 2 k^2 \ssp) \nonumber \\
&\quad + 8 k^2 \ss(\cc \dd(2 \E(p) - p) - \ss + k^2 \sst)\tc \ts \td \nonumber \\
&\quad - 4k^2(p(-1+\ssp(1-k^2(\ssp-2))) \nonumber \\
&\quad + 2 \E(p)(1+ \ssp(-1+k^2 (\ssp-2))) \nonumber \\
&\quad + \cc \ss \dd (-1 + k^2 \ssp(1 + \tcp))) \tsp \nonumber \\
&\quad + 2 k^4(2 \E(p)-p)\ssp(\ssp-2)\tsf]/(\sqrt r \D^2), \nonumber
\end{align}
\begin{align*}
y_t &=
4 k [k^2 \cc(2 \E(p) - p) \ssp \tcp \ts \td  \\
&\quad + \cc \ddp \ts ((p-2 \E(p)) \td + 2 k^2 \ssp \tc \ts) \\
&\quad + \dd \ss (2 \E(p) \tc (1 + k^2(\ssp - 2) \tsp) \\ 
&\quad - p \tc (1 + k^2(\ssp -2) \tsp) \\
&\quad + \td \ts (1 + k^2 \ssp (\tsp - 2)))]/(\sqrt r \D^2),
\end{align*}
\begin{align}
&\D = 1 - k^2 \ssp \tsp, \nonumber \\
&P_t = \frac{4 k \ts \td f_1(p,k)}{\sqrt r \D}, \qquad \nu \in N_1, \label{PtN1} \\
&f_1(p,k) = \ss \dd - (2 \E(p) - p) \cc. \nonumber
\end{align}

In order to describe roots of the equation $f_1(p) = 0$, we need the following statements.

We denote by
$E(k)$ and $K(k)$ the complete elliptic integrals of the first and second kinds respectively, see Sec.~\ref{sec:append}.

\begin{proposition}[Lemma 2.1 \cite{max3}]
\label{propos:lem21max3}
The equation
$$
2E(k) - K(k) = 0, \qquad k \in [0, 1),
$$
has a unique root
$
k_0 \in (0, 1)$.
Moreover,
\begin{align*}
k \in [0, k_0) \ &\Rightarrow \ 2E - K > 0, \\
k \in (k_0, 1) \ &\Rightarrow \ 2E - K < 0.
\end{align*}
\end{proposition}

Notice that for $k = \frac{1}{\sqrt 2}$ we have
\be{K1/sqrt2}
K = \frac{1}{4 \sqrt \pi} \left( \Gamma\left(\frac{1}{4}\right)\right)^2, \qquad
E = \frac{2 K^2 + \pi}{4 K}
\then
2 E - K = \frac{\pi}{2K} > 0,
\ee
see~\cite{lawden}, page 89, Chap. 3, exercise 24. Thus
\be{1/sqrt2k0}
\frac{1}{\sqrt 2} < k_0 < 1.
\ee

The graph of the function
$k \mapsto 2 E(k) - K(k)$ is given at Fig.~\ref{fig20}. Numerical simulations show that
$k_0 \approx 0.909$.

To the value
$k = k_0$ there corresponds the periodic Euler elastic in the form of figure 8, see Fig.~\ref{fig:elastica5}.

\onefiglabel{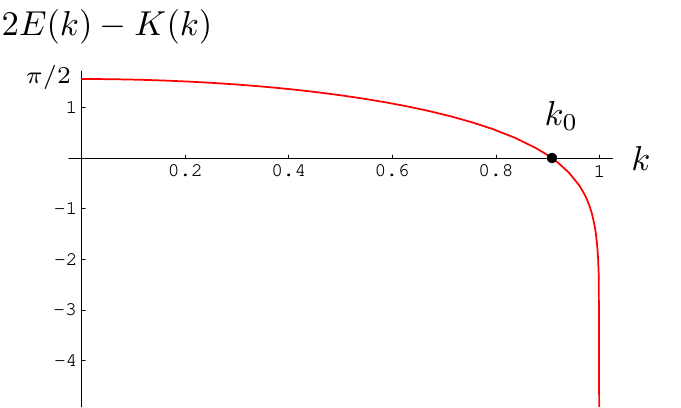}{Definition of~$k_0$}{fig20}


\begin{proposition}[Propos. 2.1 \cite{max3}]
\label{propos:prop21max3}
For any
$k \in [0, 1)$
the function
$$
f_1(p,k) = \ss \dd - (2 \E(p) -  p) \cc
$$
has a countable number of roots
$p_n^1$,  $n \in \Z$.
These roots are odd in
$n$:
$$
p_{-n}^1 = - p_n^1, \qquad n \in \Z,
$$
in particular,
$p_0^1 = 0$. The roots
$p_n^1$ are localized as follows:
$$
p_n^1 \in (-K + 2 K n, \ K + 2 K n), \qquad n \in \Z.
$$
In particular, the roots
$p_n^1$ are monotone in
$n$:
$$
p_{n}^1 < p_{n+1}^1, \qquad n \in \Z.
$$
Moreover, for
$n \in \N$
\begin{align*}
k \in [0, k_0) \ &\Rightarrow \ p_n^1 \in (2K n, K + 2Kn), \\
k = k_0 \ &\Rightarrow \ p_n^1 = 2Kn, \\
k \in (k_0, 1) \ &\Rightarrow \ p_n^1 \in (-K + 2Kn, 2Kn),
\end{align*}
where
$k_0$ is the unique root of the equation
$2E(k) - K(k) = 0$,
see Propos.~$\ref{propos:lem21max3}$.
\end{proposition}

\begin{proposition}[Cor. 2.1 \cite{max3}]
\label{propos:cor21max3}
The first positive root
$p = p_1^1$ of the equation
$f_1(p) = 0$ is localized as follows:
\begin{align*}
k \in [0, k_0) \ &\Rightarrow \ p_1^1 \in (2K, 3K), \\
k = k_0 \ &\Rightarrow \ p_1^1 = 2K, \\
k \in (k_0, 1) \ &\Rightarrow \ p_1^1 \in (K, 2K).
\end{align*}
\end{proposition}

The graph of the function $k \mapsto p_1^1(k)$ is shown at Fig.~\ref{fig25}, and the graph  $k \mapsto p_1^1(k)/K(k)$ is given at Fig.~\ref{fig26}.  Recall that the period of pendulum corresponds to
$p = 2 K$, this value is denoted at the axis of ordinates at Fig.~\ref{fig26}.

\twofiglabel
{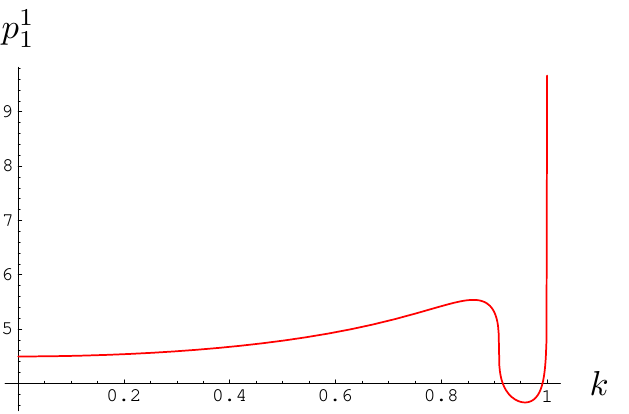}{$k \mapsto p_1^1$, $\nu \in N_1$}{fig25}
{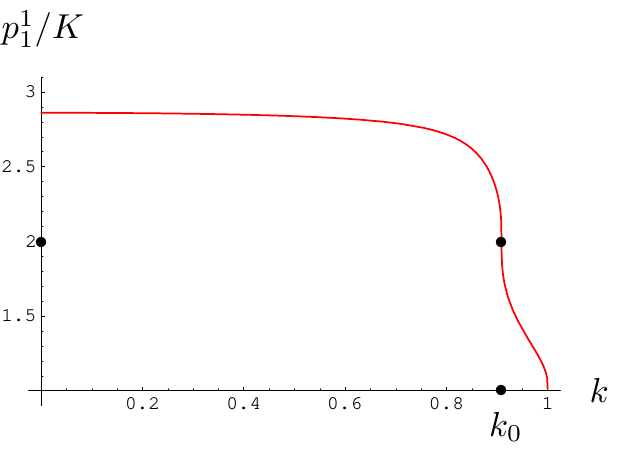}{$k \mapsto p_1^1/K$, $\nu \in N_1$}{fig26}

Now we can obtain the following description of roots of the equation $P_t = 0$ for $\nu \in N_1$.

\begin{proposition}
\label{propos:P=0N1}
Let $\nu \in N_1$. Then:
$$
P_t = 0
\iff
\orr{f_1(p) = 0}{\ts = 0}
\iff
\orr{p = p_n^1, \quad n \in \Z}{\ts = 0.}
$$
\end{proposition}
\begin{proof}
Apply Eq.~\eq{PtN1} and Propos.~\ref{propos:prop21max3}.
\end{proof}

\begin{remark}
We can propose a visual way for evaluating the roots $p_n^1$ to the equation $f_1(p)=0$, see Figs.~\ref{fig:p1k08}, \ref{fig:p1k095}. Given an inflectional elastica, one should take its inflection point $O$, construct tangent lines to the elastica through the point $O$, and denote them $A_1 B_1$, $A_2 B_2$, \dots; the tangents are ordered by the length of the  elastic arcs $\mathrm{l}(A_n B_n)$. Then the number $p_n^1$ corresponds to the length $\mathrm{l}(A_n B_n)$; precisely, $p_n^1 =  \dfrac{\sqrt r t_n}{2}$, $t_n = \mathrm{l}(A_n B_n)$, since elasticae are parametrized by arc length and in view of ~\eq{taupN1}.

On  the arc $A_1B_1$  the pendulum makes more than one oscillation in the case $k < k_0$ (Fig.~\ref{fig:p1k08}), and less than one oscillation in the case $k > k_0$ (Fig.~\ref{fig:p1k095}); thus in the first case $p_1^1 > 2 K(k)$, and in the second case $p_1^1 < 2 K(k)$. This observation provides one more illustration to Propos.~\ref{propos:cor21max3} and Fig.~\ref{fig26}.

\twofiglabel
{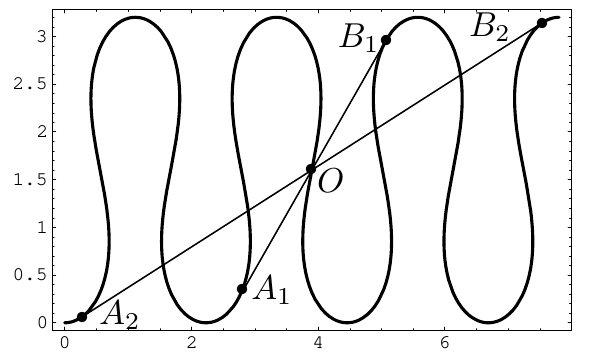}{Computing $p_n^1$ for $k < k_0$}{fig:p1k08}
{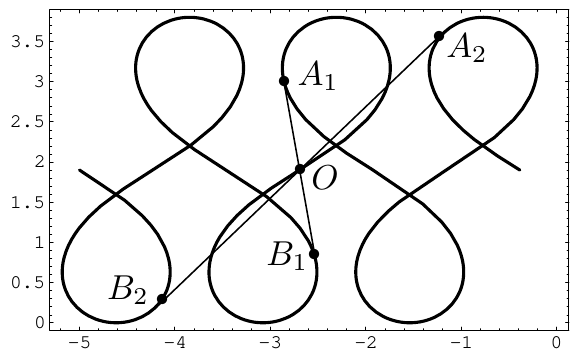}{Computing $p_n^1$ for $k > k_0$}{fig:p1k095}

\end{remark}

\subsection{Roots of equation $P=0$ for $\nu \in N_2$}
\label{subsec:P=0N2}
In order to find the expression for $P_t$ for $\nu \in N_2^+$, we apply the transformation of Jacobi's functions $k \mapsto \frac 1k$, see Subsubsec.~\ref{subsubsec:ell_coordsC2+} and~\eq{k->1/k1}, \eq{k->1/k2}, to equality~\eq{PtN1}:
\begin{align}
&P_t = \frac{4 k_1 \sn(\tau_1, k_1) \dn(\tau_1, k_1) f_1(p_1, k_1)}
{\sqrt r (1 - k_1^2 \sn^2(p_1, k_1) \sn^2(\tau_1, k_1))},
\qquad \nu \in N_1, \quad k_1 \in(0,1),
\label{PtN1k1} \\
&\tau_1 = \frac{\sqrt{r} (\f_t + \f)}{2}, \qquad
p_1 = \frac{\sqrt{r} (\f_t - \f)}{2}. \nonumber
\end{align}
The both sides of equality~\eq{PtN1k1} are analytic single-valued functions of the elliptic coordinates $(k_1, \f, r)$, so this equality is preserved after analytic continuation to the domain $k_1 \in (1,+ \infty)$, i.e., $\nu \in N_2^+$.

Denote $k_2 = \frac{1}{k_1}$, then the formulas for the transformation $k \mapsto \frac 1k$ of Jacobi's functions~\eq{k->1/k1},  \eq{k->1/k2} give the following:
\begin{align*}
P_t
&=
\frac{4 \frac{1}{k_2} \sn(\tau_1, \frac{1}{k_2}) \cn(\tau_1, \frac{1}{k_2}) f_1(p_1, \frac{1}{k_2})}
{\sqrt r (1 - \frac{1}{k_2^2} \sn^2(p_1, \frac{1}{k_2}) \sn^2(\tau_1, \frac{1}{k_2}))}\\
&=
\frac{4 \frac{1}{k_2} k_2 \sn(\frac{\tau_1}{k_2} , k_2) \cn(\frac{\tau_1}{k_2}, k_2) f_2(p_2, k_2)}
{\sqrt r (1 - \frac{1}{k_2^2} k_2^2 \sn^2(\frac{p_1}{k_2}, k_2) k_2^2 \sn^2(\frac{\tau_1}{k_2}, k_2))} \\
&=
\frac{4 \sn(\tau_2, k_2) \cn(\tau_2, k_2) f_2(p_2, k_2)}
{\sqrt r (1 - k_2^2 \sn^2(p_2, k_2) \sn^2(\tau_2, k_2))},
\end{align*}
where
\begin{align}
\tau_2 &
= \frac{\tau_1}{k_2} = \frac{\sqrt r (\f_t + \f)}{2 k_2} = \frac{\sqrt r (\p_t + \p)}{2},
\label{tau2tau1} \\
p_2
&= \frac{p_1}{k_2} = \frac{\sqrt r (\f_t - \f)}{2 k_2} = \frac{\sqrt r (\p_t - \p)}{2},
\label{p2p1}
\end{align}
and
\begin{align*}
f_2(p_2,k_2)
&=
\ds f_1\left(p_1, \frac{1}{k_2}\right) \\
&=
\sn\left(p_1, \frac{1}{k_2}\right) \dn\left(p_1, \frac{1}{k_2}\right)  - 
\left(2 \E\left(p_1, \frac{1}{k_2}\right) - p_1\right) \cn\left(p_1, \frac{1}{k_2}\right)  \\
&=
k_2 \sn\left(\frac{p_1}{k_2}, k_2\right) \cn\left(\frac{p_1}{k_2}, k_2\right)  \\
&\qquad\qquad - \left(\frac{2}{k_2} \E\left(\frac{p_1}{k_2}, k_2\right) - 2 \frac{1-k_2^2}{k_2^2} p_1 - p_1 \right) \dn\left(\frac{p_1}{k_2}, k_2\right) \\
&=
\frac{1}{k_2}[ k_2^2 \sn(p_2, k_2) \cn(p_2, k_2)  + \dn(p_2, k_2)( (2-k_2^2) p_2 - 2 \E(p_2, k_2))].
\end{align*}
Summing up, we have
\begin{align}
&P_t =
\frac{4 \ts \tc f_2(p,k)}{\sqrt r \D},
\qquad \nu \in N_2^+, \label{PtN2+} \\
&f_2(p,k) =
\frac{1}{k}[ k^2 \ss  \cc  + \dd ((2-k^2) p - 2 \E(p))], \nonumber \\
&\tau = \frac{\sqrt r (\p_t + \p)}{2}, \qquad
p = \frac{\sqrt r (\p_t - \p)}{2},
\qquad \D = 1 - k^2 \ssp \tsp. \nonumber
\end{align}

We will need the following statement.

\begin{proposition}[Proposition 3.1 \cite{max3}]
\label{propos:propos31max3}
The function $f_2(p)$ has no roots $p \neq 0$.
\end{proposition}

\begin{proposition}
\label{propos:P=0N2}
Let $\nu \in N_2$. Then
$$
P_t = 0
\iff
\orr{p = 0}{\ts \tc = 0.}
$$
\end{proposition}
\begin{proof}
In the case $\nu \in N_2^+$, we obtain from~\eq{PtN2+} and Propos.~\ref{propos:propos31max3}:
$$
P_t = 0
\iff
\orr{f_2(p) = 0}{\ts \tc = 0}
\iff
\orr{p = 0}{\ts \tc = 0.}
$$

The case $\nu \in N_2^-$ is obtained by the inversion $\map{i}{N_2^+}{N_2^-}$. The inversion $i$ maps as follows:
\begin{gather*}
(\b, c, r) \mapsto (-\b, -c, r), \qquad
(\t, x, y) \mapsto (-\t, x, -y),\\
P \mapsto - P, \qquad (k,\f) \mapsto (k,\f), \qquad (\tau, p) \mapsto (\tau,p),
\end{gather*}
thus equality~\eq{PtN2+} yields
\be{PtN2-}
P_t =
- \frac{4 \ts \tc f_2(p,k)}{\sqrt r \D},
\qquad \nu \in N_2^-,
\ee
and the statement for the case $\nu \in N_2^-$ follows.
\end{proof}

\subsection{Roots of equation $P=0$ for $\nu \in N_3$}
Passing to the limit $k \to 1 - 0$ in equalities~\eq{PtN2+}, \eq{PtN2-}, we obtain the following:
\begin{align}
&P_t =
\pm \frac{4 \tanh \tau f_2(p,1)}{\sqrt r \cosh \tau (1 - \tanh^2 p \tanh^2 \tau)}, \qquad
\nu \in N_3^{\pm}, \label{PtN3} \\
&p = \frac{\sqrt r (\f_t - \f)}{2}, \qquad
\tau = \frac{\sqrt r (\f_t + \f)}{2}, \nonumber \\
&f_2(p,1) = \lim_{k \to 1 - 0} f_2(p,k) = \frac{2p - \tanh p}{\cosh p}.
\end{align}

\begin{proposition}
\label{propos:P=0N3}
Let $\nu \in N_3$. Then
$$
P_t = 0 \iff
\orr{p=0}{\tau = 0.}
$$
\end{proposition}
\begin{proof}
We have $(2p - \tanh p)' = 2 - \frac{1}{\cosh^2 p} > 1$, thus
$$
f_2(p,1) = 0 \iff
2 p - \tanh p = 0 \iff
p = 0,
$$
and the statement follows from~\eq{PtN3}.
\end{proof}

\subsection{Roots of equation $P = 0$ for $\nu \in N_6$}

\begin{proposition}
\label{propos:P=0N6}
If $\nu \in N_6$, then $P_t \equiv 0$.
\end{proposition}
\begin{proof}
$\ds P_t = x_t \sin \frac{\t_t}{2} - y_t \cos \frac{\t_t}{2} = \frac 1c \sin ct \sin \frac{ct}{2} - \frac 1c (1 - \cos ct) \cos \frac{ct}{2} \equiv 0$.
\end{proof}

The visual meaning of this proposition is simple: an arc of a circle has the same angles with its chord at the initial and terminal points.

\subsection{Roots of system $y = 0$, $\t = 0$}
Notice that
$$
\begin{cases}
\t_t = 0 \\ y_t = 0
\end{cases}
\iff
\begin{cases}
\t_t = 0 \\ P_t  = x_t \sin \frac{\t_t}{2} - y_t \cos \frac{\t_t}{2} = 0,
\end{cases}
$$
so we can replace the first system by the second one and use our previous results for equations $\t_t = 0$ and $P_t = 0$.

\begin{proposition}
\label{propos:y=0theta=0N1}
Let $\nu \in N_1$. Then
$$
\begin{cases}
\t_t = 0 \\ P_t   = 0
\end{cases}
\iff
\orrr{k = k_0, \ p = 2 Kn,}
{p = p^1_n, \ \tc = 0,}
{p = 2 Kn, \ \ts = 0,}
\quad n \in \Z.
$$
\end{proposition}
\begin{proof}
By virtue of Propos.~\ref{propos:theta=0N1}, \ref{propos:P=0N1}, we have
\begin{align}
&\begin{cases}
\t_t = 0 \\ P_t   = 0
\end{cases}
\iff
\begin{cases}
p = 2 Km \text{ or } \tc = 0  \\ p = p^1_n \text{ or } \ts = 0
\end{cases}
\\
&\quad \iff
\begin{cases}
p = 2 Km   \\ p = p^1_n
\end{cases}
\text{ or }
\begin{cases}
 \tc = 0  \\ p = p^1_n
\end{cases}
\text{ or }
\begin{cases}
p = 2 Km   \\  \ts = 0
\end{cases}
\text{ or }
\begin{cases}
 \tc = 0  \\  \ts = 0.
\end{cases}
\end{align}
By Propos.~\ref{propos:prop21max3},
$$
\begin{cases}
p = 2 Km   \\ p = p^1_n
\end{cases}
\iff
p = p^1_n = 2 Kn
\iff
\begin{cases}
k = k_0 \\ p = 2 Kn.
\end{cases}
$$
Now it remains to notice that the system $\tc = 0$, $\ts = 0$ is incompatible, and the proof is complete.
\end{proof}

\begin{proposition}
\label{propos:y=0theta=0N2}
Let $\nu \in N_2$. Then
$$
\begin{cases}
\t_t = 0 \\ P_t   = 0
\end{cases}
\iff
\orr{p = Kn, \ \tau = Km,}
{p = 0,}
\quad n, \ m  \in \Z.
$$
\end{proposition}
\begin{proof}
Taking into account Propos.~\ref{propos:theta=0N2} and~\ref{propos:P=0N2}, we obtain
$$
\begin{cases}
\t_t = 0 \\ P_t   = 0
\end{cases}
\iff
\begin{cases}
p = Kn \\ p=0 \text{ or } \tau = Km
\end{cases}
\iff
\orr{p = Kn, \ \tau = Km, \text{ or}}
{p = 0.}
$$
\end{proof}

\begin{proposition}
\label{propos:y=0theta=0N3}
Let $\nu \in N_3$. Then
$$
\begin{cases}
\t_t = 0 \\ P_t   = 0
\end{cases}
\iff
t = 0.
$$
\end{proposition}
\begin{proof}
Follows immediately from Propos.~\ref{propos:theta=0N3}, \ref{propos:P=0N3}.
\end{proof}

\begin{proposition}
\label{propos:y=0theta=0N6}
Let $\nu \in N_6$. Then
$$
\begin{cases}
\t_t = 0 \\ P_t   = 0
\end{cases}
\iff
ct = 2 \pi n, \qquad n \in \Z.
$$
\end{proposition}
\begin{proof}
Follows immediately from Propos.~\ref{propos:theta=0N6}, \ref{propos:P=0N6}.
\end{proof}

\subsection{Roots of system $y=0$, $\t = \pi$ for $\nu \in N_1$}
The structure of solutions to the system $y_t = 0$, $\t_t = \pi$ is much more complicated than that of the system $y_t = 0$, $\t_t = 0$ studied above.

First of all, for any normal extremal
\be{chain1}
\begin{cases}
\t_t = 0 \\ y_t   = \pi
\end{cases}
\iff
\begin{cases}
\ds \cos \frac{\t_t}{2} = 0 \\ 
\ds Q_t  = x_t \cos \frac{\t_t}{2} + y_t \sin \frac{\t_t}{2} = 0.
\end{cases}
\ee

From now on we suppose in this subsection that $\nu \in N_1$.

In the same way as at the beginning of Subsec.~\ref{subsec:P=0N1}, in the coordinates $\tau$, $p$ given by~\eq{taupN1} we obtain
\begin{align*}
\cos \frac{\t_t}{2}
&=
(\ddp - k^2 \ssp \tcp)(\tdp + k^2 \ccp \tsp)/\D^2 \\
&= (1 - 2 k^2 \ssp + k^2 \ssp \tsp)(\tdp + k^2 \ccp \tsp)/\D^2, \\
Q_t &=
2 \E(p) - p + k^2 \tsp (2 \cc \ss \dd - (2 \E(p) - p)(2-\ssp)).
\end{align*}
Thus 
$$
\cos \frac{\t_t}{2} =0 \iff \tsp = (2 k^2 \ssp - 1)/(k^2 \tsp).
$$
 Substituting this value for $\tsp$ into $Q_t$, we get rid of the variable $\tau$ in the second equation in~\eq{chain1}:
\begin{align}
&\restr{Q_t}{\tsp = (2 k^2 \ssp - 1)/(k^2 \tsp)} =
\frac{2}{\ssp} g_1(p,k), \nonumber \\
&g_1(p,k) = (1 - k^2 + k^2 \ccf)(2 \E(p) - p) + \cc \ss \dd (2 k^2 \ssp - 1).
\label{g1(p,k)N1}
\end{align}
So we can continue chain~\eq{chain1} as follows:
$$
\begin{cases}
\ds \cos \frac{\t_t}{2} = 0 \\ Q_t  = 0
\end{cases}
\iff
\begin{cases}
\tsp = (2 k^2 \ssp - 1)/(k^2 \tsp) \\ g_1(p,k)  = 0.
\end{cases}
$$
We proved the following statement.

\begin{proposition}
\label{propos:theta=piy=0N1}
Let $\nu \in N_1$. Then
\be{chain2}
\begin{cases}
\t_t = \pi \\ y_t  = 0
\end{cases}
\iff
\begin{cases}
\tsp = (2 k^2 \ssp - 1)/(k^2 \tsp) \\ g_1(p,k)  = 0.
\end{cases}
\ee
\end{proposition}

Now we study solvability of the second system in~\eq{chain2} and describe its solutions in the domain $\{p \in (0, 2 K)\}$. For the study of global optimality of normal extremal trajectories, it is essential to know the first Maxwell point. By Propos.~\ref{propos:theta=0N1}, the first Maxwell point corresponding to $\eps^1$ occurs at $p = 2 K$, so for the study of the global optimal control problem we can restrict ourselves by the domain $\{p \in (0, 2 K)\}$. What concerns the local problem, in the forthcoming paper~\cite{elastica_conj} we show that only the Maxwell strata $\MAX_t^1$, $\MAX_t^2$, but not $\MAX_t^3$ are important for local optimality. But for the global problem, the stratum $\MAX^3_t$ is very important: in fact, on this stratum extremal trajectories lose global optimality, i.e., $\MAX^3_t$ provides a part of the cut locus~\cite{elastica_conj}.

The second system in~\eq{chain2} is compatible iff the equation $g_1(p,k)=0$ has solutions $(p,k)$ such that $\ds 0 \leq \frac{2 k^2 \ssp - 1}{k^2 \ssp} \leq 1$, or, which is equivalent,
\be{2k2s2}
2 k^2 \ssp - 1 \geq 0.
\ee

After the change of variable
\be{p=F(u,k)}
p = F(u,k) = \int _0^u \frac{dt}{\sqrt{1-k^2\sin^2t}}
\iff u = \am(p,k),
\ee
where $\am(p,k)$ is Jacobi's amplitude (see Sec.~\ref{sec:append}),
we obtain
\begin{align}
g_1(p,k) &= h_1(u,k), \nonumber \\
h_1(u,k)&=(1 - k^2 + k^2 \cos^4 u)(2 E(u,k) - F(u,k)) \nonumber \\
&\qquad + \cos u \sin u \sqrt{1-k^2 \sin^2 u}(2 k^2 \sin^2 u - 1).
\label{h1(u,k)=}
\end{align}
Denote
\begin{align}
h_2(u,k) &= \frac{h_1(u,k)}{1 - k^2 + k^2 \cos^4 u} \nonumber \\
&=
2 E(u,k) - F(u,k) + \frac{\cos u \sin u \sqrt{1-k^2 \sin^2 u}(2 k^2 \sin^2 u - 1)}{1 - k^2 + k^2 \cos^4 u},
\label{h2(u,k)=}
\end{align}
a direct computation gives
\begin{align}
\label{dh2dua1}
&\pder{h_2}{u} =
\frac{\sin^2 u \sqrt{2 - k^2 + k^2 \cos 2u}}{4 \sqrt 2 (1-k^2+k^2 \cos^4u)^2} \ a_1(u,k), \\
&a_1(u,k) = c_0 + c_1 \cos 2u + c_2 \cos^2 2u, \label{a1(u,k)} \\
&c_0 = 8 - 10k^2 + 4k^4, \nonumber \\
&c_1 = 4k^2(3-2k^2), \nonumber\\
&c_2 = 2k^2(2k^2-1). \nonumber
\end{align}
First we study roots of the function $a_1(u,k)$. In view of the symmetry relations
\be{a1symm}
a_1(u+\pi,k) = a_1(\pi -u,k) = a_1(u,k),
\ee
we can restrict ourselves by the segment $u \in [0, \frac{\pi}{2}]$. Consider the corresponding quadratic function
$$
a_1 = c_0 + c_1 t + c_2 t^2, \qquad t = \cos 2 u \in[-1, 1].
$$

If $k = \frac{1}{\sqrt 2}$, then
\be{a1=0k1/sqrt}
a_1 = 4(1+t) = 0 \iff t = -1.
\ee

Let $k \in (\frac{1}{\sqrt 2}, 1)$. Then $c_0 > 0$, $c_1> 0$, $c_2 > 0$, thus $a_1> 0$ for $t \in [0,1]$. On the other hand,
$$
\restr{a_1}{t=-1} = c_0 - c_1 + c_2 = 8(1-k^2)(1-2k^2) < 0.
$$
Thus the quadratic function $a_1 = c_0 + c_1 t + c_2 t^2$ has a unique root $t_{a_1}$ at the interval $t \in (-1,0)$. Consequently, the function $a_1(u,k)$ given by~\eq{a1(u,k)} has a unique zero $u_{a_1} = \frac 12 \arccos t_{a_1}$ at the segment $u \in [0, \frac{\pi}{2}]$, moreover, $u_{a_1} \in (\frac{\pi}{4}, \frac{\pi}{2})$.  We prove the following statement.

\begin{proposition}
\begin{itemize}
\item[$(1)$]
The set 
$
\left\{(u,k) \in \R \times \left[\frac{1}{ \sqrt 2} ,1\right]\mid a_1(u,k) = 0\right\}
$ 
is a \\ smooth curve.
\item[$(2)$]
There is a function
$$
\map{u_{a_1}}{\left[\frac{1}{\sqrt 2}, 1\right]}{\left(\frac{\pi}{4},\frac{\pi}{2}\right]},
\qquad u = u_{a_1}(k),
$$
such that
\begin{align*}
&k = \frac{1}{\sqrt 2}, \ 1 \then u_{a_1}(k) = \frac{\pi}{2}, \\
&k \in \left(\frac{1}{\sqrt 2},  1\right) \then u_{a_1}(k) \in \left( \frac{\pi}{4}, \frac{\pi}{2}\right),
\end{align*}
and for $\ds k = \frac{1}{\sqrt 2}, \ 1$
\be{k=1/sqrt2,1}
a_1(u,k) = 0
\iff
u = u_{a_1}(k) + \pi n = \frac{\pi}{2}+ \pi n,
\ee
while for $\ds k \in \left(\frac{1}{\sqrt 2},  1\right)$
\be{u=ua1(k)}
a_1(u,k) = 0
\iff
\orr{u = u_{a_1}(k) + 2 \pi n}{u = \pi - u_{a_1}(k) + 2 \pi n.}
\ee
Moreover,
\be{ua1reg}
u_{a_1} \in C \left[\frac{1}{\sqrt 2}, 1\right] \bigcap C^{\infty} \left(\frac{1}{\sqrt 2}, 1\right).
\ee
\end{itemize}
\end{proposition}
\begin{proof}
We assume in this proof that $\ds (u,k) \in \R \times \left[\frac{1}{\sqrt 2}, 1\right]$.

(1) We have
$$
\pder{a_1}{u} = - 2 \sin 2 u (c_1 + 2 c_2 \cos 2 u).
$$
It is easy to show that
\begin{multline*}
\left\{(u,k) \mid a_1(u,k) = 0, \ \pder{a_1}{u}(u,k) = 0 \right\} \\
=
\left\{(u,k)=\left(\frac{\pi}{2} + \pi n, \sqq\right), \
(u,k)= \left(\frac{\pi}{2} + \pi n, 1\right)  \right\}.
\end{multline*}
Further, for $u = \frac{\pi}{2} + \pi n$ we have
$$
a_1 = c_0 - c_1 + c_2 = 8 (1-k^2)(1-2k^2),
$$
which has regular zeros at $k = \sqq$, $k = 1$. Thus at the points $(u,k)=\left(\frac{\pi}{2} + \pi n, \sqq\right)$ and $(u,k)= \left(\frac{\pi}{2} + \pi n, 1\right)$ we have
$$
a_1(u,k) = 0, \qquad \pder{a_1}{k}(u,k) \neq 0.
$$
By implicit function theorem, the equation $a_1(u,k)=0$ determines a smooth curve.

(2)
For $k \in \left(\sqq, 1\right)$, we already defined before this proposition
$$
u_{a_1}(k) = \frac 12 \arccos t_{a_1}(k) \in \left(\frac{\pi}{4}, \frac{\pi}{2}\right),
$$
where $t_{a_1}(k) \in (-1,0)$ is the unique root of the quadratic polynomial $a_1 = c_0 + c_1 t + c_2 t^2$. We define now
$$
u_{a_1}\left(\sqq\right) = u_{a_1}(1) = \frac{\pi}{2}.
$$

For $k = \sqq$, we have by virtue of~\eq{a1=0k1/sqrt}:
$$
a_1 = 0
\iff t = \cos 2u = -1 \iff u = \frac{\pi}{2} + \pi n.
$$

For $k \in \left(\sqq, 1\right)$ and $u \in \left[ 0, \frac{\pi}{2}\right]$, we get
$$
a_1 = 0
\iff t = \cos 2u = t_{a_1}\in(-1, 0) \iff u = u_{a_1},
$$
and in view of the symmetry relations~\eq{a1symm}, implication~\eq{u=ua1(k)} follows.

Let $k=1$, then
$$
a_1(u,1) = 2 + 4 \cos 2u + 2 \cos^2 2u = 0
\iff
u = \frac{\pi}{2} + \pi n,
$$
and implication~\eq{k=1/sqrt2,1} is proved.

Finally, the regularity relations for the function $u_{a_1}(k)$ specified in~\eq{ua1reg} follow from the implicit function theorem by the argument of item (1).
\end{proof}

The plot of the curve $\{a_1(u,k) = 0\}$ in the domain $\{u \in (0, \pi)\}$ is presented at Fig.~\ref{fig:a1(u,k)=0}.

\onefiglabelsize{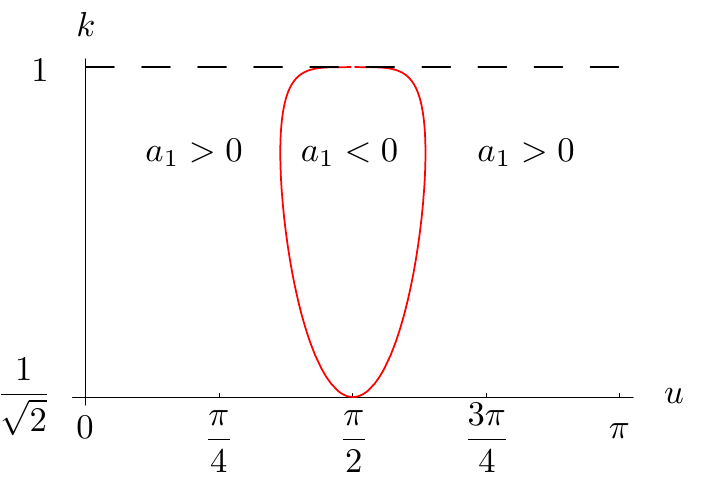}{$a_1(u,k) = 0$}{fig:a1(u,k)=0}{0.6}

The distribution of signs of the function $a_1(u,k)$ in the connected components of the domain $\{a_1(u,k)\neq 0\}$ shown at Fig.~\ref{fig:a1(u,k)=0} follows from the relations
\begin{align*}
&u = 0 \then a_1(u,k) = c_0 + c_1 + c_2 = 8 > 0, \\
&a_1(u,k) = 0, \ u \neq \frac{\pi}{2} + 2 \pi n \then \pder{a_1}{u} \neq 0.
\end{align*}

Now we study the structure and location of the curve
$$
\g_{h_1} = \left\{(u,k) \in (0,\pi) \times \left[\sqq, 1\right] \mid h_1(u,k) = 0\right\},
$$
so below in this subsection we suppose that $u \in (0,\pi)$, $k \in  [\sqq, 1]$.

Recall that the function
$$
h_2(u,k) = h_1(u,k) \underbrace{(1-k^2+k^2 \cos^4 u)^{-1}}_{> \, 0}
$$
has derivative~\eq{dh2dua1}
$$
\pder{h_2}{u} =
\underbrace{
\frac{\sin^2 u \sqrt{2-k^2 + k^2 \cos 2u}}{4 \sqrt 2 (1-k^2 + k^2 \cos^4 u)^2}
}_{>\, 0}
\, \cdot \,  a_1(u,k).
$$

A direct computation from~\eq{h1(u,k)=} gives
$$
h_1(u,k) = \frac 23 u^3 + o(u^3), \qquad u \to 0,
$$
thus
\be{h1h2>0}
h_1(u,k) > 0 \text{ and } h_2(u,k) > 0 \text{ as } u \to + 0.
\ee

If $u \in (0, u_{a_1}(k))$, then
\begin{align*}
a_1(u,k) > 0 &\then \pder{h_2}{u} > 0 \then h_2 \uparrow \text{ w.r.t. } u
\then h_2 > 0 \\
&\then h_1 > 0,
\end{align*}
thus $\g_{h_1} \cap \{u \in (0, u_{a_1}(k))\} = \emptyset$.

Now we study the curve $\g_{h_1}$ in the domain $\{u \in [u_{a_1}(k), \frac{\pi}{2}]\}$. We have
\begin{align}
\label{u=pi/2h1}
&u = \frac{\pi}{2} \then h_1(u,k) = (1-k^2)(2E(k)-F(k)), \\
&u = \pi \then h_1(u,k) = 2 (2 E(k) - F(k)). \nonumber
\end{align}

Notice that the unique root $k_0$ of the equation $2 E(k) - F(k)=0$ satisfies the inequality $k_0 \in (\sqq, 1)$, see~\eq{1/sqrt2k0}.

Taking into account Propos.~\ref{propos:lem21max3}, we obtain:
\begin{align*}
&k \in \left[\sqq, k_0\right),
\left( u = \frac{\pi}{2} \text{ or } u = \pi \right)
\then h_1(u,k) > 0, \\
&k = k_0,
\left( u = \frac{\pi}{2} \text{ or } u = \pi \right)
\then h_1(u,k) = 0, \\
&k \in \left(k_0, 1\right),
\left( u = \frac{\pi}{2} \text{ or } u = \pi \right)
\then h_1(u,k) < 0.
\end{align*}
If $k \in (k_0,1)$ then:
\begin{align*}
&u = u_{a_1}(k) \then h_2(u,k) > 0, \\
&u \in \left[u_{a_1}(k),\frac{\pi}{2}\right] \then h_2(u,k) \downarrow \text{ w.r.t. } u, \\
&u = \frac{\pi}{2} \then h_2(u,k) < 0, \\
&u \in \left[ \frac{\pi}{2}, \pi - u_{a_1}(k)\right] \then h_2(u,k) \downarrow \text{ w.r.t. } u, \\
&u = \pi - u_{a_1}(k)\then h_2(u,k) < 0, \\
&u \in \left[\pi - u_{a_1}(k), \pi\right] \then h_2(u,k) \uparrow \text{ w.r.t. } u, \\
&u = \pi \then h_2(u,k) < 0.
\end{align*}
Consequently, for $k\in (k_0, 1)$ the equation $h_2(u,k)=0$, or, equivalently, $h_1(u,k) = 0$, has a unique root $u = u_{h_2}(k)$ at the interval $u \in (0,\pi)$, moreover, $u_{h_2}(k) \in \left(u_{a_1}(k), \frac{\pi}{2}\right)$.

A similar argument shows that for $k = k_0$ the equation $h_2(u,k) = 0$ has a unique root $u = u_{h_2}(k_0)$ at the interval $u \in (0,\pi)$, moreover, $u_{h_2}(k_0) = \frac{\pi}{2}$.

In particular, we proved that
\be{kink0,1}
k \in [k_0,1) \then
h_2(\pi - u_{a_1}(k),k) < 0, \quad h_1(\pi - u_{a_1}(k),k) < 0.
\ee

Now we determine the largest root of the function
$$
\a(k) = h_1(\pi - u_{a_1}(k), k), \qquad k \in \left( \sqq,1\right).
$$
Notice that implication~\eq{kink0,1} means that $\a(k) < 0$ for $k \in [k_0, 1)$. By virtue of~\eq{K1/sqrt2},
$$
\a\left(\sqq\right) = h_1\left(\frac{\pi}{2}, \sqq\right) = \frac 12 \left(2E\left(\sqq\right) - F\left(\sqq\right)\right) > 0,
$$
thus the continuous function $\a(k)$ has roots at the interval $k \in \left(\sqq, k_0\right)$. Denote
\be{k*:=}
k_* = \sup \left\{ k \in \left(\sqq, k_0\right) \mid \a(k) = 0 \right\},
\ee
see Fig.~\ref{fig:kstar}.
It follows that $k_* \in \left(\sqq, k_0\right)$. 

If $k \in (k_*, 1)$, then:
\begin{align*}
&u = u_{a_1}(k) \then h_1(u,k) > 0, \\
&u = \pi - u_{a_1}(k) \then h_1(u,k) < 0.
\end{align*}
Thus there exists a function
$$
\map{u_{h_1}}{(k_*,1)}{\left(\frac{\pi}{4}, \frac{3 \pi}{4}\right)}, \qquad u = u_{h_1}(k),
$$
such that for $k \in (k_*, 1)$, $u \in (0, \pi - u_{a_1}(k))$
\begin{align*}
&(u,k) \in \g_{h_1} \iff h_1(u,k) = 0
\iff u = u_{h_1}(k), \\
&u_{a_1}(k) < u_{h_1}(k) < \pi - u_{a_1}(k).
\end{align*}
Further, we define
\begin{align}
&u_* = \pi - u_{a_1}(k_*), \qquad u_{h_1}(k_*) = u_*,  \label{u*:=}
\\
&u_{h_1}(1) = u_{a_1}(1) = \frac{\pi}{2}. \nonumber
\end{align}
For $k \in (k_*, 1)$, $u \in (0, \pi)$, the curve $\g_{h_1}$ does not intersect the curve $\{a_1(u,k) = 0\}$. Taking into account equalities~\eq{h2(u,k)=}, \eq{dh2dua1}, we conclude from implicit function theorem that $u_{h_1} \in C[k_*, 1] \cap C^{\infty}(k_*,1)$.

We proved the following statement.

\begin{lemma}
\label{lem:h1=0}
There exist numbers $k_* \in \left(\sqq, k_0\right)$, $u_* \in \left(\frac{\pi}{2}, \frac{3\pi}{4}\right)$ and a function
$$
\map{u_{h_1}}{[k_*, 1]}{\left(\frac{\pi}{2}, \frac{3\pi}{4}\right)}
$$
such that for $k \in [k_*, 1]$, $u \in (0, \pi - u_{a_1}(k))$
$$
h_1(u,k) = 0 \iff u=u_{h_1}(k).
$$
Moreover, $u_{h_1} \in C[k_*, 1] \cap C^{\infty}(k_*,1)$ and $u_{h_1}(k_*) = u_*$, $u_{h_1}(k_0) = u(1) = \frac{\pi}{2}$.
\end{lemma}
A plot of the function $u_{h_1}(k)$ is presented at Fig.~\ref{fig:u=uh1(k)}.

\twofiglabel
{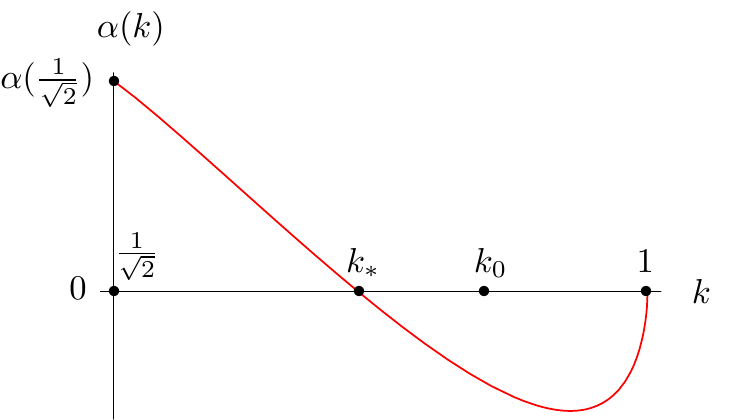}{Definition of $k_*$}{fig:kstar}
{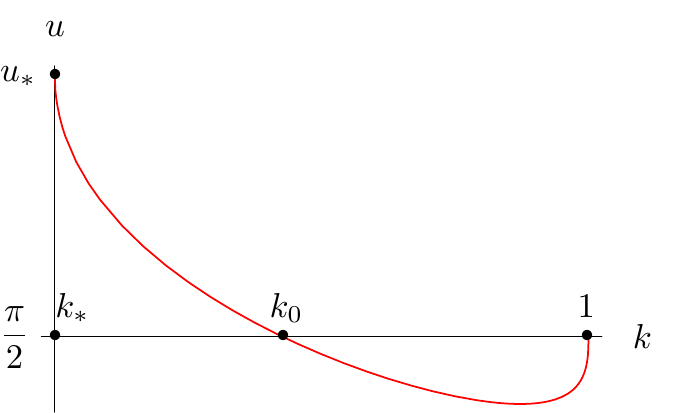}{Plot of $k \mapsto u_{h_1}(k)$}{fig:u=uh1(k)}

Numerical simulations give $k_* \approx 0.841$, $u_* \approx 1.954$.

Lemma~\ref{lem:h1=0} describes the first solutions in $u$ to the equation $h_1(u,k) = 0$ for $k \in [k_*, 1]$. Now we show that this equation has no solutions for $k \in [\sqq, k_*)$, $u \in (0, \pi]$. First we prove the following statement.

\begin{lemma}
\label{lem:dh2dkneq0}
The function $h_2(u,k)$ defined by~\eq{h2(u,k)=} satisfies the inequality
$$
\pder{h_2}{k} < 0 \text{ for } k \in \left(\sqq, 1\right), \ u \in \left[ \frac{\pi}{2}, \frac{3 \pi}{4}\right].
$$
\end{lemma}
\begin{proof}
Direct computation gives
\begin{align*}
&\pder{h_2}{k} = \frac{c_1}{128 k (1-k^2)(1-k^2 + k^2 \cos^4 u)^2}, \\
&c_1 = c_2 E(u,k) + c_3 F(u,k) + c_4 d_1 \sin 2u, \\
&c_2 = - 2(2k^2-1)(8-5k^2 + k^2(4\cos 2u+ \cos 4u))^2 < 0, \\
&c_3 = - 2 (1-k^2) (8-5k^2 + k^2(4\cos 2u+ \cos 4u))^2 < 0, \\
&c_4 = k^2 \sqrt{4 - 2k^2 + 2k^2 \cos 2u}>0, \\
&d_1 = 72-90k^2+28k^4-(48-97k^2+34k^4)\cos 2u + (8-6k^2+4k^4)\cos 4u \\
&\qquad\qquad + k^2 (2k^2-1)\cos 6u,
\end{align*}
so it is enough to prove that $d_1(u,k) > 0$ for $k\in(\sqq, 1)$, $u \in [\frac{\pi}{2}, \frac{3\pi}{4}]$. We have
\begin{align*}
&d_1 = 4e_1, \\
&e_1 = 16 -21 k^2 + 6k^4 - (12-25k^2+10k^4)t+(4-3k^2+2k^4)t^2+k^2(2k^2-1)t^3,\\
&t=\cos2u\in[-1,0],
\end{align*}
so it suffices to prove that for any $k \in (\sqq, 1)$, the cubic polynomial $e_1(t)$ is positive at the segment $t\in[-1,0]$. At the boundary of this segment we have:
\begin{align*}
&t=0 \then
e_1 = 16-21k^2+6k^4> 0 \text{ for } k^2 \in [0, 1], \\
&t=-1 \then e_1 = 16(2+k^2)(1+k^2)> 0,
\end{align*}
thus it is enough to prove that
\be{tin(-1,0)}
t \in (-1,0), \ \pder{e_1}{t} = 0 \then e_1 > 0.
\ee
We have
\begin{align*}
&\pder{e_1}{t} = -12 + 25k^2 -10k^4 + (8-6k^2+4k^4)t + (6k^4-3k^2)t^2, \\
&\pder{e_1}{t}=0 \iff t=t_1(k) \text{ or } t=t_2(k),\\
&t_1(k) = \frac{-4+3k^2-2k^4-2\sqrt{4-15k^2+43k^4-48k^6+16k^8}}{3k^2(2k^2-1)}, \\
&t_2(k) = \frac{-4+3k^2-2k^4+2\sqrt{4-15k^2+43k^4-48k^6+16k^8}}{3k^2(2k^2-1)}, \\
&\restr{e_1(t,k)}{t=t_1(k)} = \frac{16}{27k^4(2k^2-1)^2}(\a_1(k)+ (\b_1(k))^{3/2}), \\
&\restr{e_1(t,k)}{t=t_2(k)} = \frac{16}{27k^4(2k^2-1)^2}(\a_1(k)- (\b_1(k))^{3/2}), \\
&\a_1(k) = (1-k^2)(8-37k^2+146k^4-250k^6+224k^8-64k^{10}), \\
&\b_1(k) = 4-15k^2+43k^4-48k^6+16k^8,
\end{align*}
and it is enough to prove that $\a_1^2 > \b_1^3$ for $k \in (\sqq,1)$.

We have
\begin{align*}
&\a_3 = \a_1^2-\b_1^3 = 27k^4(1-k^2)^2(1-2k^2)^2 \a_4, \\
&\a_4 = 7-19 k^2+71k^4 -32k^6.
\end{align*}
Via standard calculus arguments, one can easily  prove that for $k\in(\sqq,1)$ we have $\a_4>0$, thus $\a_3 > 0$ and $\a_1 > \b_1^{3/2}$. The statement of this lemma follows.
\end{proof}

Let $k \in (\sqq, 1)$, $u\in(\frac{\pi}{2}, \frac{3\pi}{4})$. By implicit function theorem, we obtain from Lemma~\ref{lem:dh2dkneq0} that the equation $h_2(u,k) =0$ defines a smooth curve such that each its connected component is the graph of a smooth function $k = k(u)$, $u \in (\frac{\pi}{2}, \frac{3\pi}{4})$.

We have $h_2(u_*,k_*) =0$, so there exists such a connected component containing the point $(u_*,k_*)$; denote by $k = k_{h_2}(u)$,  $u \in (\frac{\pi}{2}, \frac{3\pi}{4})$ the function whose graph is this component:
$$
h_2(u,k_{h_2}(u))\equiv 0, \qquad k_{h_2}(u_*)=k_*.
$$
Notice that
$$
k = k_{h_2}(u) \iff u = u_{h_1}(k), \quad k \in [k_*, k_0], \ u \in \left[\frac{\pi}{2}, u_*\right].
$$
Now we prove that there are no other connected components of the curve $\{h_2(u,k)=0\}$ in addition to the component given by the graph of $k=k_{h_2}(u)$. By contradiction, suppose that there is such a component $k = \tk_{h_2}(u)$, $u \in (\frac{\pi}{2}, \frac{3 \pi}{4})$. The curve $\{k=\tk_{h_2}(u)\}$ cannot intersect the curve $\{u = \pi - u_{a_1}(k) = 0\}$ at a point where $k > k_*$ (this would contradict definition~\eq{k*:=} of the number $k_*$) or $k=k_*$ (this would mean that $\tk_{h_2}(u)\equiv k_{h_2}(u)$). Consequently,
\be{ktildeh2<}
\tk_{h_2}(u) < k_{h_2}(u), \qquad u \in \left(\frac{\pi}{2}, \frac{3 \pi}{4}\right).
\ee
Consider the limit $\lim_{u\to \pi/2 + 0} \tk_{h_2}(u)$.

a) Suppose that there exists $\hk = \lim_{u\to \pi/2 + 0} \tk_{h_2}(u)$. By virtue of inequality~\eq{ktildeh2<}, $\hk \in [\sqq, k_0]$. Since $h_2(u,\tk_{h_2}(u))\equiv 0$, it follows that $h_2(\frac{\pi}{2}, \hk) = 0$, thus  $h_1(\frac{\pi}{2}, \hk) = 0$. By virtue of~\eq{u=pi/2h1}, it follows that $2E(\hk)-F(\hk)=0$. By Propos.~\ref{propos:lem21max3}, we have $\hk = k_0$, which is impossible since $k_{h_2}(\frac{\pi}{2})=k_0$ and the curve $\{h_2(u,k)=0\}$ is smooth at the point $(k=k_0, u = \frac{\pi}{2})$.

b) Consequently, the limit $\lim_{u\to \pi/2 + 0} \tk_{h_2}(u)$ does not exist. But we can choose a converging sequence $(u_n,k_n) \to (\frac{\pi}{2} + 0, \hk)$, $\hk \in [\sqq, k_0]$, and come to a contradiction in the same way as in item a).

So we proved that the curve
$$
\left\{(u,k)\in \left[\frac{\pi}{2}, \frac{3\pi}{4}\right] \times \left( \sqq, k_0\right) \mid h_2(u,k)=0\right\}
$$
consists of the unique connected component
$$
k = k_{h_2}(u), \qquad u \in \left[\frac{\pi}{2}, \frac{3\pi}{4}\right].
$$

We have $\ds \der{k_{h_2}}{u} = - \pder{h_2}{u}/\pder{h_2}{k}$, and in view of equality~\eq{dh2dua1} and Lemma~\ref{lem:dh2dkneq0}, it follows that $\ds \sgn \der{k_{h_2}}{u} = \sgn a_1(u,k)$. Thus
\begin{align*}
&u \in \left[\frac{\pi}{2}, u_*\right) \then a_1(u,k_{h_2}(u))< 0 \then \der{k_{h_2}}{u} < 0, \\
&u \in \left(u_*, \frac{3\pi}{4}\right] \then a_1(u,k_{h_2}(u))> 0 \then \der{k_{h_2}}{u} > 0,
\end{align*}
so $k = k_*$ is the minimum of the function $k=k_{h_2}(u)$, $u \in \left(\frac{\pi}{2}, \frac{3\pi}{4}\right]$.

Then it follows that
$$
(u,k) \in \left( 0, \frac{3\pi}{4}\right] \times\left[\sqq,k_*\right) \then h_2(u,k) > 0.
$$
Taking into account that
$$
(u,k) \in \left[\frac{3\pi}{4}, \pi\right] \times\left[\sqq,k_*\right) \then a_1(u,k) > 0 \then \pder{h_2}{u} > 0,
$$
we obtain finally
$$
(u,k) \in \left(0, \pi\right] \times\left[\sqq,k_*\right) \then h_2(u,k) > 0 \then h_1(u,k) > 0.
$$
In particular, the pair $(u_*, k_*)$ is the unique solution to the system $h_1(u,k) = a_1(u,k)=0$ in the domain $(u,k) \in \left(\frac{\pi}{2}, \frac{3\pi}{4}\right] \times\left[\sqq,1\right)$.

Summing up, we proved the following statement.

\begin{proposition}
\label{propos:h1(u,k)=0}
\begin{itemize}
\item[$(1)$]
The set $\g_{h_1} = \{(u,k) \in (0, \pi] \times (0,1] \mid h_1(u,k) = 0\}$ is a smooth connected curve.
\item[$(2)$]
The system of equations $h_1(u,k) = 0$, $a_1(u,k) = 0$ has a unique solution $(u_*, k_*) \in \left(\frac{\pi}{2}, \frac{3\pi}{4}\right] \times\left[\sqq,1\right]$. Moreover, $k_* \in \left(\sqq, k_0\right)$, $u_* \in \left(\frac{\pi}{2}, \frac{3 \pi}{4}\right)$.
\item[$(3)$]
The curve $\g_{h_1}$ does not intersect the domain $\{(u,k) \in (0,\pi] \times [0, k_*)\}$.
\item[$(4)$]
There exist functions
\begin{align*}
&u = u_{h_1}(k), \qquad u_{h_1}\in C[k_*, 1] \cap C^{\infty}(k_*, 1), \\
&k = k_{h_2}(u), \qquad k_{h_2} \in C^{\infty}\left[\frac{\pi}{2}, \pi\right],
\end{align*}
such that in the domain  $\{(u,k) \in (0,\pi] \times [k_*, 1]\}$ holds the following:
\begin{align*}
&\g_{h_1} \cap \{u \in (0, u_*], \ k \in [k_*, 1]\} = \{u = u_{h_1}(k)\}, \\
&\g_{h_1} \cap \left\{u \in \left[\frac{\pi}{2}, \pi\right],  k \in \left[\sqq, k_0\right]\right\} = \{k = k_{h_2}(u)\}.
\end{align*}
\end{itemize}
The function $u=u_{h_1}(k)$ satisfies the bounds:
\begin{align*}
&k \in[k_*, k_0) \then u_{h_1}(k) \in \left(\frac{\pi}{2}, \frac{3\pi}{4}\right), \\
&k = k_0 \then u_{h_1}(k) = \frac{\pi}{2}, \\
&k \in(k_0, 1) \then u_{h_1}(k) \in \left(\frac{\pi}{4}, \frac{\pi}{2}\right).
\end{align*}
In particular, for $k \in [k_*, 1]$
$$
\min\{u>0 \mid h_1(u,k) = 0 \} = u_{h_1}(k).
$$
\end{proposition}

Now we return to the full system~\eq{chain2}, in particular, to the condition of compatibility~\eq{2k2s2}. After the change of variable~\eq{p=F(u,k)} this condition reads
$$
\b(u,k) = 2k^2 \sin^2 u - 1 \geq 0.
$$
We prove that this inequality holds on the curve $u = u_{h_1}(k)$, $k\in [k_*, 1]$.

We have:
\begin{align*}
\b(u,k) = 0 &\iff \sin^2 u = \frac{1}{2k^2} \iff \cos 2u = \frac{k^2-1}{k^2} \\
&\then a_1(u,k) = \frac{2(2k^2-1)}{k^2} > 0 \text{ for } k \in [k_*, 1].
\end{align*}
In other words, the curve $\{\b(u,k)=0\}$ is contained in the domain $\{a_1(u,k) > 0\}$. Thus the function $\b(u,k)$ preserves sign on each connected component of the domain $D_{a_1} = \{a_1(u,k) \leq 0, \ k \in [k_*, 1]\}$; since $\b(\frac{\pi}{2}, k) = 2k^2-1>0$, the function $\b(u,k)$ is positive on  $D_{a_1}$. On the other hand, the curve $\{u = u_{h_1}(k) \mid k \in [k_*, 1]\}$ is contained in the domain $D_{a_1}$, thus
\be{betauk>0}
k \in [k_*, 1], \ u = u_{h_1}(k) \then
\b(u,k) > 0.
\ee
The plot of the curves $\{h_1=0\}$, $\{a_1=0\}$, $\{\b=0\}$ is presented at Fig.~\ref{fig:h1a1beta=0}.
The elastica corresponding to $k = k_*$ is plotted at Fig.~\ref{fig:elastic_kstar}: for this elastica, the tangent line at the inflection point touches the preceding and the next waves of the elastica; moreover, $k=k_*$ is the minimal of such $k$.

\onefiglabelsize{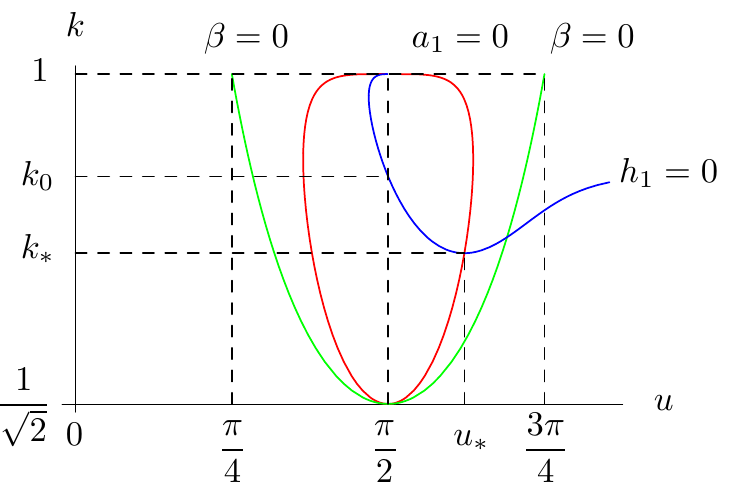}{The curves $\{h_1=0\}$, $\{a_1=0\}$, $\{\b=0\}$}{fig:h1a1beta=0}{0.65}

\onefiglabelsize{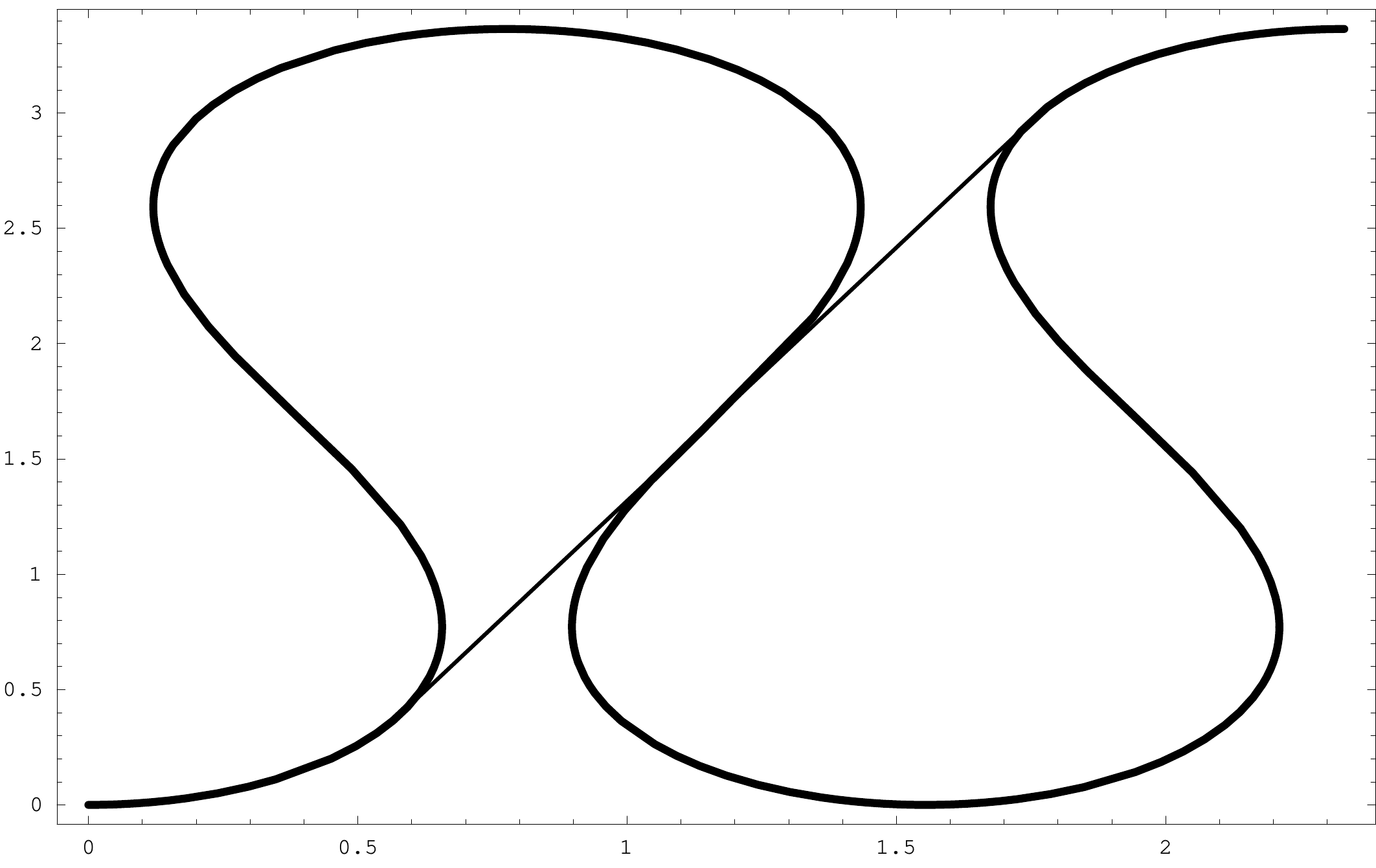}{Elastica with $k = k_*$}{fig:elastic_kstar}{0.6}

We return back from the variables $(u,k)$ to the initial variables $(p,k)$ via the formulas~\eq{p=F(u,k)}, and obtain the following statement.

\begin{proposition}
\label{propos:g1(p,k)=0}
Let the function $g_1(p,k)$ be given by~\eq{g1(p,k)N1}.
\begin{itemize}
\item[$(1)$]
The set
$$
\g_{g_1} = \{(p,k)\mid k \in (0,1), \ p \in (0,2K(k)), \ g_1(p,k) = 0\}
$$
is a smooth connected curve.
\item[$(2)$]
The curve $\g_{g_1}$ does not intersect the domain
\begin{align*}
&\{(p,k)\mid k \in (0, k_*), \ p \in (0, 2 K(k))\}. 
\end{align*}
\item[$(3)$]
The function
$$
p = p_{g_1}(k) = F(u_1(k), k), \qquad p_{g_1} \in C^{\infty}(k_*, 1),
$$
satisfies the condition
$$
\min \{p > 0 \mid g_1(p,k) = 0 \} = p_{g_1}(k), \qquad k \in [k_*, 1).
$$
The function $p = p_{g_1}(k)$ satisfies the bounds:
\begin{align*}
&k \in [k_*, k_0) \then p_{g_1}(k) \in \left(K, \frac 32 K\right), \\
&k = k_0 \then p_{g_1}(k) = K,  \\
&k \in (k_0, 1) \then p_{g_1}(k) \in \left(\frac 12 K,  K\right).
\end{align*}
\item[$(4)$]
For any $k \in [k_*, 1)$
$$
p = p_{g_1}(k) \then 2 k^2 \sn^2(p,k) - 1 \in (0,1].
$$
\item[$(5)$]
If $k \in (0, k_*)$, then the system of equations~\eq{chain2} has no solutions $(p, \tau)$ with $p \in (0, 2 K(k))$. If $k \in [k_*, 1)$, then the minimal $p \in (0, 2K(k))$ such that the system~\eq{chain2} has a solution $(p,\tau)$ is $p = p_{g_1}(k)$.
\end{itemize}
\end{proposition}

So we described the first solution to system~\eq{chain2} derived in Propos.~\ref{propos:theta=piy=0N1}.

\subsection{Roots of system $y=0$, $\t = \pi$ for $\nu \in N_2$}
Similarly to Propos.~\ref{propos:theta=piy=0N1}, we have the following statement.

\begin{proposition}
\label{propos:theta=piy=0N2}
\be{chain3}
\begin{cases}
\t_t = \pi \\
y_t = 0
\end{cases}
\iff
\begin{cases}
\ds \tsp = \frac{2\ssp -1}{k^2 \ssp} \\
g_1(p,k)=0
\end{cases}
\ee
where
\begin{multline}
\label{g1N2}
g_1(p,k) = \frac 1k [k^2\cc\ss\dd(2\ssp-1) \\ +(1-2\ssp+k^2\ssf)(2\E(p)-(2-k^2)p)].
\end{multline}
\end{proposition}
\begin{proof}
Let $\nu \in N_2^+$. We apply the equivalence relation~\eq{chain1}. Further, in order to obtain expressions for $\cos \frac{\t_t}{2}$ and $Q_t$ through the variables $\tau_2$, $p_2$ given by~\eq{tau2tau1}, \eq{p2p1}, we apply the transformation of Jacobi's functions $k \mapsto \frac{1}{k}$ in the same way as we did in Subsec.~\ref{subsec:P=0N2}, and obtain
\begin{align}
&\cos\frac{\t_t}{2} = \frac{(1-2\ssp+k^2\ssp\tsp)(\tcp+\ddp\tsp)}{\D^2},
\label{costhetaN2} \\
&Q_t = \frac 1k [(2\E(p) - (2-k^2)p) \nonumber \\
& \qquad\qquad +\tsp(2k^2 \cc\ss\dd-(2\E(p)-(2-k^2)p)(2-k^2\ssp)]. \nonumber
\end{align}
Consequently, $\cos \frac{\t_t}{2} = 0 \iff \tsp = \frac{2\ssp-1}{k^2\ssp}$. Then direct computation gives
$$
\restr{Q_t}{\tsp = (2\ssp-1)/(k^2\ssp)} = \frac{2}{k^2\ssp}g_1(p,k),
$$
where the function $g_1(p,k)$ is defined in~\eq{g1N2}. The statement of this proposition is proved for $\nu \in N_2^+$, and for $\nu \in N_2^-$ it is obtained via inversion $\map{i}{N_2^+}{N_2^-}$.
\end{proof}

Now we study solvability of the system of equations~\eq{chain3} in the domain $p \in (0,K)$. This bound on $p$ is given by the minimal $p=K$ for points in $\MAX_t^2$, see Propos.~\ref{propos:theta=0N2}.

After the change of variables~\eq{p=F(u,k)}, we have
\begin{align*}
g_1(p,k) &= h_1(u,k) \\
&= \frac 1k \left[ k^2 \cos u \sin u \sqrt{1-k^2 \sin^2u}(2\sin^2u-1) \right.\\
&\qquad \left. 
\vphantom{\sqrt{1-k^2 \sin^2u}}
+(1-2\sin^2u+k^2\sin^4u)(2E(u,k)-(2-k^2)F(u,k))\right], \\
&p\in(0,K) \iff u \in \left(0, \frac{\pi}{2}\right).
\end{align*}
Introduce the functions
$$
h_2(u,k) = \frac{h_1(u,k)}{\b(u,k)},
\qquad
\b(u,k) = 1 - 2 \sin^2 u + k^2 \sin^4 u.
$$
Then
\begin{align}
&\pder{h_2}{u} = \frac{k^3}{16 \b^2\sqrt{1-k^2\sin^2u}} \sin^2 (2u) h_3(u,k), \label{dh2duN2} \\
&h_3(u,k) = 2(2-k^2+2k^2\cos 2u + (2-k^2)\cos^2 2u). \nonumber
\end{align}
The quadratic polynomial
$$
h_3(t) = 2(2-k^2 + 2 k^2 t + (2-k^2)t^2)
$$
is positive for all $t\in \R$, thus $h_3(u,k) > 0$ for all $u \in \R$, $k \in (0,1)$. Then equality~\eq{dh2duN2} implies that $\pder{h_2}{u}>0$, thus $h_2$ increases in $u$ on all intervals where $\b(u,k)\neq 0$. It is easy to show that
$$
\b(u,k) = 0, \ u \in \left(0, \frac{\pi}{2}\right) \iff u = u_{\b}(k) = \arcsin \frac{1}{\sqrt{1+k'}}, \ k' = \sqrt{1-k^2}.
$$

If $u \in (0, u_{\b}(k))$, then
\begin{align}
\b(u,k)>0 &\then h_2(u,k) \uparrow \text{ in } u \nonumber \\
\intertext{and since $h_2(0,k) = h_1(k) = 0$,}
&\then h_2(u,k) > 0 \then h_1(u,k) > 0. \label{h1>01}
\end{align}

If $u = u_{\b}(k)$, then
\be{h1>02}
h_1(u,k) = \frac{k^3k'}{(1+k')^3}>0.
\ee
Further,
\begin{align*}
u \to u_{\b}(k) + 0 &\then \b(u,k) \to - 0, \ h_1(u,k) \to \frac{k^3k'}{(1+k')^3}>0 \\
&\then h_2(u,k) \to - \infty.
\end{align*}

Finally, if $u\in \left(u_{\b}(k), \frac{\pi}{2}\right]$, then
$$
\b(u,k) < 0 \then h_2(u,k) \uparrow \text{ in } u.
$$
We have
\begin{align*}
&h_2\left(\frac{\pi}{2},k\right) = \frac{1}{k} \, \g(k), \\
&\g(k) = 2 E(k) - (2-k^2) F(k).
\end{align*}
By Propos.~\ref{propos:lem23max3} presented below and proved in~\cite{max3}, $\g(k) < 0$ for any $k \in (0,1)$. Thus
\begin{align}
h_2\left(\frac{\pi}{2},k\right) < 0 &\then h_2(u,k) < 0 \quad \forall \ u \in \left( u_{\b}(k), \frac{\pi}{2}\right] \nonumber \\
&\then h_1(u,k) > 0 \quad \forall \ u \in \left( u_{\b}(k), \frac{\pi}{2}\right].
\label{h1>03}
\end{align}

Summing up inequalities~\eq{h1>01}, \eq{h1>02}, \eq{h1>03}, we proved that
$$
h_1(u,k) > 0 \quad \forall \ u \in \left(0, \frac{\pi}{2}\right], \ k \in (0,1).
$$

We return back to the initial variable $p$ via the change of variables~\eq{p=F(u,k)}, and obtain the following statement.

\begin{proposition}
\label{propos:g1(p,k)=0N2}
Let the function $g_1(p,k)$ be given by~\eq{g1N2}. Then for any $k\in (0,1)$, $p \in (0, K(k))$ we have $g_1(p,k) > 0$.
\end{proposition}

In fact, numerical simulations show that the equation $g_1(p,k) = 0$ has solutions $p>K$.


Here we present the statement used above in the proof of Propos.~\ref{propos:g1(p,k)=0N2}.

\begin{proposition}[Lemma 2.3 \cite{max3}]
\label{propos:lem23max3}
The function $\g(k) = 2 E(k) - (2 - k^2)K(k)$ is negative for
$k \in (0, 1)$.
\end{proposition}

\subsection{Roots of system $y=0$, $\t = \pi$ for $\nu \in N_3$}
\begin{proposition}
\label{propos:y=0theta=piN3}
If $\nu \in N_3$, then the system of equations $y_t = 0$, $\t_t=\pi$ is incompatible for $t > 0$.
\end{proposition}
\begin{proof}
Let $\nu \in N_3^+$. We pass to the limit $k \to 1-0$ in Propos.~\ref{propos:theta=piy=0N2}, \ref{propos:g1(p,k)=0N2} and obtain that the system of equations $y_t = 0$, $\t_t = \pi$ has no roots for $p \in (0, K(1-0))$, $p = \frac{\sqrt r t}{2}$. But $K(1-0) = \lim_{k \to 1-0}K(k) = + \infty$. Thus the system in question is incompatible for $t > 0$ and $\nu \in N_3^+$. The same result for $\nu \in N_3^-$ follows via the inversion $\map{i}{N_3^+}{N_3^-}$.
\end{proof}

\subsection{Roots of system $y=0$, $\t = \pi$ for $\nu \in N_6$}
\begin{proposition}
\label{propos:y=0theta=piN6}
If $\nu \in N_6$, then the system of equations $y_t = 0$, $\t_t=\pi$ is incompatible.
\end{proposition}
\begin{proof}
As always, we can restrict ourselves by the case $\nu \in N_6^+$. Then it is obvious that the system is incompatible:
$$
y_t = \frac{1-\cos ct}{c} = 0, \qquad \t_t = ct = \pi + 2 \pi k.
$$
\end{proof}

\subsection{Complete description of Maxwell strata}

Now we can summarize our previous results and obtain the following statement.

\begin{theorem}
\label{th:Maxwell_complete}
\begin{itemize}
\item[$(1.1)$]
$N_1 \cap \MAX^1_t = \{ \nu \in N_1 \mid p = 2 Kn, \ \cn \tau \neq 0\}$,
\item[$(1.2)$]
$N_1 \cap \MAX^2_t = \{ \nu \in N_1 \mid p = p^1_n, \ \sn \tau \neq 0\}$,
\item[$(1.3+)$]
$N_1 \cap \MAX^{3+}_t = \{ \nu \in N_1 \mid (k,p) = (k_0, 2 Kn) \text{ or } (p=p^1_n, \cn \tau = 0) \text{ or } (p=2Kn, \ts = 0) \}$,
\item[$(1.3-)$]
$N_1 \cap \MAX^{3-}_t = \{ \nu \in N_1 \mid g_1(p,k) = 0, \ \tsp = (2k^2 \ssp -1)/(k^2 \ssp)\}\footnote{where the function $g_1(p,k)$ is given by~\eq{g1(p,k)N1}}$,
$N_1 \cap \MAX^{3-}_t\cap\{p \in (0, 2K)\} = \{k \in [k_*, 1), \ p = p_{g_1}(k), \ \tsp = (2k^2 \ssp -1)/(k^2 \ssp)\}\footnote{where $k_*$ and  $p_{g_1}(k)$ are described in Propos.~\ref{propos:g1(p,k)=0}}$,
\item[$(2.1)$]
$N_2 \cap \MAX^1_t = \{ \nu \in N_2 \mid p = Kn, \ \cn \tau \ts \neq 0\}$,
\item[$(2.2)$]
$N_2 \cap \MAX^2_t = \emptyset$,
\item[$(2.3+)$]
$N_2 \cap \MAX^{3+}_t = \{ \nu \in N_2 \mid p =  Kn, \ \ts \tc = 0 \}$,
\item[$(2.3-)$]
$N_2 \cap \MAX^{3-}_t = \{ \nu \in N_2 \mid g_1(p,k) = 0, \ \tsp = (2 \ssp -1)/(k^2 \ssp)\}\footnote{where the function $g_1(p,k)$ is given by~\eq{g1N2}}$,
$N_2 \cap \MAX^{3-}_t\cap\{p \in (0, K)\} = \emptyset$,
\item[$(3.1)$]
$N_3 \cap \MAX^1_t = \emptyset$,
\item[$(3.2)$]
$N_3 \cap \MAX^2_t = \emptyset$,
\item[$(3.3+)$]
$N_3 \cap \MAX^{3+}_t = \emptyset$,
\item[$(3.3-)$]
$N_3 \cap \MAX^{3-}_t = \emptyset$,
\item[$(6.1)$]
$N_6 \cap \MAX^{1}_t = \{\nu \in N_6 \mid ct = 2 \pi n \}$,
\item[$(6.2)$]
$N_6 \cap \MAX^2_t = \emptyset$,
\item[$(6.3+)$]
$N_6 \cap \MAX^{3+}_t = \{\nu \in N_6 \mid ct = 2 \pi n \}$,
\item[$(6.3-)$]
$N_6 \cap \MAX^{3-}_t = \emptyset$.
\end{itemize}
\end{theorem}
\begin{proof}
It remains to compile the corresponding items of Th.~\ref{th:MAX_gen} with appropriate propositions of this section:
\begin{itemize}
\item[$(1.1)$]
Propos.~\ref{propos:theta=0N1},
\item[$(1.2)$]
Propos.~\ref{propos:P=0N1},
\item[$(1.3+)$]
Propos.~\ref{propos:y=0theta=0N1},
\item[$(1.3-)$]
Propos.~\ref{propos:theta=piy=0N1}, \ref{propos:g1(p,k)=0},
\item[$(2.1)$]
Propos.~\ref{propos:theta=0N2},
\item[$(2.2)$]
Propos.~\ref{propos:P=0N2},
\item[$(2.3+)$]
Propos.~\ref{propos:y=0theta=0N2},
\item[$(2.3-)$]
Propos.~\ref{propos:theta=piy=0N2}, \ref{propos:g1(p,k)=0N2},
\item[$(3.1)$]
Propos.~\ref{propos:theta=0N3},
\item[$(3.2)$]
Propos.~\ref{propos:P=0N3},
\item[$(3.3+)$]
Propos.~\ref{propos:y=0theta=0N3},
\item[$(3.3-)$]
Propos.~\ref{propos:y=0theta=piN3},
\item[$(6.1)$]
Propos.~\ref{propos:theta=0N6},
\item[$(6.2)$]
item (6.2) of Th.~\ref{th:MAX_gen},
\item[$(6.3+)$]
Propos.~\ref{propos:y=0theta=0N6},
\item[$(6.3-)$]
Propos.~\ref{propos:y=0theta=piN6}.
\end{itemize}
\end{proof}

\section{Upper bound on cut time}
\label{sec:up_bound}

Let $q_s$, $s>0$, be an extremal trajectory of an optimal control problem of the form~\eq{gp1}--\eq{gp3}. The \ddef{cut time} for the trajectory $q_s$ is defined as follows:
$$
\tcut = \sup \{t_1 > 0 \mid q_s \text{ is optimal on } [0, t_1]\}.
$$
For normal extremal trajectories $q_s = \Exp_s(\lam)$, the cut time becomes a function of the initial covector $\lam$:
$$
\map{\tcut}{N=T_{q_0}^*M}{[0, + \infty]}, \qquad t = \tcut(\lam).
$$

Short arcs of regular extremal trajectories are optimal, thus $\tcut(\lam) > 0$ for any $\lam \in N$. On the other hand, some extremal trajectories can be optimal on an arbitrarily long segment $[0, t_1]$, $t_1 \in(0, + \infty)$; in this case $\tcut = + \infty$.

Denote the first Maxwell time as follows:
$$
t_1^{\MAX}(\lam) = \inf \{ t > 0 \mid \lam \in \MAX_t \}.
$$
By Propos.~\ref{propos:Maxwell}, a normal extremal trajectory $q_s$ cannot
be optimal after a Maxwell point, thus
\be{tcut<=}
\tcut(\lam) \leq t_1^{\MAX}(\lam).
\ee

Now we return to Euler's elastic problem. For this problem we can define
the first instant in the Maxwell sets $\MAX^i$, $i = 1, 2, 3$:
$$
t_1^{\MAX^i}(\lam) = \inf \{ t > 0 \mid \lam \in \MAX^i_t\}.
$$
Since $t_1^{\MAX}(\lam) \leq t_1^{\MAX^i}(\lam)$, we obtain from inequality~\eq{tcut<=}:
$$
\tcut(\lam) \leq \min(t_1^{\MAX^i}(\lam)), \qquad i = 1, 2, 3.
$$
Now we combine this inequality with the results of Sec.~\ref{sec:MAX_complete}
and obtain an upper bound on cut time in Euler's elastic problem. To this
end we define the following function:
\begin{align}
&\map{\tt}{N}{(0, + \infty]}, \qquad \lam \mapsto \tt(\lam), \nonumber\\
&\lam \in N_1 \then \tt = \frac{2}{\sqrt r} p_1(k),\nonumber \\
&\qquad
p_1(k) = \min(2K(k), p_1^1(k)) =
\begin{cases}
2 K(k), & k \in (0, k_0] \\
p_1^1(k), &k \in [k_0, 1)
\end{cases}
\label{p1(k)N1}\\
&\lam \in N_2 \then \tt = \frac{2k}{\sqrt r} p_1(k), \qquad p_1(k) = K(k),
\nonumber\\
&\lam \in N_6 \then \tt = \frac{2\pi}{|c|}, \nonumber\\
&\lam \in N_3\cup N_4\cup N_5\cup N_7 \then \tt = + \infty. \nonumber
\end{align}

\begin{theorem}
\label{th:tcut_bound}
Let $\lam \in N$. We have
\be{tcut_bound}
\tcut(\lam) \leq \tt(\lam)
\ee
in the following cases:
\begin{itemize}
\item[$(1)$]
$\lam  = (k, p, \tau) \in N_1$, $\tc \ts \neq 0$, or
\item[$(2)$]
$\lam \in N \setminus N_1$.
\end{itemize}
\end{theorem}
\begin{proof}
(1) Let $\lam = (k, p, \tau) \in N_1$, $\tc \ts \neq 0$.
Then Th.~\ref{th:Maxwell_complete} yields the following:
\begin{align*}
&k \in (0, k_0] \then
\tt(\lam) = \frac{2}{\sqrt r} 2 K = t_1^{\MAX^1}(\lam), \\
&k \in (k_0, 1) \then
\tt(\lam) = \frac{2}{\sqrt r} p_1^1(k) = t_1^{\MAX^2}(\lam).
\end{align*}

(2) Let $\lam = (k,p,\tau) \in N_2$, then we obtain from
Th.~\ref{th:Maxwell_complete}:
\begin{align*}
&\ts \tc \neq 0 \then \tt(\lam) = \frac{2 K(k) k}{\sqrt r} = t_1^{\MAX^1}(\lam),\\
&\ts \tc = 0 \then \tt(\lam) = \frac{2 K(k)k}{\sqrt r} = t_1^{\MAX^{3+}}(\lam).
\end{align*}

If $\lam = (\b, c, r) \in N_6$, then Th.~\ref{th:Maxwell_complete} implies
that
$$
\tt(\lam) = \frac{2\pi}{|c|} = t_1^{\MAX^1}(\lam) = t_1^{\MAX^3}(\lam).
$$

If $\lam \in N_3$, then there is nothing to prove since $\tt(\lam) = + \infty$.

If $\lam \in N_4 \cup N_5 \cup N_7$, then there is also nothing to prove
since in this case the extremal trajectory $q_s$ is optimal on the whole
ray $s \in [0, + \infty)$, and $\tcut (\lam) = \tt(\lam) = + \infty$.
\end{proof}

In the proof of Th.\ref{th:tcut_bound}, we used the explicit
description~\eq{p1(k)N1} of the function $p_1(k) = \min(2K(k), p_1^1(k))$
which follows directly from Prop.~\ref{propos:lem21max3}.

\begin{remark}
In the subsequent work~\cite{elastica_conj} we prove that if $\lam = (k,p,\tau) \in N_1$ and $\tc \ts = 0$, then the corresponding point $q_t = \Exp_t(\lam)$, $t = \tt(\lam)$ is conjugate, thus the trajectory $q_s$ is not optimal for $s > \tt(\lam)$; consequently, $\tcut(\lam) \leq \tt(\lam)$, compare with item (1) of Th.~\ref{th:tcut_bound}. So the bound~\eq{tcut_bound} is valid for \emph{all} $\lam \in N$. 
\end{remark}

Notice the different role of the Maxwell strata $\MAX^{3+}$ and $\MAX^{3-}$
for  the upper bound of the cut time obtained in Th.~\ref{th:tcut_bound}.
On the one hand, the stratum $\MAX^{3+}$ generically does not give better
bound on cut time than the strata $\MAX^1$, $\MAX^2$ since generically
$$
t_1^{\MAX^{3+}} = \min\left(t_1^{\MAX^1}, t_1^{\MAX^2}\right),
$$
see Th.~\ref{th:Maxwell_complete}, this follows mainly from the fact that
the system of equations determining the stratum $\MAX^{3+}$ consists of the
equations determining the strata $\MAX^1$ and $\MAX^2$:
$$
\begin{cases}
y_t = 0 \\ \t_t = 0
\end{cases}
\iff
\begin{cases}
P_t = 0 \\ \t_t = 0
\end{cases}
$$
see Th.~\ref{th:MAX_gen}.

The situation with the stratum $\MAX^{3-}$ is drastically different. By item
$(1.3-)$ of Th.~\ref{th:Maxwell_complete}, we have
\begin{align}
&\nu = (k,p,\tau) \in N_1 \cap \MAX^{3-}, \label{nuN1MAX3-1}\\
&k \in [k_*, 1), \quad p = p_{g_1}(k), \quad
\tsp = \frac{2 k^2 \ssp - 1}{k^2 \ssp} \in [0, 1]. \label{nuN1MAX3-2}
\end{align}
Moreover, from Propos.~\ref{propos:g1(p,k)=0}, Th.~\ref{th:Maxwell_complete},
Propos.~\ref{propos:prop21max3} it follows that
\begin{align*}
&k \in [k_*, k_0) \then
p_{g_1}(k) < \frac 32 K < 2 K = p_1(k), \\
&k = k_0 \then p_{g_1}(k) = K < 2 K = p_1(k), \\
&k \in (k_0, 1) \then p_{g_1}(k) < K < p_1^1(k) = p_1(k).
\end{align*}
That is, $p_{g_1}(k) < p_1(k)$ for all $k \in [k_*, 1)$, consequently,
$$
t_1^{\MAX^{3-}}(\lam) < \tt(\lam) = \min\left(t_1^{\MAX^1}(\lam), t_1^{\MAX^2}(\lam)\right)
$$
for all $\lam = \nu$ defined by~\eq{nuN1MAX3-1}, \eq{nuN1MAX3-2}.

It is natural to conjecture that for such $\lam$ we have
\be{tcut=tMAX3-}
\tcut(\lam) = t_1^{\MAX^{3-}}(\lam)
\ee
and we will prove this equality in the subsequent work~\cite{elastica_conj}.

Although, the covectors $\lam = \nu$ defined by~\eq{nuN1MAX3-1},
\eq{nuN1MAX3-2} form a codimension~2 subset of $N$, so the
equality~\eq{tcut=tMAX3-} defines the cut time for a codimension one subset
of extremal trajectories. The question on exact description of cut time foe
arbitrary extremal trajectories is under investigation now.

An essential progress in the description of the cut time was achieved via
the study of the global properties of the exponential mapping. Moreover,
a precise description of locally optimal extremal trajectories (i.e., stable
Euler elasticae) was obtained due to the detailed study of conjugate points.
These results will be presented in the subsequent work~\cite{elastica_conj}.

\section[Appendix: Jacobi's elliptic integrals and functions]
{Appendix: \\ Jacobi's elliptic integrals and functions}
\label{sec:append}

We base upon the textbooks of D.F.~Lawden~\cite{lawden} and 
E.T.~Whittaker, G.N.~Watson~\cite{whit_watson}.

\subsection{Jacobi's elliptic integrals}
Elliptic integrals of the first kind:
$$
F(\f, k) = \int _0^{\f} \frac{dt}{\sqrt{1 - k^2 \sin^2 t}},
$$
and of the second kind:
$$
E(\f, k) = \int_0^{\f} \sqrt{1 - k^2 \sin^2 t} \, dt .
$$
Complete elliptic integrals:
\begin{align*}
&K(k) = F\left(\frac{\pi}{2}, k\right) = \int _0^{\pi/2} \frac{dt}{\sqrt{1 - k^2 \sin^2 t}}, \\
&E(k) = E\left(\frac{\pi}{2}, k\right) = \int_0^{\pi/2} \sqrt{1 - k^2 \sin^2 t} \, dt.
\end{align*}

\subsection{Definition of Jacobi's elliptic functions}
\begin{align}
&\f = \am u \quad \Leftrightarrow \quad u = F(\f, k), \nonumber\\
&\cn u = \cos \am u, \label{cn} \\
&\sn u = \sin \am u,  \label{sn} \\
&\dn u = \sqrt{1 - k^2 \sn^2 u}, \label{dn} \\
&\E(u) = E(\am u, k).
\end{align}

\subsection{Standard formulas on Jacobi's elliptic functions}

\paragraph{Derivatives with respect to $u$}

\begin{align*}
&\am' u = \dn u, \\
&\sn'u = \cn u \dn u, \\
&\cn' u = - \sn u \dn u, \\
&\dn'u = - k^2 \sn u \cn u, 
\end{align*}

\paragraph{Derivatives with respect to modulus $k$}

\begin{align*}
&\pder{\sn u}{k} = 
\frac 1k u \cn u \dn u + \frac{k}{1-k^2} \sn u \cn^2 u - \frac{1}{k(1-k^2)} \E(u) \cn u \dn u, \\
&\pder{\cn u}{k} = 
-\frac 1k u \cn u \dn u - \frac{k}{1-k^2} \sn^2 u \cn u + \frac{1}{k(1-k^2)} \E(u) \sn u \dn u, \\
&\pder{\dn u}{k} = 
-\frac{k}{1-k^2} \sn^2 u \dn u - k u \sn u \cn u + \frac{k}{1-k^2} \E(u) \sn u \cn u, \\
&\pder{\E(u)}{k} = 
\frac{k}{1-k^2} \sn u \cn u\dn u -k u  \sn ^2 u - \frac{k}{1-k^2} \E(u) \cn^2 u, \\
&\der{K}{k} = \frac{E - (1 - k^2)K}{k (1 - k^2)}, \\
&\der{E}{k} = \frac{E - K}{k}.
\end{align*}

\paragraph{Integrals}
\begin{align*}
&\int_0^u \dn^2 t \, dt  = \E(u).
\end{align*}

\paragraph{Addition formulas}
\begin{align*}
&\sn(u + v) = \frac{\sn u \cn v \dn v + \cn u \dn u \sn v}{1 - k^2 \sn^2 u \sn^2 v}, \\
&\cn(u + v) = \frac{\cn u \cn v  -  \sn u \dn u \sn v \dn v}{1 - k^2 \sn^2 u \sn^2 v}, \\
&\dn(u + v) = \frac{\dn u \dn v  - k^2 \sn u  \cn u \sn v \cn v}{1 - k^2 \sn^2 u \sn^2 v}, \\
&\E(u + v) = \E(u) + \E(v) - k^2 \sn u \sn v \sn(u+v).
\end{align*}

\paragraph{Degeneration}
\begin{align}
k \to +0 \quad &\Rightarrow \quad \sn u  \to \sin u, \quad  \cn u \to \cos u, \quad  \dn u \to 1, 
\quad  \E(u) \to u,
\label{degen0}\\
k \to 1-0 \quad &\Rightarrow \quad \sn u  \to \tanh u, \quad  \cn u,  \ \dn u \to \frac{1}{\cosh u}, 
\quad  \E(u) \to \tanh u.
\label{degen1}
\end{align}

\paragraph{Transformation $k \mapsto \dfrac 1k$}

\begin{align}
&\sn\!\left(u,\frac{1}{k}\right) = k \sn\!\left(\frac{u}{k},k\right), \qquad &&\cn\!\left(u,\frac{1}{k}\right) = \dn\!\left(\frac{u}{k},k\right), \label{k->1/k1} \\
&\dn\!\left(u,\frac{1}{k}\right) =  \cn\!\left(\frac{u}{k},k\right),
&&\E\!\left(u,\frac{1}{k}\right) = \frac{1}{k} \E\!\left(\frac{u}{k},k\right) - \frac{1-k^2}{k^2} u.  \label{k->1/k2}
\end{align}

\newpage
\addcontentsline{toc}{section}{\listfigurename}
\listoffigures


\newpage
\addcontentsline{toc}{section}{\refname}
 \begin{thebibliography}{99}

 \bibitem{ABCK}
A.~Agrachev, B.~Bonnard, M.~Chyba, I.~Kupka,
Sub-Riemannian sphere in Martinet flat case. {\em J.~ESAIM: Control,
Optimization and Calculus of Variations}, 1997, v.2,
377--448.

 \bibitem{notes}
A.A.~Agrachev, Yu. L. Sachkov,  
{\em Geometric control theory}, Fizmatlit, Moscow 2004; English transl.
{\em Control Theory from the Geometric Viewpoint},
Springer-Verlag, Berlin 2004.


\bibitem{cdc99}
A.\,A.~Agrachev, Yu.L.~Sachkov,
An Intrinsic Approach to the Control of Rolling Bodies,
{\em Proceedings of the 38-th
IEEE Conference on Decision and Control}, Phoenix, Arizona, USA, December
7--10, 1999, vol.~1, 431--435.

\bibitem{antman}
S.S.Antman,
The influence of elasticity on analysis: Modern developments,
{\em Bulletin American Math. Society}, 1983, v. 9, No. 3, 267--291.

\bibitem{arnold_osob_kaustic}
V.I. Arnold, {\em 
Singularities of caustics and wave fronts}, Kluwer, 1990. 

\bibitem{arthur_walsh}
A.M.Arthur, G.R.Walsh, On the Hammersley's minimum problem for a rolling sphere,
{\em Math. Proc. Cambridge Phil. Soc.}, v. 99 (1986), 529--534.


\bibitem{DBernoulli}
D.Bernoulli, 26th letter to L. Euler (October, 1742), In: Fuss, {\em Correspondance math\'ematique et physique}, t.2, St. Petersburg, 1843.

\bibitem{JBernoulli}
J.Bernoulli,
V\'eritable hypoth\`ese de la r\'esistance des solides, avec la demonstration de la corbure des corps qui font ressort,
In: {\em Collected works}, t.2, Geneva, 1744.


\bibitem{birkhoff64}
G.Birkhoff, C.R. de Boor, Piecewise polynomial interpolation and approximation, In: {\em Approximation of Functions} (Proc. Sympos. General Motors Res. Lab., 1964), Elsevier, Amsterdam, 1965, 164--190.


\bibitem{born}
M.Born,
{\em Stabilit\"at der elastischen Linie in Ebene und Raum},
Preisschrift und Dissertation, G\"ottingen, Dieterichsche Universit\"ats-Buchdruckerei G\"ottingen, 1906. Reprinted in: {\em Ausgew\"ahlte Abhandlungen}, G\"ottingen, Vanderhoeck \& Ruppert, 1963, Vol. 1, 5--101.



\bibitem{brock_dai}
Brockett R., Dai L.
Non-holonomic kinematics and the role of el\-lip\-tic functions in
constructive controllability//
In: Nonholonomic Motion Planning, Z.~Li and J.~Canny, Eds.,
Kluwer, Boston, 1993, 1--21.

\bibitem{cesari}
L. Cesari,
{\em Optimization --- Theory and Applications. Problems with Ordinary Differential Equations},
Springer-Verlag,
New York, Heidelberg, Berlin,
1983.

\bibitem{euler}
L.Euler,
{\em Methodus inveniendi lineas curvas maximi minimive proprietate gaudentes, sive Solutio problematis isoperimitrici latissimo sensu accepti}, Lausanne, Geneva, 1744.

\bibitem{golumb_jerome82}
M.Golumb, J.Jerome,
Equilibria of the curvature functional and manifolds of non-linear interpolating spline curves,
{\em SIAM J. Math. Anal.} v. 13 (1982), 421--458.

\bibitem{jacquet}
S. Jacquet, Regularity of sub-Riemannian distance and cut locus,
{\em Preprint No. 35}, May 1999, Universita degli Studi di Firenze, Dipartimento di Matematica Applicata ``G. Sansone'', Italy.

\bibitem{Jacobi}
C.G.J. Jacobi, 
{\em Volresungen \"uber Dynamik}, G. Reimer, Berlin 1891.

\bibitem{jerome73}
J.W.Jerome,
Minimization problems and linear and nonlinear spline functions, I: Existence,
{\em SIAM  J. Numer. Anal.} 10 (1973), 808--819.

\bibitem{jerome75}
J.W.Jerome,
Smooth interpolating curves of prescribed length and minimum curvature,
{\em Proc. Amer. Math. Soc.} 51 (1975), 62--66.

\bibitem{jurd_ball_plate}
V.Jurdjevic,
The geometry of the ball-plate problem,
{\em Arch. Rat. Mech. Anal.}, v. 124 (1993), 305--328.

\bibitem{jurd_noneuclid}
V.Jurdjevic,
Non-Euclidean elastica,
{\em Am. J. Math.}, v. 117 (1995), 93--125.

\bibitem{jurd_book}
V.~Jurdjevic,
{\em Geometric Control Theory},
Cambridge University Press, 1997.

\bibitem{lawden}
D.F.~Lawden, {\em Elliptic functions and applications},
Springer-Verlag, 1989.

\bibitem{linner96}
A.Linn\'er,
Unified representations of non-linear splines,
{\em J. Approx. Theory} v. 84 (1996), 315--350.

\bibitem{love}
A.E.H.Love,
{\em A Treatise on the Mathematical Theory of Elasticity},
4th ed., New York: Dover, 1927.

\bibitem{manning96}
R.S.Manning, J.H.Maddocks, J.D.Kahn,
A continuum rod model of sequence-dependent DNA structure,
{\em J. Chem. Phys.} v. 105 (1996), 5626--5646.

\bibitem{manning98}
R.S.Manning, K.A.Rogers, J.H.Maddocks,
Isoperimetric conjugate points with application to the stability of DNA minicircles,
{\em Proc. R. Soc. Lond. A}, v. 454 (1998), 3047--3074.

\bibitem{mumford}
D.Mumford,
Elastica and computer vision, In:
{\em Algebraic geometry and its applications},
C.L.Bajaj, Ed., Springer-Verlag, New-York, 1994, 491--506.

\bibitem{myasnich}
O.~Myasnichenko,
Nilpotent $(3,6)$ Sub-Riemannian Problem.
{\em J. Dynam. Control Systems} 8 (2002), No. 4,
573--597.

\bibitem{saalchutz}
L.Saalsch\"utz,
{\em Der belastete Stab}, Leipzig, 1880.

\bibitem{dido_exp}
Yu. L. Sachkov, Exponential mapping in generalized Dido's problem, {\em Mat. Sbornik},  194  	(2003),  9:  63--90 (in Russian).
English translation in: {\em Sbornik: Mathematics}, {\bf 194} (2003).

\bibitem{max1}
Yu. L. Sachkov, Discrete symmetries in the generalized Dido problem (in Russian),
{\em Matem. Sbornik},
{\bf 197}  (2006),  2:  95--116.
English translation in: {\em Sbornik: Mathematics}, {\bf 197} (2006), 2: 235--257.

\bibitem{max2}
Yu. L. Sachkov, The Maxwell set in the generalized Dido problem (in Russian),
{\em Matem. Sbornik},
{\bf 197}  (2006),  4:  123--150.
English translation in: {\em Sbornik: Mathematics}, {\bf 197} (2006), 4: 595--621.

\bibitem{max3}
Yu. L. Sachkov, Complete description of the Maxwell strata in the generalized Dido problem (in Russian),
{\em Matem. Sbornik},
{\bf 197}  (2006),  6:  111--160.
English translation in: {\em Sbornik: Mathematics}, {\bf 197} (2006), 6: 901--950.


\bibitem{elastica_conj}
Yu. L. Sachkov,
Conjugate points in Euler's elastic problem,
{\em in preparation}.

\bibitem{sar_torres}
A.V.Sarychev, D.F.M. Torres,
Lipschitzian regularity of minimizers for optimal control problems with control-affine dynamics,
{\em Applied Mathematics and Optimization},
41: 237--254 (2000).


\bibitem{timoshenko}
S.Timoshenko,
{\em History of Strength of Materials},
McGraw-Hill, New-York, 1953.

\bibitem{truesdell}
C.Truesdell,
The Influence of Elasticity on Analysis: The Classic Heritage,
{\em Bulletin American Math. Society}, 1983, v. 9, No. 3, 293--310.

\bibitem{whit_watson}
E.T.~Whittaker, G.N.~Watson,
{\em 
A Course of Modern Analysis. An introduction to the general theory of infinite processes and of analytic functions; with an account of principal transcendental functions},
Cambridge University Press, Cambridge 1996.

\bibitem{math}
S. Wolfram,
{\em Mathematica: a system for doing mathematics by computer},
Addison-Wesley, Reading,  MA 1991.


\end{thebibliography}
\end{document}